\documentclass[12pt]{amsart}

\usepackage{amsmath,amssymb,amscd,amsthm,graphicx,enumerate}
\usepackage[usenames,dvipsnames]{xcolor}

\usepackage{hyperref}
\hypersetup{breaklinks=true}

\numberwithin{equation}{section}
\theoremstyle{plain}



\newtheorem{theorem}[equation]{Theorem}

\newtheorem{proposition}[equation]{Proposition}
\newtheorem{lemma}[equation]{Lemma}
\newtheorem{corollary}[equation]{Corollary}

\newtheorem*{cap_stability}{Neck Stability}

\theoremstyle{remark}
\newtheorem{remark}[equation]{Remark}

\theoremstyle{definition}
\newtheorem{definition}[equation]{Definition}

\newtheorem{question}[equation]{Question}
\newtheorem*{question*}{Question}

\newtheorem{example}[equation]{Example}

\newcommand{\C}{{\mathcal C}}

\renewcommand{\L}{{\mathcal L}}
\newcommand{\M}{{\mathcal M}}

\newcommand{\n}{\mathcal{N}}

\renewcommand{\P}{{\mathcal P}}

\renewcommand{\r}{\mathfrak{r}}
\newcommand{\R}{\mathbb R}
\newcommand{\calr}{{\mathcal R}}
\renewcommand{\S}{{\mathcal S}}
\renewcommand{\t}{\mathfrak{t}}

\newcommand{\V}{{\mathcal V}}
\newcommand{\W}{{\mathcal W}}

\newcommand{\Z}{\mathbb Z}

\newcommand{\bad}{\operatorname{Bad}}

\newcommand{\const}{\operatorname{const.}}
\newcommand{\cyl}{Cyl}

\newcommand{\diam}{\operatorname{Diam}}
\newcommand{\dist}{\operatorname{dist}}

\newcommand{\dvol}{\operatorname{dvol}}

\newcommand{\hess}{\operatorname{Hess}}

\newcommand{\im}{\operatorname{Im}}
\newcommand{\inj}{\operatorname{inj}}
\newcommand{\Int}{\operatorname{Int}}

\newcommand{\length}{\operatorname{length}}

\definecolor{gray}{gray}{0.7}

\newcommand{\ric}{\operatorname{Ric}}

\newcommand{\vol}{\operatorname{vol}}

\newcommand{\D}{\partial}
\newcommand{\de}{\delta}
\newcommand{\De}{\Delta}

\newcommand{\eps}{\epsilon}
\newcommand{\ga}{\gamma}
\newcommand{\Ga}{\Gamma}

\newcommand{\kapsold}{\operatorname{Sol_\kappa^D}}

\newcommand{\la}{\lambda}

\newcommand{\lra}{\longrightarrow}

\newcommand{\ol}{\overline}

\newcommand{\Om}{\Omega}

\newcommand{\ra}{\rightarrow}
\newcommand{\restr}{\mbox{\Large \(|\)\normalsize}}
\newcommand{\Rm}{\operatorname{Rm}}

\newcommand{\SO}{\operatorname{SO}}

\newcommand{\sphere}{Sphere}
\newcommand{\tr}{\operatorname{tr}}
\renewcommand{\th}{\theta}

\def\XXint#1#2#3{{\setbox0=\hbox{$#1{#2#3}{\int}$}
     \vcenter{\hbox{$#2#3$}}\kern-.5\wd0}}

\begin{document}

\begin{abstract}

We introduce singular Ricci flows, which are Ricci flow spacetimes 
subject to certain asymptotic conditions.  
These provide a solution to the long-standing problem of finding a 
good notion of Ricci flow through singularities, in the three
dimensional case.

We prove that Ricci flow with surgery, starting from a fixed
initial condition, subconverges to a singular Ricci flow as 
the surgery parameter tends to zero.  We establish a number of
geometric and analytical properties of singular Ricci flows.

\end{abstract}

\title{Singular Ricci flows I}
\author{Bruce Kleiner}
\address{Courant Institute of Mathematical Sciences\\
251 Mercer St. \\
New York, NY  10012}
\email{bkleiner@cims.nyu.edu}

\author{John Lott}
\address{Department of Mathematics\\
University of California at Berkeley \\
Berkeley, CA  94720}
\email{lott@berkeley.edu}
\thanks{Research supported by NSF grants DMS-1105656,
DMS-1207654 and DMS-1405899, and a Simons Fellowship}
\date{November 16, 2015}
\maketitle

\tableofcontents

\section{Introduction}

\subsection*{Overview}
It has been a long-standing problem in geometric analysis to find a good
notion of a Ricci flow 
through singularities
\cite{Feldman-Ilmanen-Knopf_shrinking,Perelman_entropy}.
The motivation comes from the fact that 
a Ricci flow with a smooth initial condition can develop singularities without
blowing up everywhere; hence one would like to continue the flow beyond
the singular time.
From the broader perspective 
of the analysis of PDE's, this is just one instance of
the widespread phenomenon of the breakdown of classical solutions, 
a phenomenon that is central to geometric PDE's, and which
is often handled by 
using generalized solutions.
For example, for mean curvature flow of hypersurfaces in 
$\R^n$, there are notions of generalized
solutions 
\cite{Brakke_motion,Chen-Giga-Goto_uniqueness,Evans-Spruck_level_set}
which became the foundation for studying existence, uniqueness,
partial regularity, compactness, and other structural properties of solutions
\cite{Brakke_motion,huisken_monotonicity,evans_spruck_IV,Ilmanen_elliptic_regularization,white_nature,white_size}.
Other geometric PDE's, such as
minimal surfaces \cite{de_giorgi_misura,reifenberg_varying,Federer-Fleming_normal_integral,reifenberg_analyticity,almgren_some,simons_minimal_varieties,federer_dimension_reducing,allard_first_variation,simon_cylindrical}, 
harmonic maps \cite{morrey_multiple_integrals,schoen_uhlenbeck_regularity,schoen_uhlenbeck_minimizing,giaquinta_giusti_quadratic,helein_symmetries,bethuel_stationary_harmonic,simon_rectifiability,lin_blow_up},  and harmonic map heat flows \cite{struwe_higher_dimensions,chen_struwe_heat_flow_harmonic,lin_wang_defect_measure},
have undergone
a similar development. 

Thus far, there has been little progress in implementing a similar program for
Ricci flow.
In the approaches used for the equations above, one first defines 
``rough'' objects ---
e.g. integral currents, sets of finite perimeter,
varifolds, Sobolev mappings --- and then, using an appropriate approximation
scheme, one produces a generalized solution within the class of rough objects,
by appealing to a suitable weak compactness result.  
Certain features, such as the existence of an ambient space, or the
fact that one has a scalar equation,   help enormously.
The absence of these features
in the case of Ricci flow creates
serious technical obstacles to using such an approach with the Ricci 
flow.

In this paper we solve the problem of flowing through singularities in
three-dimensional Ricci flow. We do this 
using a novel approach that is quite different in spirit from 
earlier work, and which may be adaptable to other geometric PDE's.
We introduce generalized solutions that 
we call singular Ricci flows.
They are smooth Ricci flow spacetimes that are possibly incomplete,
but are subject to certain asymptotic
conditions.  We show that singular Ricci flows 
have a number of good properties.
In particular, we prove a compactness result for families of spacetimes, 
and using this we obtain
the following existence theorem:

\begin{theorem}
\label{thm_existence_intro}
For every compact Riemannian $3$-manifold $M$, there is a singular Ricci flow 
with initial condition $M$.
\end{theorem}
\noindent
In addition, we establish a number of structural results.  
These results, and further results that will appear elsewhere,
strongly indicate that singular Ricci flows
provide a natural analytical
framework for 
three-dimensional Ricci
flows with singularities.

The notion of singular Ricci flows  derives partly
from the 
spectacular work of
Hamilton \cite{Hamilton_four_manifolds} 
and Perelman \cite{Perelman_surgery} on Ricci flow with surgery.    
After constructing Ricci flow with surgery, Perelman was naturally 
lead back to the problem of flowing through singularities.
In \cite[Sec. 13.2]{Perelman_entropy},
Perelman wrote, ``It is likely that by 
passing to the limit in this construction one would get a canonically defined
Ricci flow through singularities, but at the moment I don't have a proof of 
that.''  
In proving Theorem \ref{thm_existence_intro} we
partially confirm  Perelman's expectation, by showing that Ricci flow with 
surgery (for a fixed initial condition) subconverges to a singular
Ricci flow as the surgery parameter goes to zero.\footnote{\em{Note added in proof.}  Since this paper was submitted there have been several developments building on the results proven here.  Perelman's convergence conjecture  (\cite[Sec. 13.2]{Perelman_entropy}, \cite[p.1]{Perelman_surgery}) and the uniqueness question (Question~\ref{ques_uniqueness}) were both addressed in \cite{bamler_kleiner_uniqueness_stability}.   These results, together with a number of results from this paper (the existence of singular Ricci flows (Corollary~\ref{cor_existence_smooth_initial_condition}), behavior of volume (Theorem~\ref{thm_main_convergence}(4),Theorem~\ref{vol1}), and finiteness of bad worldlines  (Theorem~\ref{finitenessthm}) were used to prove the Generalized Smale Conjecture concerning the homotopy type of diffeomorphism groups of $3$-manifolds \cite{bamler_kleiner_gsc}.}

Several aspects of the work in this paper may be 
applicable in other settings.
First,   the strategy that we use here
--- defining generalized solutions using smooth 
(possibly incomplete) spacetimes
satisfying  asymptotic geometric bounds --- may be adaptable to other 
geometric PDE's. 
In addition, there are technical ingredients in our work
which may have
parallels in other situations,
such as a compactness theorem valid for possibly incomplete Riemannian manifolds,
and a dynamical analysis of ancient solutions.
A novel feature of the setup is that the
spacetimes are not assumed to be 
simply
concatenations of product spacetime regions. 
To our knowledge this is the first 
appearance of such spacetimes in Ricci flow, although of course
they are widely used 
in general relativity and mean curvature flow.

We mention that there is related
work in the literature on Ricci flow through singularities in certain
cases, such as in the K\"ahler case or under the assumption of
rotational symmetry; see the
end of the introduction for a discussion.

\subsection*{Convergence of Ricci flows with surgery}
Before formulating our first result, we briefly recall 
Perelman's version of
Ricci flow with surgery (which followed earlier work of Hamilton).
 (To make this paper accessible to a wider audience, in 
Appendix \ref{RFsurgery} we collect needed background results
about Ricci flow and Ricci flow with surgery.)

Ricci flow with surgery evolves a Riemannian $3$-manifold by 
alternating between two processes:
flowing by ordinary Ricci flow until the metric  goes
singular, and modifying the resulting limit by surgery, so as to produce a  
compact smooth Riemannian manifold that serves as a new initial condition for 
Ricci flow.  The construction is regulated by  
a global parameter $\epsilon > 0$ as well as 
decreasing
parameter functions $r, \delta,\kappa : [0, \infty) \ra
(0, \infty)$, which play the following roles:
\begin{itemize}
\item   
The scale at which surgery occurs is bounded above in terms of
$\delta$. In particular, 
surgery at time $t$ is performed by cutting along necks whose scale 
tends to zero as $\delta(t)$ goes to zero.
\item The function $r$ defines 
the canonical neighborhood scale: at time $t$, near
any point with  scalar curvature at least $r(t)^{-2}$, 
the flow is (modulo parabolic 
rescaling)  approximated
to within error $\eps$ by either 
a $\kappa$-solution (see Appendix \ref{appkappa})
or  a standard postsurgery model.
\end{itemize}

In Ricci flow with surgery,
the initial conditions are assumed to be {\em normalized}, meaning that
at each point $m$ in the initial time slice,
the eigenvalues of the curvature operator $\Rm(m)$ are bounded by one in
absolute value, and the volume of the unit ball $B(m,1)$ is at least
half the volume of the Euclidean unit ball.  
By rescaling, any
compact Riemannian manifold can be normalized.

Perelman showed that under certain constraints on the parameters, one
can implement Ricci flow with surgery for any
normalized initial condition.  His constraints allow one to make
$\de$ as small as one wants. Hence one can consider the behavior of 
Ricci flow with surgery, for a fixed initial condition, as $\de$ goes 
to zero. 

In order to formulate our convergence theorem, we will use a spacetime framework.
Unlike the case of general relativity, where one has a Lorentzian manifold,
in our setting there is a natural foliation of spacetime by time slices,
which carry Riemannian metrics.  This is formalized in the following
definition.

\begin{definition} \label{def_ricci_flow_spacetime}
A {\em Ricci flow spacetime} is a tuple $(\M,\t,\D_{\t},g)$ where:
\begin{itemize}
\item $\M$ is a smooth manifold-with-boundary.
\item 
$\t$ is the {\em time function} -- a submersion
$\t:\M\ra I$ where $I\subset\R$ is  a time interval; 
we will usually take $I=[0,\infty)$.
\item 
The boundary of
$\M$, if it is nonempty, 
corresponds to the endpoint(s) of the 
time interval: $\D\M=\t^{-1}(\D I)$.
\item $\D_{\t}$ is the {\em time vector field}, which satisfies $\D_{\t}\t\equiv 1$.
\item  $g$ is a smooth inner product on the spatial subbundle
$\ker(d\t)\subset T\M$, and $g$ 
defines a Ricci flow: ${\mathcal L}_{\D_{\t}} g = - 2 \ric(g)$.
\end{itemize}
For $0\leq a< b$, we write $\M_a=\t^{-1}(a)$, $\M_{[a,b]}=\t^{-1}([a,b])$ and 
$\M_{\leq a}=\t^{-1}([0,a])$.   
Henceforth, unless otherwise specified, when we refer to geometric quantities such as 
curvature, we will implicitly be referring to the metric on the time slices.
\end{definition}

\medskip

Note that near any point $m\in\M$, a Ricci flow spacetime 
$(\M,\t,\D_{\t},g)$ reduces to a Ricci flow
in the usual sense, because the time function $\t$ will form 
part of a chart $(x,\t)$ near $m$ 
for which the coordinate vector field $\frac{\D}{\D t}$
coincides with $\D_{\t}$; then one has $\frac{\D g}{\D t}=-2\ric(g)$.  Also, there is a canonical Ricci flow spacetime associated with any Ricci flow with surgery (see Subsection~\ref{RFsurgery}); we will often conflate this Ricci flow spacetime with the Ricci flow with surgery. 

Our first result  partially answers the question of
Perelman alluded to above, by formalizing the notion of convergence 
and obtaining subsequential limits:

\begin{theorem}
\label{thm_main_convergence}
Let $\{\M^j\}_{j=1}^\infty$ be a sequence of
three-dimensional Ricci flows with surgery
(in the sense of Perelman)
where:
\begin{itemize}
\item  The initial conditions 
 $\{\M^j_0\}$ are compact 
normalized Riemannian manifolds that
lie in a compact family
in the smooth topology, and
\item If $\de_j:[0,\infty)\ra(0,\infty)$ denotes the Perelman  
surgery parameter for
$\M^j$ then
$\lim_{j \rightarrow \infty} \delta_j(0) = 0$.
\end{itemize}
Then after passing to a subsequence,
there is a 
Ricci flow spacetime $(\M^\infty,\t_\infty,\D_{t_\infty},g_\infty)$,
and a sequence 
of diffeomorphisms
\begin{equation}
\{\M^j\supset U_j \stackrel{\Phi^j}{\ra} V_j\subset\M^\infty\}
\end{equation}
with the following properties:
\begin{enumerate}
\item $U_j\subset\M^j$ and $V_j\subset \M^\infty$ are open subsets. 
\item 
Let $R_j$ and $R_\infty$
denote the scalar curvature on $\M^j$ and $\M^\infty$, respectively.
Given $\ol{t} < \infty$ and $\ol{R} < \infty$, if $j$ is sufficiently
large then
\begin{align}
U_j &\supset \{ m_j \in \M^j \: : \: \t_j(m_j) \le \bar t,
\: R_j(m_j) \le \ol{R} \}, \\
V_j & \supset 
 \{ m_\infty \in \M^\infty \: : \: \t_\infty(m_\infty) \le \bar t,
\: R_\infty(m_\infty) \le \ol{R} \}. \notag
\end{align}
\item $\Phi^j$ is time preserving, and the sequences
$\{\Phi^j_* \partial_{t_j}\}_{j=1}^\infty$ and
$\{\Phi^j_*g_j\}_{j=1}^\infty$ converge smoothly on compact subsets of 
$\M^\infty$ to $\partial_{t_\infty}$ and $g_\infty$, respectively.
\item $\Phi^j$ is asymptotically volume preserving: Let
$\V_j, \V_\infty:[0,\infty)\ra [0,\infty)$ denote the respective
volume functions
$\V_j(t)= Vol(\M^j_t)$ and $\V_\infty(t)= Vol(\M^\infty_t)$.
Then 
 $\V_\infty:[0,\infty)\ra [0,\infty)$ is continuous
 and $\lim_{j \ra \infty} \V_j = \V_\infty$, with uniform
convergence on compact subsets of $[0,\infty)$.  
\end{enumerate}
Furthermore:
\begin{enumerate}
\renewcommand{\theenumi}{\alph{enumi}}
\item The scalar curvature function 
$R_\infty : \M^\infty_{\leq T}\ra \R$ is 
bounded below and proper
for all $T\geq 0$.
\item $\M^\infty$ satisfies the Hamilton-Ivey pinching condition
of (\ref{hi}).
\item 
$\M^\infty$ is 
$\kappa$-noncollapsed below scale $\epsilon$
and satisfies the $r$-canonical neighborhood assumption,
where $\kappa$ and $r$ are the aforementioned
parameters from Ricci flow with surgery.
\end{enumerate}
\end{theorem}

\medskip

Theorem \ref{thm_main_convergence} may be compared with other convergence results 
such as Hamilton's compactness theorem 
\cite{Hamilton_compactness_property} and its variants  
\cite[Appendix E]{Kleiner-Lott_perelman_notes}, as well as 
analogous results for sequences of Riemannian 
manifolds. 
All of these results require uniform bounds on curvature in  regions
of a given size around a basepoint, which we do not have.
Instead, our approach is based on
the fact that in a three dimensional Ricci flow with surgery, the scalar 
curvature controls the local geometry.
We first prove a general pointed compactness result for sequences of 
(possibly incomplete) Riemannian manifolds whose local geometry
is governed by a control function.  
We then apply this
general compactness result in the case when the Riemannian
manifolds are the spacetimes of Ricci flows with surgery, and the
control functions are constructed from the scalar curvature
functions.  To obtain (1)-(3) of Theorem \ref{thm_main_convergence} 
we have to rule out the possibility that part of the spacetime with
controlled time and scalar curvature escapes to infinity, i.e. is not
seen in the pointed limit.
This is done by means of a new estimate on the 
spacetime geometry of a Ricci flow with surgery; see 
Proposition \ref{curve} below.

Motivated by the conclusion of Theorem \ref{thm_main_convergence}, 
we make the following definition:
\begin{definition} 
\label{def_singular_ricci_flow}
A Ricci flow spacetime
$(\M,\t,\D_{\t},g)$ 
is a {\em singular Ricci flow} if it is $4$-dimensional, 
the initial time
slice $\M_0$ is a compact normalized Riemannian manifold 
and
\begin{enumerate}
\renewcommand{\theenumi}{\alph{enumi}}
\item The scalar curvature function $R:\M_{\leq T}\ra \R$ is 
bounded below and proper
for all $T\geq 0$.
\item $\M$ satisfies the Hamilton-Ivey pinching condition
of (\ref{hi}).
\item 
For a global parameter $\epsilon > 0$ and  
decreasing
functions $\kappa, r:[0,\infty)\ra (0,\infty)$, the spacetime
$\M$ is 
$\kappa$-noncollapsed below scale $\epsilon$ in the sense of
Appendix \ref{kappa} and satisfies the $r$-canonical neighborhood assumption
in the sense of Appendix \ref{canon}.
\end{enumerate}
\end{definition} 

\medskip
Although conditions (b) and (c) 
in Definition \ref{def_singular_ricci_flow} are pointwise conditions
imposed everywhere, 
we will show elsewhere  that
$\M$ is a singular Ricci flow 
if (b) and (c) are only assumed to hold outside of some
compact subset of $\M_{\le T}$, for all $T \ge 0$. Thus (b) and (c)
can be viewed as asymptotic conditions at infinity for a Ricci flow
defined on a noncompact spacetime.
Condition (a) implies that if a spatial slice is noncompact then
the scalar curvature tends to infinity as one approaches an end;
this latter property 
compensates for the possible lack of completeness.

With this definition, Perelman's existence theorem for Ricci flow with
surgery and Theorem \ref{thm_main_convergence} immediately imply :
\begin{corollary}
\label{cor_existence_smooth_initial_condition}
If $(M,g_0)$ is a compact 
normalized
Riemannian $3$-manifold then there exists
a singular Ricci flow having initial condition $(M,g_0)$,
with parameter functions $\kappa$ and $r$ as in
Theorem \ref{thm_main_convergence}.
\end{corollary}
From the PDE 
viewpoint, flow with surgery is a regularization of Ricci flow, while
singular Ricci flows may be considered to be generalized solutions to 
Ricci flow.  In this language, 
Corollary \ref{cor_existence_smooth_initial_condition}
gives the
existence of generalized solutions by means of a regularization procedure. 
One can compare this
with the existence proof for Brakke flows in 
\cite{Brakke_motion,Ilmanen_elliptic_regularization} or 
level set flows in \cite{Chen-Giga-Goto_uniqueness,Evans-Spruck_level_set}.

The existence assertion in Corollary \ref{cor_existence_smooth_initial_condition}
leads to the corresponding uniqueness question:
\begin{question}
\label{ques_uniqueness}
If  two singular Ricci flows have isometric initial conditions, are the
underlying Ricci flow spacetimes the same up to diffeomorphism?
\end{question}
An affirmative answer would confirm Perelman's expectation 
that Ricci flow with surgery should converge to a canonical flow 
through singularities, as it  would imply that if one takes a fixed initial condition
in Theorem \ref{thm_main_convergence}
 then one would have convergence without having
to pass to a subsequence.
Having such a uniqueness result, in conjunction with 
Theorem \ref{thm_main_convergence}, would
closely parallel the 
results of 
\cite{Lauer_convergence_mcf,Head_two_convex,Brendle-Huisken_surgery,Haslhofer-Kleiner_surgery}
that $2$-convex
mean curvature flow with surgery converges to level set flow
when the surgery parameters tend to zero.

\subsection*{The structure of singular Ricci flows}

The asymptotic conditions
in the definition of a singular Ricci flow have a number of implications
which we analyze in this paper.  In addition to clarifying the
structure of limits of Ricci flows with surgery as in Theorem 
\ref{thm_main_convergence}, the results indicate that singular Ricci flows 
are well behaved objects from geometric and analytical points of view.

To analyze the geometry of Ricci flow spacetimes, we use two different Riemannian
metrics.
\begin{definition} \label{spacetimemetric}
Let $(\M,\t,\D_{\t},g)$ be a Ricci flow spacetime.  The {\em spacetime
metric} on $\M$ is the  Riemannian metric
$g_{\M}=\hat g+d\t^2$, where
$\hat g$ is the extension of $g$ to a quadratic form on $T\M$
such that $\D_{\t}\in\ker(\hat g)$.
The {\em quasiparabolic metric} on $\M$ is the Riemannian metric
$g^{qp}_{\M}=(1+R^2)^\frac12\,\hat g+(1+R^2)d\t^2$. 
\end{definition}

For the remainder of the introduction, unless otherwise specified, $(\M,\t,\D_{\t},g)$
will denote a fixed singular Ricci flow.

Conditions (b) and (c) of Definition \ref{def_singular_ricci_flow} imply that 
the scalar curvature controls the local geometry of the singular Ricci flow.  This
has several implications.
\begin{itemize}
\item (High-curvature regions in
singular Ricci flows are topologically standard)
For every $t$, the superlevel set $\{R\geq r^{-2}(t)\}\cap \M_t$ is contained in a 
disjoint union of connected components whose diffeomorphism types
come from a small list of possibilities, with well-controlled local geometry.
In particular, each connected component $C$  of $\M_t$ has finitely many ends, and passing to
the metric completion $\bar C$ adds at most one point for each end
(Proposition \ref{prop_large_r_structure}).
\item
(Bounded geometry)
The spacetime metric
$g_{\M}$ has bounded geometry at the scale defined by the scalar
curvature, while the quasiparabolic metric $g_{\M}^{qp}$
is complete and has bounded
geometry  in the usual sense --- the injectivity radius is bounded 
below and all derivatives of curvature are uniformly bounded
(Lemma \ref{quasip}).
\end{itemize}

The local control on geometry also leads to a compactness property for singular Ricci flows:
\begin{itemize}
\item (Compactness) 
If one has a sequence $\{(\M^j,\t_j,\D_{\t_j},g_j)\}_{j=1}^\infty$
of singular Ricci flows with a fixed 
choice of functions in Definition \ref{def_singular_ricci_flow}, 
and the initial metrics $\{(\M^j_0, g_j(0))\}_{j=1}^\infty$
form a precompact set in the smooth topology,
then a subsequence converges in the sense of
Theorem \ref{thm_main_convergence}
(Proposition \ref{rfcompactness}).
\end{itemize}

The proof of the compactness result is similar to the proof of 
Theorem \ref{thm_main_convergence}. 
We also have global results concerning
the scalar curvature and volume.

\begin{itemize}
\item (Scalar curvature and volume control)
For any $T<\infty$, the scalar curvature is integrable on $\M_{\leq T}$.
The volume function $\V(t)=\vol(\M_t)$ is 
absolutely continuous and has a locally
bounded upper right derivative.
The usual formula holds for volume evolution:
\begin{equation}
\label{eqn_volume_formula}
\V(t_1)-\V(t_0)=-\int_{\M_{[t_0,t_1]}}R\;\dvol_{g_{\M}}
\end{equation}
for all $0\leq t_0\leq t_1 < \infty$
(Corollary \ref{cor_volume_continuous}).
\end{itemize}

Elsewhere we will discuss the structure of the completion of the
spacetime, and will also show:
\begin{itemize}
\item (Refined scalar curvature and volume estimates)
For all $p \in (0,1)$ and all $t$,
the scalar curvature is $L^p$ on $\M_t$.
The volume $\V(t)$ is 
locally  $\alpha$-H\"older in $t$ for some exponent $\alpha \in (0,1)$.
\end{itemize}

To describe the next results, we introduce the following definitions.
\begin{definition}
\label{def_worldline}
A path $\ga:I\ra\M$ is {\em time-preserving} if $\t(\ga(t))=t$ for  all
$t\in I$. 
The {\em worldline} of a point
$m\in\M$ is the maximal time-preserving integral curve $\ga:I\ra \M$ of the 
time vector field $\D_{\t}$, which  passes through $m$.
\end{definition}

If $\ga:I\ra \M$ is a worldline then we may have $\sup I <\infty$.   In this case,
the scalar curvature blows up along $\ga(t)$ as $t\ra \sup I$, and the worldline 
encounters a singularity.  An example would be a shrinking round space form, or a 
neckpinch.  
A worldline may also encounter a singularity going backward
in time.

\begin{definition}
A worldline $\ga:I\ra\M$ is {\em bad} if $\inf I>0$, i.e.  if it is not
defined at $t=0$.
\end{definition}

Among our structural results, perhaps the most striking is the following:
\begin{theorem}
\label{thm_finiteness_of_bad_worldlines}
Suppose that $(\M,\t,\D_{\t},g)$ is a singular Ricci flow and $t \geq 0$.
If $C$ is a connected component of $\M_t$ 
then only finitely many points 
in $C$ have  bad worldlines.  Moreover if $\ga:I\ra\M$ is a bad worldline
then for $t\in I$ sufficiently close to $\inf I$, $\ga(t)$ lies in a 
cap region of $\M_t$.
\end{theorem}
As an illustration of the theorem, consider a
singular Ricci flow that undergoes a 
generic neck pinch at time $t_0$, so that 
the time slice $\M_{t_0}$ has two ends 
($\eps$-horns
in Perelman's language) which are instantly capped off when $t>t_0$.  
In this case the theorem asserts that only finitely many (in this case two)
 worldlines emerge from 
the singularity. See Figure \ref{fig-1}, where the bad wordlines
are indicated by dashed curves.
(The point in the figure where the two dashed curves meet is not
in the spacetime.)

\begin{figure}[h]
\begin{center}
\includegraphics[height=4in]{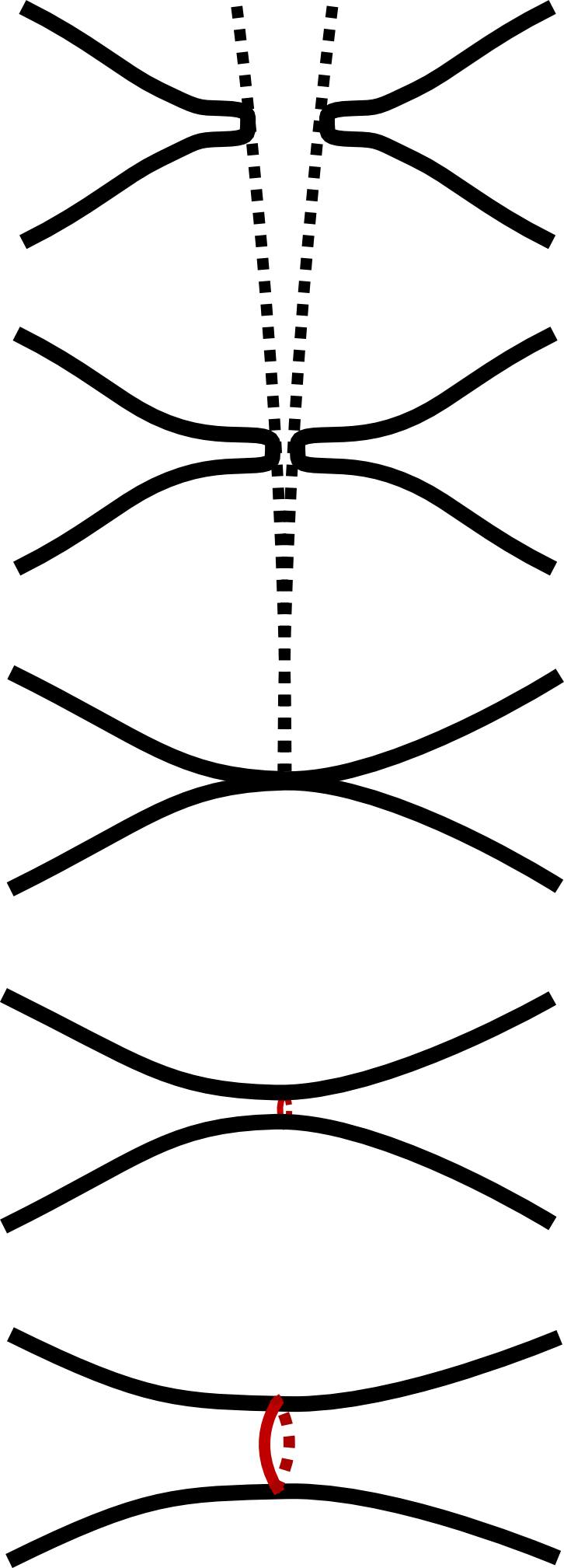}
\caption{\label{fig-1}}
\end{center}
\end{figure}

A key ingredient in the proof of Theorem \ref{thm_finiteness_of_bad_worldlines}
is a new stability property of neck regions in $\kappa$-solutions.
We recall that  $\kappa$-solutions are the class of ancient Ricci flows used
to model the high curvature part of Ricci flows with surgery. 
  We state the stability property
loosely as follows, and refer the reader to Section \ref{stab} 
for more details:
\begin{cap_stability}
Let $\M$ be a noncompact $\kappa$-solution other than the shrinking
round cylinder.
If  $\ga:I\ra \M$ is a worldline and $\M_{t_1}$
 is  sufficiently necklike at $\ga(t_1)$, 
then as
$t\ra-\infty$, the time slice $\M_t$ looks more and more necklike at $\ga(t)$.
Here the notion of necklike is scale invariant.
\end{cap_stability}
An easy but illustrative
case is the  Bryant soliton, in which worldlines other than the tip itself
move away from the tip 
(in the scale invariant sense) as
one goes backward in time.  

As mentioned above,
 Theorem \ref{thm_finiteness_of_bad_worldlines} is used in the proof of
(\ref{eqn_volume_formula}) and  the properties of volume.  

We mention some connectedness properties of singular Ricci flows.

\begin{itemize}
\item 
(Paths back to $\M_0$ avoiding high-curvature regions) Any point 
$m \in \M$ can be joined
to the initial time slice
$\M_0$ by a time-preserving curve $\ga:[0,\t(m)]\ra \M$, along
which 
$\max\{R(\ga(t))\: : \: t\in [0,\t(m)]\}$ is bounded in terms of
$R(m)$ and $\t(m)$. (Proposition \ref{curve2})
\item 
(Backward stability of components)  If $\ga_0,\ga_1:[t_0, t_1]\ra\M$
are time-preserving curves such that $\ga_0(t_1)$ and $\ga_1(t_1)$ lie in the 
same connected component of $\M_{t_1}$, 
then $\ga_0(t)$ and $\ga_1(t)$ lie in the 
same component of $\M_t$ for all $t\in [t_0,t_1]$. (Proposition 
\ref{sameconnected})
\end{itemize}

\subsection*{Related  work}
We now mention some other work that falls into the broad setting of
Ricci flow with singular structure.

A number of authors
have considered Ricci flow with low regularity initial conditions, studying
existence and/or uniqueness, and instantaneous improvement of regularity
\cite{Chen-Tian-Zhang_weak,Giesen-Topping_pseudolocality,Giesen-Topping_incomplete_surfaces,Giesen-Topping_unbounded_curvature,Koch-Lamm_rough_data,Schnuerer-Schulze-Simon_stability,Simon_deformation,Simon_almost_nonnegatively_curved,Simon_noncollapsed,Topping_pseudolocality,Topping_uniqueness_nonuniqueness}.
Ricci flow with persistent singularities has been considered in the case of
orbifold Ricci flow, and Ricci flow with conical singularities 
\cite{Chen-Tang-Zhu_positive_isotropic,Chen-Wang_bessel,Chen-Wang-long-time,Chow_entropy,Chow-Wu_negative,Hamilton_three_orbifolds,Kleiner-Lott_orbifollds,Liu-Zhang_conical,Lu_compactness,Mazzeo-Rubinstein-Sesum_conic,Phong-Song-Sturm-Wang_marked_points,Wang_smooth_approximation,Wu_two_orbifolds,Yin_conical,Yin_conical_2}

Passing to flows through singularities,
Feldman-Ilmanen-Knopf noted that in the noncompact K\"ahler setting, there
are some natural
examples of  flows through singularities which consist of a
shrinking gradient soliton
that has a conical limit at time zero, which  transmutes into an
expanding gradient soliton
\cite{Feldman-Ilmanen-Knopf_shrinking}.
Closer in spirit to this paper, Angenent-Caputo-Knopf
constructed a rotationally invariant Ricci 
flow through singularities starting with a metric on $S^{n+1}$
\cite{Angenent-Caputo-Knopf_minimally}.
They showed that the rotationally invariant neckpinches from 
\cite{Angenent-Knopf_precise},  
which have a singular limit as $t$ approaches zero from below, may be  continued
as a smooth Ricci flow on two copies of
$S^{n+1}$ for $t>0$.  The paper
\cite{Angenent-Caputo-Knopf_minimally}
also showed that the forward evolution has a unique asymptotic profile near
the singular point
in spacetime.

There has been much progress on flowing through singularities in the K\"ahler
setting.� For a flow on a
projective variety with log terminal singularities, Song-Tian 
\cite{Song-Tian_singularities} showed
that the flow can be
continued through the divisorial contractions and flips of the minimal model
program. We refer to \cite{Song-Tian_singularities} for the
precise statements.
The paper \cite{Eyssidieux-Guedj-Zeriahi_weak_solutions} has related results, but  uses a
viscosity solution approach instead of the regularization scheme in \cite{Song-Tian_singularities}.
The fact 
that the K\"ahler-Ricci flow --- like the K\"ahler-Einstein equation ---
 can be reduced to a scalar equation,
to which comparison principles may be applied, 
is an important  simplifying feature of K\"ahler-Ricci flow that is
not available in the non-K\"ahler case.  

Recently, Haslhofer-Naber
gave alternate characterizations of Ricci flow using stochastic analysis
\cite{haslhofer_naber};
they also announced an extension to
evolving metric measure spaces.

\subsection*{Concluding remarks}

The work in this paper is related to the question  of whether there is a good notion
of a generalized solution to the  Ricci flow
equation.  The results here show that singular Ricci flows give
an answer in the $3$-dimensional case, 
and in the 
$4$-dimensional case under the assumption of 
nonnegative isotropic curvature.  

By using 
different quantities to control the geometry, it may be possible
to work with  flows satisfying weaker curvature
conditions.  
The general higher dimensional case, however, 
is mysterious, and finding a good  notion 
of generalized solution is an intriguing question.  For mean curvature flows
with arbitrary smooth initial conditions, the situation  
is somewhat similar.

As mentioned before, some of the technical methods in this paper,
such as a compactness result for possibly incomplete Riemannian manifolds,
and a dynamical analysis of ancient solutions, may be useful for other
problems in geometric analysis.

\subsection*{Organization of the paper}

The remainder  of the paper is broken into two parts.  Part I, which is 
composed of Sections \ref{sec_compactness_locally_controlled}-\ref{sec_main_convergence},
is primarily concerned with the proof of the convergence result,
Theorem \ref{thm_main_convergence}.
Part II, which is composed of 
Sections \ref{sec_singular_ricci_flows_basic_properties}-\ref{sec_finiteness},
deals with results on singular Ricci flows. 
In addition, there are three appendices.

We now describe the contents section by section.

 Section \ref{sec_compactness_locally_controlled} gives a general pointed 
compactness result for Riemannian manifolds and spacetimes, whose geometry
are 
locally controlled as a function of some auxiliary function.  
Section \ref{surgerysec} develops some properties of Ricci flow with surgery; it is 
aimed at showing that the Ricci flow spacetime associated with a Ricci flow with surgery
has locally controlled geometry in the sense of Section
\ref{sec_compactness_locally_controlled}.   
Section \ref{sec_main_convergence} applies the two preceding sections to give the
proof of Theorem \ref{thm_main_convergence}.

Section \ref{sec_singular_ricci_flows_basic_properties} establishes some foundational
results about singular Ricci flows, concerning scalar
curvature and volume, as well as some results
involving the structure of the high curvature region.
Section \ref{stab} proves that neck regions in $\kappa$-solutions
have a stability property when going backward in time.
Section \ref{sec_finiteness} proves
Theorem \ref{thm_finiteness_of_bad_worldlines}, concerning
the finiteness of the number of bad worldlines, 
and gives several applications.

Appendix \ref{app_background_material} collects a variety of background material about Ricci flows and Ricci flows with surgery;
the reader may wish to quickly peruse this before proceeding to the body of the paper.  
Appendix \ref{technical} extends Proposition 
\ref{prop_cylinder_stability_1_noncompact}
to general $\kappa$-solutions.  Appendix \ref{pic}
extends the results of the paper 
to four-dimensional
Ricci flow with nonnegative isotropic curvature.

\subsection*{Notation and terminology}
We refer the reader to  Appendix \ref{sec_notation} for notation 
and terminology.
All manifolds 
that arise
will be taken to be orientable.

We thank the referee for a careful reading and helpful comments.

\section*{Part I}

\section{Compactness for spaces of locally controlled geometries}
\label{sec_compactness_locally_controlled}

In this section we prove a compactness result, Theorem \ref{theorem3.8first}, 
for sequences of controlled Riemannian manifolds.  In Subsection
\ref{spacetimecompactness} we extend the theorem to a 
compactness result for spacetimes,
meaning Riemannian manifolds equipped with time functions and
time vector fields.

\subsection{Compactness of the space of locally 
controlled Riemannian manifolds}

We will need a sequential compactness result for sequences of
Riemannian manifolds which may be incomplete, but 
whose local geometry (injectivity radius and
all derivatives of curvature) is bounded by a function of an auxiliary function
$\psi:M\ra [0,\infty)$.  A standard case of this is 
sequential compactness for sequences
$\{(M_j,g_j,\star_j)\}$ of Riemannian
$r$-balls, assuming that
the distance function $d_j(\star_j,\cdot):M_j\ra [0,r)$ is proper,
and that the geometry is bounded in terms of $d_j(\star_j,\cdot)$.  

Fix a smooth decreasing  function $\r:[0,\infty)\ra (0,1]$ and
smooth increasing functions $C_k:[0,\infty)\ra [1,\infty)$ for all $k\geq 0$.

\begin{definition} \label{definition3.1first}
Suppose that $(M, g)$ is a Riemannian manifold 
equipped with a
function $\psi:M\ra [0,\infty)$.  

Given 
$A\in [0,\infty]$,
a tensor field $\xi$ is 
{\em $(\psi,A)$-controlled} if for all $m \in \psi^{-1}([0, A))$
and $k \ge 0$, we have
\begin{equation}
\|\nabla^k\xi(m)\|\leq C_k(\psi(m)).
\end{equation}
If in addition $\psi$ is smooth then we say that
the tuple 
$(M,g,\psi)$ is {\em
$(\psi,A)$-controlled} if 
\begin{enumerate}
\item  The injectivity radius of $M$ at $m$ is at least
$\r(\psi(m))$ for all $m\in \psi^{-1}([0,A))$, and
\item The tensor field $\Rm$, and $\psi$ itself, are 
$(\psi,A)$-controlled.
\end{enumerate}
\end{definition}

Note that if $A = \infty$ then $\psi^{-1}([0, A))$ is all of $M$,
so there are quantitative bounds on the geometry at each point of $M$.
Note also that the value
of a control function may not reflect the actual bounds on the geometry, in the 
sense that the geometry may be more regular near $m$ than the value of $\psi(m)$
suggests.   This creates flexibility in choosing a control function,
which is useful in applications below.

\begin{example}
Suppose that $(M, g, \star)$ is a complete pointed Riemannian manifold. 
Put $\psi(m) = d(\star, m)$. Then $\Rm$ is $(\psi, A)$-controlled
if and only if for all $r \in (0, A)$ and $k \ge 0$, we have
$\| \nabla^k \Rm \| \le C_k(r)$ on $B(\star, r)$.
\end{example}

\begin{example}
Suppose that $\psi$ has constant value $c > 0$.  If $A \le c$ then
$(M, g, \psi)$ is vacuously $(\psi, A)$-controlled. If $A > c$ then
$(M, g, \psi)$ is $(\psi, A)$-controlled if and only if for all $m \in M$,
we have $\inj(m) \ge \r(c)$ and $\| \nabla^k \Rm(m) \| \le C_k(c)$.
\end{example}

There are compactness results in Riemannian geometry saying that
one can extract a subsequential limit from a sequence of complete pointed
Riemannian manifolds having uniform local geometry.
This last condition means that for each $r > 0$, 
one has quantitative uniform bounds on the geometry of
the $r$-ball around the basepoint; c.f.
\cite[Theorem 2.3]{Hamilton_compactness_property}. In such a case, one can
think of the distance from the basepoint as a control function.
We will give a compactness theorem for Riemannian manifolds
(possibly incomplete) equipped with more general control functions.

For notation, 
if $\Phi : U \rightarrow V$ is a diffeomorphism then we will write
$\Phi_*$ for both the pushforward 
action of $\Phi$ on contravariant tensor fields on $U$,
and the pullback action of $\Phi^{-1}$ on covariant tensor fields on $U$.

\begin{theorem}[Compactness for controlled manifolds] \label{theorem3.8first}
Let 
\begin{equation}
\{(M_j,g_j,\star_j,\psi_j)\}_{j=1}^\infty
\end{equation}
be a sequence of pointed tuples
which are $(\psi_j,A_j)$-controlled, where 
$\lim_{j\ra\infty} A_j = \infty$ 
and $\sup_j\psi_j(\star_j)
<\infty$.  Then after passing to a subsequence,
there is a pointed 
$(\psi_\infty,\infty)$-controlled
tuple
\begin{equation}
(M_\infty,g_\infty,\star_\infty,\psi_\infty)
\end{equation}
and a sequence
of diffeomorphisms 
$\{M_j\supset {U}_j \stackrel{\Phi^j}{\lra} {V}_j\subset 
M_\infty\}_{j=1}^\infty$
such that
\begin{enumerate}
\item Given
$A, r < \infty$, for all sufficiently large $j$
the open set $U_j$ contains the ball $B(\star_j,r)$ in the Riemannian manifold
$(\psi_j^{-1}([0,A)),g_j)$  and likewise
$V_j$ contains the ball $B(\star_\infty,r)$ in the Riemannian manifold
$(\psi_\infty^{-1}([0,A)),g_\infty)$.
\item 
Given $\eps >0$ and $k\geq 0$, for all sufficiently large $j$ we have 
\begin{equation}
\|\Phi^j_* g_j- g_\infty\|_{C^k(V_j)}<\eps
\end{equation}
and 
\begin{equation}
\|\Phi^j_* \psi_j-\psi_\infty\|_{C^k(V_j)}<\eps.
\end{equation}
\item
 $M_\infty$ is connected and, in particular,
every  $x\in M_\infty$ belongs to $V_j$ for 
$j$ large.
\end{enumerate}
\end{theorem}

We will give the proof of Theorem \ref{theorem3.8first}
in Subsection \ref{pf3.8first}.
We first describe some general results about controlled Riemannian
manifolds.

\subsection{Some properties of controlled Riemannian manifolds}

One approach to proving Theorem \ref{theorem3.8first} would be to
imitate what one does when one has curvature and injectivity radius
bounds on 
$r$-balls, replacing the control function based on distance to the
basepoint by the control function $\psi$. While this could be done,
it would be somewhat involved.  Instead, we will perform a
conformal change on the Riemannian manifolds in order to put
ourselves in a situation where the geometry is indeed controlled by
the distance from the basepoint. We then take a subsequential limit 
of the conformally changed metrics, and at the end perform another
conformal change to get a subsequential limit of the original sequence.

Let $(M,g, \psi)$ be $(\psi, A)$-controlled.
Put 
\begin{equation} \label{3.5}
\widetilde g= \left( \frac{C_1}{\r} \circ\psi \right)^{2} g.
\end{equation}
We will only consider   $\widetilde g$  on the subset
 $\psi^{-1}([0,A))$, where it is smooth.
The next two lemmas are about $g$-balls
and $\widetilde{g}$-balls.

\begin{lemma} \label{lemma3.21}
For each finite $a \in (0, A]$,
each $\star \in \psi^{-1}([0, a))$
and each $r < \infty$, 
there is some $R = R(a,r) < \infty$ so that
the ball  
$B_g(\star,r)$ in the Riemannian manifold
$(\psi^{-1}([0,a)),g)$ 
is contained in the ball $B_{{\widetilde g}}(\star, R)$
in the Riemannian manifold $(\psi^{-1}([0,a)),{\widetilde g})$.
\end{lemma}
\begin{proof}
On $\psi^{-1}([0,a))$ we have 
$\widetilde g\leq \left( \frac{C_1(a)}{\r(a)} \right)^{2}g$.  
Therefore any path in $\psi^{-1}([0,a))$
with $g$-length at most $r$ has $\widetilde{g}$-length at most 
$\frac{C_1(a)}{\r(a)} r$,
so we may take
$R= \frac{C_1(a)}{\r(a)} r$.
\end{proof}

Let $\star$ be a basepoint in $\psi^{-1}([0, A))$.

\begin{lemma} \label{ballinside}
For all $R \in  (0, A - \psi(\star))$, the ball
$B_{\widetilde{g}}(\star, R)$ in the Riemannian manifold
$(\psi^{-1}([0, A)),\widetilde{g})$ 
is contained in $\psi^{-1}([0, \psi(\star) + R))$.
\end{lemma}
\begin{proof}
Given $m \in B_{\widetilde{g}}(\star, R)$, let 
$\gamma : [0,L] \ra \psi^{-1}([0, A))$ 
be a smooth
path from $\star$ to $m$ with unit $\widetilde{g}$-speed 
and $\widetilde{g}$-length $L \in (0, R)$. 
Then
\begin{align}
\psi(m) - \psi(\star) = & \int_0^L \frac{d}{dt} \psi(\gamma(t)) \: dt
\le \int_0^L \left| d \psi \right|_{\widetilde{g}}(\gamma(t)) 
\: dt \\
= & \int_0^L \frac{\r}{C_1}(\psi(\gamma(t))) \cdot
\left| d \psi \right|_{g} (\gamma(t))
 \: dt \notag \\
\le & \int_0^L \r(\psi(\gamma(t)))
 \: dt \le \int_0^L 1 \: dt 
= L. \notag 
\end{align}
The lemma follows.
\end{proof}

We now look at completeness properties of $\widetilde{g}$-balls.

\begin{lemma} \label{compactclosure}
For all $R \in (0, A - \psi(\star))$, the ball
$B_{\widetilde{g}}(\star, R)$ in the Riemannian manifold
$\left( \psi^{-1}([0,A)),
{\widetilde{g}} \right)$ has compact closure in
$\psi^{-1}([0,A))$.
\end{lemma}
\begin{proof} 
Choose $R^\prime \in (R, A - \psi(\star))$. 
From \cite[Chapter 1, Theorem 2.4]{Ballmann_lectures}, it suffices to
show that any $\widetilde{g}$-unit speed geodesic $\gamma : [0, L) \ra
\psi^{-1}([0, A))$ with $\gamma(0) = \star$, having 
$\widetilde{g}$-length $L \in (0,R^\prime)$, can be
extended to $[0, L]$. From Lemma \ref{ballinside},
$\gamma([0, L)) \subset \psi^{-1}([0, \psi(\star) + R^\prime))$.
Hence the $g$-injectivity radius along $\gamma([0, L))$ is bounded
below by $\r(\psi(\star) + R^\prime))$. 
For large $K$,
the points $\{\gamma(L - \frac{1}{k})\}_{k = K}^\infty$
form a Cauchy sequence in
  $(\psi^{-1}([0, A)), d_{\widetilde{g}})$.
As $d_{\widetilde{g}}$ and $d_g$ are biLipschitz on
$\psi^{-1}([0, \psi(\star) + R^\prime))$, the sequence is also Cauchy in
$\left( \psi^{-1}([0, A)), d_g \right)$.
From the uniform positive lower bound on the $g$-injectivity radius
at $\gamma(L - \frac{1}{k})$,
there is a limit in $\psi^{-1}([0, A))$. The lemma follows.
\end{proof}

\begin{corollary} \label{lemma3.4}
If $A = \infty$ then 
$(M,\widetilde g)$
is complete.
\end{corollary}
\begin{proof}
This follows from Lemma \ref{compactclosure} and
\cite[Chapter 1, Theorem 2.4]{Ballmann_lectures}.
\end{proof}

Finally, we give bounds on the geometry of $\widetilde{g}$-balls.

\begin{lemma} \label{morebounds}
Given $\r$, $\{C_k\}_{k=1}^\infty$ and $S < \infty$, there exist 
\begin{enumerate}
\item A smooth decreasing function $\widetilde{\r} : [0, \infty) 
\rightarrow (0,1]$, and
\item Smooth increasing functions $\widetilde{C}_k : 
[0, \infty) \ra [1, \infty)$,
$k \ge 0$,
\end{enumerate}
with the following properties.  Suppose that 
$(M, g, \psi)$ is $(\psi, A)$-controlled and
$\psi(\star) \le S$. Then
\begin{enumerate}
\renewcommand{\theenumi}{\alph{enumi}}
\item $\Rm_{\widetilde{g}}$ and $\psi$ are 
$\left( d_{\widetilde{g}}(\star, \cdot), A-S \right)$-controlled on
the Riemannian manifold $\left( \psi^{-1}([0, A)), \widetilde{g} \right)$
(in terms of the functions $\{\widetilde{C}_k\}_{k=1}^\infty$).
\item 
If $R < A - S - 1$ then
$
\inj_{\widetilde{g}} \ge \widetilde{\r}(R)
$
pointwise on $B_{\widetilde{g}}(\star, R)$.
\end{enumerate}
\end{lemma}
\begin{proof}
Conclusion (a) (along with the concomitant functions
$\{\widetilde{C}_k\}_{k=1}^\infty$)
follows from Lemma \ref{ballinside}, the assumption that
$\Rm_{g}$ and $\psi$ are 
$\left( \psi, A \right)$-controlled,
and the formula for the Riemannian curvature of
a conformally changed metric. 

To prove conclusion (b), suppose that
$R < A - S - 1$ and $m \in B_{\widetilde{g}}(\star, R)$.
Since $d_{\widetilde{g}}$ and $d_g$ are biLipschitz on the ball
$B_{\widetilde{g}}(\star, R+1)$ in the Riemannian manifold
$\left( \psi^{-1}([0, R+1+S)), \widetilde{g} \right)$, we can find
$\epsilon = \epsilon(R, S, \{C_k\}) > 0$ so that
$B_g(m, \epsilon) \subset B_{\widetilde{g}}(m, 1)
\subset B_{\widetilde{g}}(\star, R+1)$.
Since we have a $g$-curvature bound on $B_{\widetilde{g}}(\star, R+1)$,
and a lower $g$-injectivity radius bound at $m$, we obtain a
lower volume bound $\vol(B_g(m, \epsilon), g) \ge v_0 = 
v_0(R, S, \r, \{C_k\}) > 0$. Since $g$ and $\widetilde{g}$ are
relatively bounded on $B_g(m, \epsilon)$, this gives a lower volume bound
$\vol(B_{\widetilde{g}}(m, 1), \widetilde{g}) \ge v_1 = 
v_1(R, S, \r, \{C_k\}) > 0$. Using the curvature bound of part (a) and
\cite[Theorem 4.7]{Cheeger-Gromov-Taylor_finite},
we obtain a lower bound
$\inj_{\widetilde{g}}(m) \ge i_0 = i_0(R, S, \r, \{C_k\}) > 0$.
This proves the lemma.
\end{proof}

\subsection{Proof of Theorem \ref{theorem3.8first}} \label{pf3.8first}

Put $\widetilde g_j= \left( \frac{C_1}{\r} \circ\psi_j \right)^{2} g_j$.   
Consider the tuple
$(M_j,{\widetilde g}_j,\star_j,\psi_j)$. Recall that $\lim_{j \ra \infty} 
A_j = \infty$.

For the moment, we replace
the index $j$ by the index $l$.
Using Lemma \ref{compactclosure}, Lemma \ref{morebounds}
and a standard compactness theorem
\cite[Theorem 2.3]{Hamilton_compactness_property},
after passing to a subsequence we can find
a complete pointed Riemannian manifold
$(M_\infty, g_\infty, \star_\infty)$, domains $\widetilde{U}_l \subset M_l$ and
$\widetilde{V}_l \subset M_\infty$, and diffeomorphisms 
$\widetilde{\Phi}^l : \widetilde{U}_l \rightarrow \widetilde{V}_l$ so that 
\begin{enumerate}
\renewcommand{\theenumi}{\alph{enumi}}
\item $\star_l \in \widetilde{U}_l$.
\item $\star_\infty \in \widetilde{V}_l$.
\item ${\widetilde V}_l$ has compact closure.
\item For any compact set $K \subset M_\infty$, we have
$K \subset \widetilde{V}_l$ for all sufficiently large $l$.
\item Given $\eps >0$ and $k\geq 0$, we have 
\begin{equation} \label{3.20}
\|\widetilde{\Phi}^l_* \widetilde{g}_l- \widetilde{g}_\infty\|_{C^k({\widetilde V}_l, \widetilde{g}_\infty)}
<\eps
\end{equation}
for all sufficiently large $l$.
\end{enumerate}

Using Lemma \ref{morebounds} again, 
after passing to a further subsequence if necessary,
we can assume that there is a smooth
$\psi_\infty$ on $M_\infty$ so that
for all $\eps >0$ and $k\geq 0$, we have 
\begin{equation}
\|\widetilde{\Phi}^l_* \psi_l- 
\psi_\infty\|_{C^k({\widetilde V}_l,{\widetilde g}_\infty)}<\eps
\end{equation}
for all sufficiently large $l$.
Put $g_\infty = 
\left( \frac{C_1}{\r} \circ \psi_\infty \right)^{-2}
\widetilde{g}_\infty$.

We claim that
if the sequence $\{l_j\}_{j=1}^\infty$
increases rapidly enough then
the conclusions of the theorem can be made to hold with
$V_j = B_{\widetilde{g}_\infty}(\star_\infty, j) \subset M^\infty$, 
$U_j = \left( \widetilde{\Phi}^{l_j} \right)^{-1} (V_j) \subset M^{l_j}$
and $\Phi^j = \widetilde{\Phi}^{l_j} \Big|_{U_j}$.
To see this, we note that
\begin{itemize}
\item 
Given $A,r < \infty$, Lemma \ref{lemma3.21} implies that
the ball $B(\star_\infty, r)$ in the Riemannian manifold
$(\psi_\infty^{-1}([0, A)), g_\infty)$
will be contained in $V_j$ for all sufficiently large $j$.
\item 
If the sequence $\{l_j\}_{j=1}^\infty$
increases rapidly enough then for large $j$, the map 
$\Phi^j$ is arbitrarily close to an isometry
and $(\Phi^j)_* \psi_j $ is arbitrarily close to $\psi_\infty$ on
$V_j$. 
Hence given $A,r < \infty$, the ball
$B(\star_j, r)$ in the Riemannian manifold $(\psi_j^{-1}([0, A)), g_j)$
will be contained in $U_j$ for all sufficiently large $j$, so
conclusion (1) of the theorem holds.
\item The metrics $g_\infty$ and ${\widetilde g}_\infty$ are
biLipschitz on $V_j$.
Then if the sequence $\{l_j\}_{j=1}^\infty$
increases rapidly enough, conclusion (2) of the theorem can be made to hold.
\item 
Conclusion (3) of the theorem follows from the definition of $V_j$.
\end{itemize}

This proves Theorem \ref{theorem3.8first}. \qed

\subsection{Compactness of the space of locally controlled
spacetimes} \label{spacetimecompactness}

We now apply Theorem \ref{theorem3.8first} to prove a compactness result
for spacetimes.

\begin{definition} \label{definition3.1}
A {\em spacetime} is a Riemannian manifold 
$\left( \M, g_{\M} \right)$ equipped with a 
submersion
$\t: \M\ra \R$ and a smooth
vector field $\D_{\t}$ such that $d\t(\D_{\t})\equiv 1$.

Given 
$A\in [0,\infty]$ and $\psi \in C^\infty(\M)$, 
we say that the tuple 
$\left( \M,g_{\M},\t,\D_{\t},\psi \right)$ is {\em
$(\psi,A)$-controlled} if 
$\left( \M,g_{\M},\psi \right)$ is $(\psi, A)$-controlled in
the sense of Definition \ref{definition3.1first} and, in addition,
the tensor fields $\t$ and $\D_{\t}$ are
$(\psi,A)$-controlled.
\end{definition}

\begin{theorem}[Compactness for controlled spacetimes] \label{theorem3.8}
Let 
\begin{equation}
\left\{ \left( \M^j,g_{\M^j},\t_j,(\D_{\t})_j,\star_j,\psi_j \right)
\right\}_{j=1}^\infty
\end{equation}
be a sequence of pointed tuples
which are $(\psi_j,A_j)$-controlled, where 
$\lim_{j\ra\infty} A_j = \infty$ 
and $\sup_j\psi_j(\star_j)
<\infty$.  Then after passing to a subsequence,
there is a pointed 
$(\psi_\infty,\infty)$-controlled
tuple
\begin{equation}
\left( \M^\infty,g_{\M^\infty}, \t_\infty,(\D_{\t})_\infty,\star_\infty,
\psi_\infty \right)
\end{equation}
and a sequence
of diffeomorphisms
\begin{equation}
\{ \M^j\supset U_j \stackrel{\Phi^j}{\lra} V_j\subset \M^\infty\}_{j=1}^\infty
\end{equation}
of open sets
such that
\begin{enumerate}
\item Given $A < \infty$ and $r < \infty$, for all sufficiently large $j$,
the open set $U_j$ contains the ball $B(\star_j,r)$ in the Riemannian manifold
$\left( \psi_j^{-1}([0,A)),g_{\M^j} \right)$  and likewise
$V_j$ contains the ball $B(\star_\infty,r)\subset 
\left( \psi_\infty^{-1}([0,A)),g_{\M^\infty} \right)$.
\item $\Phi^j$ exactly preserves the time functions:
\begin{equation}
\t_\infty\circ\Phi^j=\t_j\;\text{for all}\;j\,.
\end{equation}
\item $\Phi^j$ asymptotically preserves the tensor fields 
$g_{\M^j}$, $(\D_{\t})_j$, and $\psi_j$:
if  $\xi_j\in \left\{ g_{\M^j},\left(\D_{\t} \right)_j,\psi_j \right\}$
and $\xi_\infty$ is the corresponding element of 
$\left\{ g_{\M^\infty},\left(\D_{\t}\right)_\infty,\psi_\infty \right\}$,
then for all $\eps >0$ and $k\geq 0$, we have 
\begin{equation}
\|\Phi^j_*\xi_j-\xi_\infty\|_{C^k(V_j)}<\eps
\end{equation}
for all sufficiently large $j$.
\item 
$\M^\infty$ is connected and, in particular, every $x\in \M^\infty$ belongs to 
$V_j$ for large $j$.
\end{enumerate}
\end{theorem}
\begin{proof}
Put $\widetilde{g}_{\M^j} = \left( \frac{C_1}{\r} \circ \psi_j
\right)^{2} g_{\M^j}$.
Consider the pointed spacetime
$\left( \M^j,{\widetilde g}_{\M^j},\t_j,(\D_{\t})_j,\star_j,\psi_j \right)$.

\begin{lemma} \label{lemma3.18again}
For every $A<\infty$, the  tensor fields
 $\Rm_{\widetilde g_j}$, $\t_j$, $(\D_{t})_j$, and $\psi_j$  are all 
 $\left( d_{{\widetilde g}_{\M^j}}(\star_j,\cdot),A \right)$-controlled 
for all large $j$,
in the sense of Lemma \ref{morebounds}.
\end{lemma}
\begin{proof}
This follows as in the proof of Lemma \ref{morebounds}.
\end{proof}

We follow the proof of Theorem \ref{theorem3.8first} up to the
construction of $\widetilde{\Phi}^l : \widetilde{U}_l \rightarrow
\widetilde{V}_l$ and $\psi_\infty$. Using Lemma \ref{lemma3.18again},
after passing to a further subsequence if necessary,
we can assume that there are smooth $\t_\infty$ and 
$(\partial_{\t})_\infty$
on $\M^\infty$ so that if
$\xi_l \in \{\t_l, (\partial_{\t})_l, \psi_l\}$ then
for all $\eps >0$ and $k\geq 0$, we have 
\begin{equation}
\|\widetilde{\Phi}^l_* \xi_l- 
\xi_\infty\|_{C^k({\widetilde V}_l,{\widetilde g}_{\M^\infty})}<\eps
\end{equation}
for all sufficiently large $l$.
Again, $\t_\infty : \M^\infty \rightarrow \R$ is a submersion and
$(\partial_{\t})_\infty \t_\infty = 1$. 
As before, put
$g_{\M^\infty} =
\left( \frac{C_1}{\r} \circ \psi_\infty \right)^{-2}
    {\widetilde g}_{\M^\infty}$.

Let $\{\phi_s\}$ be the
flow generated by $(\partial_{\t})_\infty$; this exists for at
least a small time interval if the starting point is in a given
compact subset of $M_\infty$. 
Put $V^\prime_j = B_{\widetilde{g}_\infty}(\star_\infty, j)$.
Then there
is some $\Delta_j > 0$ so that 
$\{\phi_s\}$ exists on $V^\prime_j$
for $|s| < \Delta_j$. 
Given $l_j \gg 0$, put
$U_j = (\widetilde{\Phi}^{l_j})^{-1}(V^\prime_j)$. 
Assuming that $l_j$ is
large enough, we can define  
$\Phi^j : U_j \rightarrow \M^\infty$ by
\begin{equation}
\Phi^j(m) = \phi_{\t_j(m) - \t_\infty(\widetilde{\Phi}^{l_j}(m))} 
(\widetilde{\Phi}^{l_j}(m)). 
\end{equation}
By construction, $\t_\infty(\Phi^j(m)) = \t_j(m)$, and so
conclusion (2) of the theorem holds. 
If $l_j$ is large then $\Phi^j$ will be a diffeomorphism
to its image.
Putting
$V_j = \Phi^j(U_j)$, if $l_j$ is large enough then
$V_j$ can be made arbitrarily close to $V^\prime_j$.
It follows that conclusions (1), (3) and (4) of the theorem hold.
\end{proof}

\begin{remark}
In Hamilton's compactness theorem
\cite{Hamilton_compactness_property}, the comparison map $\Phi^j$ preserves both
the time function and the time vector field. In Theorem \ref{theorem3.8}
the comparison map $\Phi^j$ preserves the time function, but 
the time vector field is only preserved asymptotically. This is
good enough for our purposes.
\end{remark}

\section{Properties of Ricci flows with surgery} 
\label{surgerysec}

In this section we prove several estimates for Ricci flows with surgery.  These
will be used  in the proof of Theorem \ref{thm_main_convergence}, to show that 
the sequence of Ricci flow spacetimes has the local control required for
the application of the spacetime
compactness theorem, Theorem \ref{theorem3.8}.   

The arguments in this section require familiarity with some basic 
properties of Ricci flow with surgery. For the reader's convenience,
we have collected these
properties in Appendix A.8. 
The reader may wish to review this material before proceeding.

Let $\M$ be a Ricci flow with surgery in the sense of Perelman
\cite[Section 68]{Kleiner-Lott_perelman_notes}.
As recalled in Appendix \ref{RFsurgery},
there are parameters --- or more precisely 
positive decreasing
parameter functions --- 
associated with $\M$:
\begin{itemize}
\item The canonical neighborhood scale function $r(t)>0$.
We can assume that $r(0) < \frac{1}{10}$.
\item The noncollapsing function $\kappa(t)>0$.
\item The parameter $\de(t)>0$.
This has a dual role: it is the quality of the
surgery neck, and it enforces a scale buffer between the canonical neighborhood
scale $r$, the 
intermediate
scale $\rho$ and the surgery scale $h$.
\item The 
intermediate
scale $\rho(t)=\de(t)r(t)$,
which defines the threshold for discarding entire connected components at the
singular time.
\item The surgery scale $h(t) < \de^2(t)r(t)$.
\item 
The global parameter $\epsilon > 0$. This enters in the
definition of a canonical neighborhood. For the Ricci flow with surgery
to exist, a necessary condition is that $\epsilon$ be small enough.
\end{itemize}

In this section, canonical neighborhoods are those defined for
Ricci flows with surgery, as in
\cite[Definition 69.1]{Kleiner-Lott_perelman_notes}.
The next lemma gives a sufficient condition for parabolic
neighborhoods to be unscathed.

\begin{lemma} \label{lemma2.1}
Let $\M$ be a Ricci flow with surgery, with normalized initial
condition.
Given 
$T > \frac{1}{100}$,  there are numbers $\mu = \mu(T) \in (0,1)$, 
$\sigma = \sigma(T) \in (0,1)$, 
$i_0 = i_0(T) > 0$ and 
$A_k = A_k(T) < \infty$, $k \ge 0$, with the following
property.  If 
$t \in \left( \frac{1}{100}, T \right]$
and 
$|R(x,t)| < \mu \rho(0)^{-2} - r(T)^{-2}$,
put $Q = |R(x,t)| + r(t)^{-2}$. Then
\begin{enumerate}
\item The forward parabolic ball
$P_+(x,t, \sigma Q^{- \: \frac12})$ and the backward
parabolic ball $P_-(x,t, \sigma Q^{- \: \frac12})$
are unscathed.
\item 
$|\Rm| \le A_0 Q$,
$\inj \ge i_0 Q^{- \: \frac12}$ and
$|\nabla^k \Rm| \le A_k Q^{1 + \frac{k}{2}}$
on the union
$P_+(x,t, \sigma Q^{- \: \frac12}) \cup
P_-(x,t, \sigma Q^{- \: \frac12})$
of the forward and
backward parabolic balls.
\end{enumerate}
\end{lemma}
\begin{proof}
By \cite[Lemma 70.1]{Kleiner-Lott_perelman_notes}, we have
$R(x^\prime, t^\prime) \le 8Q$ for all $(x^\prime,t^\prime) \in P_-(x, t, 
\eta^{-1} Q^{- \: \frac12})$.
where $\eta < \infty$ is a universal constant.
The same argument works for 
$P_+(x, t, \eta^{-1} Q^{- \: \frac12})$.
Since $R$ is proper on time slices (c.f. 
\cite[Lemma 67.9]{Kleiner-Lott_perelman_notes}), it follows that
$B(x,t, \eta^{-1} Q^{- \: \frac12})$ 
has compact closure in its time slice.

If $\mu \le \frac18$ then
\begin{align}
8Q & = 8 \left( |R(x,t)| + r(t)^{-2} \right) \le
8 \left( |R(x,t)| + r(T)^{-2} \right) \\
& \le
8 \mu \rho(0)^{-2} \le \rho(t^\prime)^{-2} \notag
\end{align}
for all
$t^\prime \in [t - \: \eta^{-2} Q^{-1}, t
+ \: \eta^{-2} Q^{-1}]$. 
Hence if $\sigma < \eta^{-1}$ then 
the forward and backward parabolic balls
$P_{\pm}(x,t, \sigma Q^{- \: \frac12})$ do not intersect the
regions that are affected by the surgery procedure;
see Remark \ref{surgeryremark}.

To show that the balls
$P_{\pm}(x,t, \sigma Q^{- \: \frac12})$ are unscathed, for an appropriate
value of $\sigma$, it remains to
show that $P_-(x,t, \sigma Q^{- \: \frac12})$ does not intersect the
time-zero slice. We have
$t - \sigma^2 Q^{-1} \ge t - \sigma^2 r(t)^2$.
Since $t > \frac{1}{100}$ and
$r(t) < r(0) < \frac{1}{10}$, if $\sigma < \frac12$ then
$t - \sigma^2 Q^{-1} \ge \frac{1}{200}$ and
the balls $P_{\pm}(x,t, \sigma Q^{- \: \frac12})$ are unscathed.

The Hamilton-Ivey estimate of (\ref{hi}) gives an explicit upper bound
$|\Rm| \le A_0 Q$ on
$P_\pm(x, t, \sigma Q^{- \: \frac12})$.
Using the distance distortion estimates for Ricci flow
\cite[Section 27]{Kleiner-Lott_perelman_notes}, there is a universal constant
$\alpha > 0$ so that whenever 
$(x^\prime, t^\prime)
\in P_{\pm}(x, t, \frac12 \sigma Q^{- \: \frac12})$, 
we have 
$P_-(x^\prime, t^\prime, 
\alpha Q^{- \: \frac12})
\subset
P_+(x, t, \sigma Q^{- \: \frac12})
\cup
P_-(x, t, \sigma Q^{- \: \frac12})$.
Then
Shi's local derivative estimates 
\cite[Appendix D]{Kleiner-Lott_perelman_notes} give estimates
$|\nabla^k \Rm| \le A_k Q^{1 + \frac{k}{2}}$ on 
$P_{\pm}(x, t, \frac12 \sigma Q^{- \: \frac12})$.

Since $t \le T$, we have
$\kappa(t) \ge \kappa(T)$. 
The $\kappa$-noncollapsing statement
gives an explicit lower bound 
$\inj \ge i_0 Q^{- \frac12}$ on a slightly
smaller parabolic ball, which after reducing $\sigma$,
we can take to be of the form
$P_{\pm}(x,t, \sigma Q^{- \: \frac12})$. 
\end{proof}

If $\M$ is a Ricci flow solution and
$\ga:[a, b] \ra \M$ is a time-preserving spacetime curve then
we define 
$\length_{g_{\M}}(\gamma)$
using the spacetime metric $g_{\M} = dt^2 + g(t)$.
The next lemma says that given a point $(x_0, t_0)$ in a $\kappa$-solution
(in the sense of Appendix \ref{appkappa}),
it has a large backward parabolic neighborhood so that any point
$(x_1, t_1)$ in the parabolic neighborhood can be connected to
$(x_0, t_0)$ by a time-preserving curve whose length is controlled
by $R(x_0,t_0)$, and along 
which the scalar curvature is controlled by $R(x_0,t_0)$.

\begin{lemma} \label{boxlemma}
Given $\kappa>0$, there exist $A = A(\kappa) < \infty$ and 
$C = C(\kappa) <\infty$
with the following property.
If $\M$ is a $\kappa$-solution and $(x_0,t_0) \in \M$ then
\begin{enumerate}
\item There is some
$(x_1,t_1)\in P_-(x_0,t_0, \frac12 AR(x_0, t_0)^{- \: \frac12})$
with $R(x_1,t_1)\leq \frac13
R(x_0, t_0)$.
\item 
The scalar curvature on
$P_-(x_0,t_0,2AR(x_0, t_0)^{- \: \frac12})$ 
is at most $\frac12 C R(x_0, t_0)$,
 and
\item 
Given 
$(x_1, t_1) \in P_-(x_0,t_0,\frac12 AR(x_0, t_0)^{- \: \frac12})$,
there is a time-preserving curve $\gamma : [t_1, t_0] \rightarrow 
P_- \left( x_0, t_0, \frac34 A R(x_0, t_0)^{- \frac12} \right)$
from $(x_1, t_1)$ to $(x_0, t_0)$
with 
\begin{equation}
\length_{g_{\M}}(\gamma) \le
\frac12 C \left(R(x_0, t_0)^{- \: \frac12} +  R(x_0, t_0)^{- \: 1} \right).
\end{equation}
\end{enumerate}
\end{lemma}
\begin{proof}
To prove (1), suppose, by way of contradiction,
that for each $j \in \Z^+$ there is a $\kappa$-solution $\M^j$
and some $(x^j_0, t^j_0) \in \M^j$ so that
$R > \frac13 R(x^j_0, t^j_0)$ on $P_-(x^j_0,t^j_0, \frac12 j)$.
By compactness of the space of normalized pointed
$\kappa$-solutions (see Appendix \ref{appkappa}),
after normalizing so that $R(x^j_0, t^j_0) = 1$ and passing to a subsequence,
there is a limiting $\kappa$-solution 
$\M'$, defined for $t \le 0$, with  $R\geq \frac13$ everywhere. 
By the weak maximum principle for complete noncompact manifolds
 \cite[Theorem A.3]{Kleiner-Lott_perelman_notes} and the evolution equation
for scalar curvature,
there is a universal constant $\Delta > 0$ (whose exact value
isn't important) so that
if $R \ge \frac13$ on a time-$t$ slice then there is a singularity
by time $t + \Delta$. Applying this with $t = - 2 \Delta$
gives a contradiction.

Part (2) of the lemma,
for some value of $C$,
follows from the 
compactness of the space of normalized pointed $\kappa$-solutions.

To prove (3),
the curve which starts as a worldline  from
$(x_1, t_1)$ to $(x_1, t_0)$, and then moves as a minimal geodesic from
$(x_1, t_0)$ to $(x_0, t_0)$ in the time-$t_0$ slice, has $g_{\M}$-length
at most 
$\frac12 A R(x_0, t_0)^{- \: \frac12} +  \frac14 A^2 R(x_0, t_0)^{- \: 1}$.
By a slight perturbation to make it time-preserving,
we can construct $\gamma$ 
in $P_- \left( x_0, t_0, \frac34 A R(x_0, t_0)^{- \frac12} \right)$
with length at most
$\frac12 (A + 1) R(x_0, t_0)^{- \: \frac12} +  
\frac14 (A+1)^2 R(x_0, t_0)^{- \: 1}$. After redefining $C$,
this proves the lemma.
\end{proof}

The next proposition extends the preceding lemma from $\kappa$-solutions
to points in Ricci flows with surgery.
Recall that $\epsilon$ is the global parameter in the definition of
Ricci flow with surgery.

\begin{proposition} \label{curve}
There is an $\epsilon_0 > 0$ so that if $\epsilon < \epsilon_0$ then
the following holds.
Given $T < \infty$, suppose that $\rho(0) \le \frac{r(T)}{\sqrt{C}}$,
where $C$ is the constant from Lemma \ref{boxlemma}.
Then for any $R_0 < \frac{1}{C} \rho(0)^{-2}$,
there are $L = L(R_0, T) < \infty$ and
$R_1 = R_1(R_0, T) < \infty$ with the following 
property. 
Let $\M$ be a Ricci flow with surgery having normalized initial conditions.
Given $(x_0, t_0) \in \M$ with $t_0 \leq T$,
suppose that $R(x_0,t_0)\leq R_0$.
Then there is a time preserving curve $\ga:[0,t_0]\ra \M$ with
$\gamma(t_0) = (x_0, t_0)$ and $\length_{g_{\M}}(\ga)\leq L$ so that
$R(\ga(t))\leq R_1$ for all $t\in [0,t_0]$.
\end{proposition}
\begin{proof}
We begin by noting that
we can find $\epsilon_0 > 0$ so that if $\epsilon < \epsilon_0$
then for any $(x,t) \in \M$ with $t \le T$ which 
is in a canonical neighborhood,
Lemma \ref{boxlemma}(2) implies that
$R \le C R(x, t)$ on 
$P_-(x, t, AR(x, t)^{- \: \frac12})$.
If in addition $R(x, t) \le 
\frac{1}{C} \rho(0)^{-2}$ then for any 
$(x^\prime,t^\prime) \in P_-(x, t, AR(x, t)^{- \: \frac12})$,
we have 
\begin{equation}
R(x^\prime,t^\prime) \le  \rho(0)^{-2} \le
\rho(t^\prime)^{-2}.
\end{equation} 
Hence the parabolic neighborhood does not intersect the surgery region.

To prove the proposition, we start with $(x_0, t_0)$, and
inductively form a sequence of points 
$(x_i,t_i)$
and the curve $\gamma$ as follows, starting with $i = 1$.\\
{\bf Step 1 :}
If $R(x_{i-1},t_{i-1}) \ge r(t_{i-1})^{-2}$ then go to Substep A.
If $R(x_{i-1},t_{i-1}) < r(t_{i-1})^{-2}$ then go to Substep B. \\
{\bf Substep A :}
Since $R(x_{i-1},t_{i-1}) \ge r(t_{i-1})^{-2}$, the point
$(x_{i-1},t_{i-1})$ is in a canonical neighborhood.  
As will be explained, the backward parabolic ball 
$P_-(x_{i-1}, t_{i-1}, AR(x_{i-1}, t_{i-1})^{- \: \frac12})$
does not intersect the surgery region.
Applying Lemma \ref{boxlemma}
and taking $\epsilon_0$ small,
we can find
$(x_i, t_i) \in P_-(x_{i-1}, t_{i-1}, 
AR(x_{i-1}, t_{i-1})^{- \: \frac12})$
with
$R(x_i, t_i) \le \frac12 R(x_{i-1},t_{i-1})$, and a time-preserving
curve 
\begin{equation}
\gamma : [t_i, t_{i-1}] \rightarrow P_-(x_{i-1}, t_{i-1}, A
R(x_{i-1}, t_{i-1})^{- \: \frac12})
\end{equation}
from $(x_i, t_i)$ to $(x_{i-1}, t_{i-1})$
whose length is at most 
\begin{equation}
C \left( R(x_{i-1}, t_{i-1})^{- \: \frac12} +
R(x_{i-1}, t_{i-1})^{- \: 1} \right),
\end{equation}
along
which the scalar curvature is at most $C R(x_{i-1}, t_{i-1})$.
If $t_i > 0$ then go to Step 2.
If $t_i = 0$ then the process is terminated. 
\\
{\bf Substep B :} Since $R(x_{i-1},t_{i-1}) < r(t_{i-1})^{-2}$, put
$x_i = x_{i-1}$ and 
$t_i = \inf \{t : R(x_{i-1}, s) \le r(s)^{-2} \mbox{ for all } 
s \in [t, t_{i-1}]\}$.
Define $\gamma : [t_{i}, t_{i-1}] \rightarrow \M$ to be the worldline
$\gamma(s) = (x_{i-1}, s)$. 

If $t_i > 0$ then go to Step 2. (Note that 
$R(x_i, t_i) = r(t_i)^{-2}$.)
If $t_i =0$ then the process is terminated. \\
{\bf Step 2 :} Increase $i$ by one and go to 
Step 1. \\

To recapitulate the iterative process, if $R_0$ is large then
there may initially
be a sequence of Substep A's.  Since the curvature decreases by a
factor of at least two for each of these, the number of these
initial substeps is bounded
above by $\log_2 \left( \frac{R_0}{r(0)^{-2}} \right)$. Thereafter,
there is some $(x_{i-1}, t_{i-1})$ so that $R(x_{i-1}, t_{i-1}) <
r(t_{i-1})^{-2}$. We then go backward in time along a segment of
a   worldline
until we either hit a point $(x_i, t_i)$ with $R(x_i, t_i) = 
r(t_i)^{-2}$, or we hit time zero. If we hit $(x_i, t_i)$ then
we go back to Substep A, which produces a point $(x_{i+1}, t_{i+1})$
with at most half as much scalar curvature, etc.

We now check the claim in Substep A that the 
backward parabolic ball 
$P_-(x_{i-1}, t_{i-1}, 
AR(x_{i-1}, t_{i-1})^{- \: \frac12})$  
does not intersect the surgery region.
In the initial sequence of Substep A's, we 
always have
$R(x_{i-1}, t_{i-1}) \le R_0 < \frac{1}{C} \rho(0)^{-2}$.  
If we return to Substep A sometime after the initial sequence, we have
\begin{equation}
R(x_{i-1}, t_{i-1}) = r(t_{i-1})^{-2} \le 
r(T)^{-2} \le \frac{1}{C} \rho(0)^{-2}.
\end{equation}
Either way, from the first paragraph of the proof, we conclude that
$P_-(x_{i-1}, t_{i-1}, 
AR(x_{i-1}, t_{i-1})^{- \: \frac12})^{- \: 1})$
does not intersect the surgery region.

We claim that the iterative process terminates.  If not, the decreasing
sequence $(t_i)$ approaches some $t_\infty > 0$. For an infinite
number of $i$, we must have $R(x_i,t_i) = r(t_i)^{-2}$, which
converges to $r(t_\infty)^{-2}$. Consider a large $i$ with
$R(x_i,t_i) = r(t_i)^{-2}$. The result of Substep A is a point
$(x_{i+1}, t_{i+1})$ with $R(x_{i+1}, t_{i+1}) \le \frac12
R(x_i,t_i) = \frac12 r(t_i)^{-2} \sim \frac12 r(t_\infty)^{-2}$. 
If $i$ is large then this is
less than $r(t_{i+1})^{-2} \sim r(t_\infty)^{-2}$.
Hence one goes to Substep B to find $(x_{i+1}, t_{i+2})$ with
$R(x_{i+1}, t_{i+2}) = r(t_{i+2})^{-2}$.  However,
there is a double-sided
bound on $\frac{\partial R}{\partial t}(x_{i+1}, t)$
for $t \in [t_{i+2}, t_{i+1}]$, coming from
the curvature bound on a backward parabolic ball in
\cite[Lemma 70.1]{Kleiner-Lott_perelman_notes} and Shi's local estimates
\cite[Appendix D]{Kleiner-Lott_perelman_notes}.
This bound implies that the amount of backward time required to go from a 
point with scalar curvature
$R(x_{i+1}, t_{i+1}) \le \frac12 r(t_i)^{-2} \sim \frac12 r(t_\infty)^{-2}$
to a point with scalar curvature 
$R(x_{i+1}, t_{i+2}) = r(t_{i+2})^{-2} \sim r(t_\infty)^{-2}$
satisfies
\begin{equation}
t_{i+1} - t_{i+2} \ge \const r(t_\infty)^2.
\end{equation}
This contradicts the fact that $\lim_{i \ra \infty} t_i = t_\infty$.
 
We note that the preceding argument can be made effective.  This gives
a upper bound $N$ on the number of points $(x_i, t_i)$, of the form
$N = N(R_0, T)$. 
We now estimate the length of $\gamma$.  The contribution to the
length from segments arising from Substep B is at most $T$.
The contribution from segments arising from Substep A is bounded above by
$NC (r(0) + r(0)^2)$. 

It remains to estimate the scalar curvature along $\gamma$.
Along a portion of $\gamma$ arising from Substep A, the
scalar curvature is bounded above by $CR(x_{i-1}, t_{i-1}) \le
C \max(R_0, r(T)^{-2})$. Along a portion of $\gamma$ arising from
Substep B, the scalar curvature is bounded above by $r(T)^{-2}$.
Thus we can take $R_1 = (C+1)(R_0 + r(T)^{-2})$.
This proves the proposition.
\end{proof}

Finally, we give an estimate on the volume of the high-curvature
region in a Ricci flow with surgery.  This estimate will be used
to prove the volume convergence statement in Theorem 
\ref{thm_main_convergence}.

\begin{proposition} \label{volprop}
Given $T < \infty$,
there are functions $\epsilon^T_{1,2} : [0, \infty) \rightarrow
[0, \infty)$, with $\lim_{\overline{R} \ra \infty} 
\epsilon^T_1(\overline{R}) = 0$
and $\lim_{\delta \ra 0} \epsilon^T_2(\delta) = 0$, having the
following property. 
Let $\M$ be a Ricci flow with surgery, with normalized initial condition.
Let $\V(0)$ denote its initial volume.
Given $\overline{R} > r(T)^{-2}$, if
$t$ is not a surgery time then let $\V^{\ge \overline{R}}(t)$ be the
volume of the corresponding superlevel set of $R$ in $\M_t$. 
If $t$ is a surgery time, let $\V^{\ge \overline{R}}(t)$ be the
volume of the corresponding
superlevel set of $R$ in the postsurgery manifold $\M_t^+$.
Then if $t \in [0, T]$, we have
\begin{equation}
\V^{\ge \overline{R}}(t) \le \left( \epsilon^T_1(\overline{R}) + 
\epsilon^T_2(\delta(0)) \right) \V(0).
\end{equation}
\end{proposition}
\begin{proof}
Suppose first that $\M$ is a smooth Ricci flow.
Given $t \in [0, T]$,
let $i_t : \M_0 \rightarrow \M_t$ be the identity map.
For $x \in \M_0$,
put 
\begin{equation}
J_t(x) = \frac{i_t^* \dvol_{g(t)}}{\dvol_{g(0)}} (x).
\end{equation}
Let $\gamma_x : [0,t] \ra \M$ be the worldline of $x$.  
From the Ricci flow equation,
\begin{equation}
J_t(x) = e^{- \int_0^t R(\gamma_x(s)) \: ds}.
\end{equation}

Suppose that $m \in \M_t$ satisfies
$R(m) \ge \overline{R}$. 
Then $R(m) > r(t)^{-2}$.
Let $x \in \M_0$ be the point where the
worldline of $m$ hits $\M_0$.
From (\ref{gradest}) we have
\begin{equation} \label{etaeqn}
  \frac{dR(\gamma_x(s))}{ds}
\le \eta R(\gamma_x(s))^2
\end{equation}
as long as $R(\gamma_x(s)) \ge r(s)^{-2}$. Let $t_1$ be the smallest
number so that $R(\gamma_x(s)) \ge r(s)^{-2}$ for all $s \in [t_1, t]$.
Since $r(0) < \frac{1}{10}$, the normalized
initial conditions imply that 
$t_1 > 0$. From (\ref{etaeqn}), if $s \in [t_1, t]$ then
\begin{equation} \label{changeofR}
R(\gamma_x(s)) \ge \frac{1}{R(m)^{-1} + \eta (t-s)}.
\end{equation}
In particular,
\begin{equation}
r(t_1)^{-2} = R(t_1) \ge \frac{1}{R(m)^{-1} + \eta (t-t_1)}.
\end{equation}

From (\ref{changeofR}),
\begin{equation}
\int_{t_1}^t R(\gamma_x(s)) \: ds \ge - \frac{1}{\eta} \log
\frac{R(m)^{-1}}{R(m)^{-1} + \eta(t - t_1)}.
\end{equation}
Hence
\begin{align}
e^{- \int_{t_1}^t R(\gamma_x(s)) \: ds}
 \le &
\left( 
\frac{R(m)^{-1}}{R(m)^{-1} + \eta(t - t_1)}
\right)^{\frac{1}{\eta}} \\
\le & \left( R(m) r(t_1)^2 \right)^{- \: \frac{1}{\eta}} 
\le \left( R(m) r(T)^2 \right)^{- \: \frac{1}{\eta}}. \notag
\end{align}

On the other hand, for all $s \ge 0$, equation (\ref{scalarsurg})
gives
\begin{equation}
R(\gamma_x(s)) \: \ge \: - \: \frac{3}{1+2s},
\end{equation}
so
\begin{equation}
e^{- \int_0^{t_1} R(\gamma_x(s)) \: ds} \le (1+2t_1)^{\frac32}.
\end{equation}
Thus 
\begin{equation}
J_t(x) \le (1+2T)^{\frac32}
\left( R(m) r(T)^2 \right)^{- \: \frac{1}{\eta}}.
\end{equation}

Integrating over such $x \in \M_0$, 
arising from worldlines emanating from
$\{m \in \M_t \: : \:
R(m) \ge \overline{R}\}$,
we conclude that
\begin{equation}
\V^{\ge \overline{R}}(t) \le 
(1+2T)^{\frac32}
\left( \overline{R} \:  r(T)^2 \right)^{- \: \frac{1}{\eta}} \V(0),
\end{equation}
so the conclusion of the proposition holds in this case with
\begin{equation} \label{eps1}
\epsilon^T_1(\overline{R}) = 
(1+2T)^{\frac32}
\left( \overline{R} \: r(T)^2 \right)^{- \: \frac{1}{\eta}}.
\end{equation}

Now suppose that $\M$ has surgeries.
For simplicity of notation, we assume that $t$ is not a surgery time;
otherwise we replace $\M_t$ by $\M_t^+$.
We can first apply the preceding
argument to the subset of $\{m \in \M_t \: : \:
R(m) \ge \overline{R}\}$ consisting of points whose worldline
goes back to $\M_0$. The conclusion is that the volume of
this subset is bounded above by $\epsilon^T_1(\overline{R}) \V(0)$,
where $\epsilon^T_1(\overline{R})$ is the same as in
(\ref{eps1}). Now consider the subset of
$\{m \in \M_t \: : \:
R(m) \ge \overline{R}\}$ consisting of points whose worldline
does not go back to $\M_0$. We can cover such points by the
forward images of surgery caps (or rather the subsets thereof
which go forward to time $t$), for surgeries that occur at
times $t_\alpha \le t$. Let $\V_{t_\alpha}^{cap}$ be the total volume 
of the surgery caps for surgeries that occur at time $t_\alpha$.
Let $\V_{t_\alpha}^{remove}$ be the total volume that is removed
at time $t_\alpha$ by the surgery process.  From the nature of
the surgery process \cite[Section 72]{Kleiner-Lott_perelman_notes}, there
is an increasing
function $\delta^\prime : (0, \infty) \ra (0, \infty)$, with
$\lim_{\delta \ra 0} \delta^\prime(\delta) = 0$, so that
\begin{equation}
\frac{\V_{t_\alpha}^{cap}}{\V_{t_\alpha}^{remove}} 
\le \delta^\prime(\delta(0)).
\end{equation}
This is essentially because the surgery procedure removes a long capped tube,
whose length (relative to $h(t)$) is large if $\delta(t)$ is small,
and replaces it by a hemispherical cap. 

On the other hand, using (\ref{volumesurg}),
\begin{equation}
\sum_{t_\alpha \le t}
\V_{t_\alpha}^{remove} \le (1 + 2T)^{\frac32} \V(0),
\end{equation}
since surgeries up to time $t$ cannot remove more volume than
was initially present or generated by the Ricci flow.
The time-$t$ volume coming from forward worldlines of surgery caps
is at most $(1 + 2T)^{\frac32} \sum_{t_\alpha \le t}
\V_{t_\alpha}^{cap}$.
Hence the proposition is true if we take
\begin{equation}
\epsilon^T_2(\delta) = (1 + 2T)^3 \delta^\prime(\delta).
\end{equation}
\end{proof}

\section{The main convergence result}
\label{sec_main_convergence}

In this section we prove Theorem \ref{thm_convergence_flow_with_surgery}
except for the statement about the continuity of $\V_\infty$, which will
be proved in Corollary \ref{cor_volume_continuous}.

The convergence assertion in
Theorem \ref{thm_convergence_flow_with_surgery}
involves a sequence $\{\M^j\}$ of  Ricci flows with 
surgery, where the functions $r$ and 
$\kappa$ are fixed, but $\de_j\ra 0$; hence
$\rho_j$ and $h_j$ also go to zero. 
We will conflate these Ricci flows with surgery with  
their associated Ricci flow spacetimes; see Appendix \ref{RFsurgery}.

\begin{theorem} 
\label{thm_convergence_flow_with_surgery}
Let $\{\M^j\}_{j=1}^\infty$ be a sequence of Ricci flows with surgery
with normalized initial conditions
such that:
\begin{itemize}
\item The time-zero slices $\{\M^j_0\}$ are compact 
manifolds that
lie in a compact family
in the smooth topology. 
\item $\lim_{j \rightarrow \infty} \delta_j(0) = 0$.
\end{itemize}
Then after passing to a subsequence,
there is a 
singular Ricci flow  $(\M^\infty,\t_\infty,\D_{\t_\infty},g_\infty)$, 
and a sequence 
of diffeomorphisms 
\begin{equation}
\{\M^j\supset U_j \stackrel{\Phi^j}{\lra} V_j\subset\M^\infty\}\,,
\end{equation}
so that
\begin{enumerate}
\item $U_j\subset\M^j$ and $V_j\subset \M^\infty$ are open subsets. 
\item Given $\ol{t} > 0$ and $\ol{R} < \infty$, we have
\begin{equation} \label{eqn_u_j_contains}
U_j \supset\{m \in \M^j\mid \t_j(m)\leq \ol{t},\; 
R_j(m)\leq \ol{R}\}
\end{equation}
and
\begin{equation} \label{eqn_v_j_contains}
V_j \supset
\{m \in \M^\infty\mid \t_\infty(m )\leq \ol{t},\; 
R_\infty(m)\leq \ol{R}\}
\end{equation}
for all sufficiently large $j$.
\item $\Phi^j$ is time preserving, and the sequences
$\{\Phi^j_* \partial_{\t_j}\}_{j=1}^\infty$ and
$\{\Phi^j_*g_j\}_{j=1}^\infty$ 
converge smoothly on compact subsets of $\M^\infty$
to $\partial_{t_\infty}$ and $g_\infty$, respectively.  
(Note that by
(\ref{eqn_v_j_contains}) any compact set $K\subset \M^\infty$ will
lie in the interior of $V_j$, for all sufficiently large $j$.)
\item 
For every $t \ge 0$, we have
\begin{equation} \label{Rest}
\inf_{\M^\infty_t}R \: \geq \: - \: \frac{3}{1+2t}
\end{equation}
and
\begin{equation}
\max(\V_j(t),\V_\infty(t) )\le 
V_0 \left(1 + 2 t \right)^{\frac{3}{2}},
\end{equation}
where $\V_j(t)=\vol(\M^j_t)$, $\V_\infty(t)=\vol(\M^\infty_t)$,
and $V_0=\sup_j\V_j(0)<\infty$.
\item
$\Phi^j$ is asymptotically volume preserving : if 
$\V_j, \V_\infty : [0, \infty) \ra [0, \infty)$ denote the respective
volume functions $\V_j(t) = \vol(\M^j_t)$ and $\V_\infty(t) = 
\vol(\M^\infty_t)$, then $\lim_{j \ra \infty} \V_j = \V_\infty$
uniformly on compact subsets of $[0, \infty)$. 

\end{enumerate}
\end{theorem}
\begin{proof}
Because of the normalized initial conditions, the Ricci flow
solution $g_j$ is smooth on the time interval $\left[0, \frac{1}{100}
\right]$.
As a technical device, 
we first extend $g_j$ backward in time to a 
family of metrics $g_j(t)$ which is smooth for
$t \in \left( -\infty, \frac{1}{100} \right]$.
To do this, 
as in \cite[Pf. of Theorem 1.2.6]{Hormander_linear_pde},
there is an explicit smooth extension $h_j$ of
$g_j(t) - g_j(0)$ 
to the time interval $t \in \left( - \infty, \frac{1}{100} \right]$, 
with values in smooth covariant $2$-tensor fields.
The extension on $(- \infty, 0]$ depends on the time-derivatives of
$g_j(t) - g_j(0)$ at $t = 0$ which, in turn, can be expressed in
terms of $g_j(0)$ by means of repeated differentiation of the Ricci flow
equation. 
Let $\phi : [0, \infty) \rightarrow [0,1]$ be a fixed nonincreasing
smooth function with $\phi \Big|_{[0,1]} = 1$ and 
 $\phi \Big|_{[2, \infty)} = 0$.
Given $\epsilon > 0$, 
for $t < 0$ put
\begin{equation}
g_j(t) = g_j(0) + \phi \left(- \: \frac{t}{\epsilon} \right)
h_j(t).
\end{equation}
Using the precompactness of the space of initial conditions,
we can choose $\epsilon$ small enough so that $g_j(t)$ is a
Riemannian metric for all $j$ and all $t \le 0$.
Then
\begin{enumerate}
\item $g_j(t)$ is smooth in $t \in \left( - \infty, \frac{1}{100}
\right]$, 
\item $g_j(t)$ is constant in $t$ for $t \le -2 \epsilon$, and
\item For $t \le 0$, $g_j(t)$  has
uniformly bounded curvature and curvature derivatives, independent of $j$.
\end{enumerate}

Let $g_{\M^j}$ be the spacetime Riemannian
metric  on $\M^j$ (see Definition \ref{spacetimemetric}). 
Choose a basepoint $\star_j \in \M^j_0$.

After passing to a subsequence, we can assume that for all $j$,
\begin{equation}
\rho_j(0) \le \sqrt{\frac{\mu}{j + r(j)^{-2}}},
\end{equation}
where 
$\mu = \mu(j)$
is the parameter of Lemma \ref{lemma2.1}. Put 
$\psi_j(x,t) = R_j(x,t)^2 + t^2$,
where $R_j$ is the  scalar
curvature function of the Ricci flow metric $g_j$.
Put $A_j = j^2$.
If $(x, t) \in \psi_j^{-1}([0, A_j))$ then
$|R_j(x, t)| \le j$ and $|t| \le j$.
In particular,
\begin{equation} \label{canapply}
|R_j(x, t)| \le j \le \mu \rho_j(0)^{-2} - r(j)^{-2} \le \rho_j(0)^{-2}
\le \rho_j(t)^{-2},
\end{equation}
so $(x,t)$ avoids the surgery region.

We claim that there are functions $\r$, $C_k$ so that Definition
\ref{definition3.1} holds for all $j$. Suppose that $\psi_j(x, t) \le A_j$.
When $t \le 0$
there is nothing to prove, so we assume that $t \ge 0$.
If $t \in \left[ 0, \frac{1}{100} \right]$ then the normalized 
initial conditions give uniform control on $g_j(t)$. If
$t \ge \frac{1}{100}$ then 
\begin{equation}
t\,r(t)^{-2} \ge \frac{1}{100}
r(0)^{-2} > 1 > \sigma^2,
\end{equation}
where 
$\sigma = \sigma(j)$
is the parameter from
Lemma \ref{lemma2.1} and we use the assumption that $r(0) < \frac{1}{10}$
from the beginning of Section \ref{surgerysec}. 
Then
\begin{equation}
t > \sigma^2 r(t)^2 \ge \sigma^2 \frac{1}{|R(x,t)| + r(t)^{-2}} =
\sigma^2 Q^{-1}, 
\end{equation}
so the
parabolic ball $P_-(x, t, \sigma Q^{- \: \frac12})$ of
Lemma \ref{lemma2.1} does
not intersect the initial time slice.
As $|R_j(x, t)| \le \mu \rho_j(0)^{-2} - r(j)^{-2}$ from
(\ref{canapply}),
we can apply
Lemma \ref{lemma2.1} to 
show that 
$\left( \M^j, g_{\M^j}, \psi_j \right)$
is $(\psi_j, A_j)$-controlled 
in the sense of Definition \ref{definition3.1first}.
Note that Lemma \ref{lemma2.1} gives bounds on the spatial and
time derivatives of the curvature tensor of $g_j(t)$, which
implies bounds on the derivatives of the curvature
tensor of the spacetime metric $g_{\M_j}$.

Finally, the function $\t_j$ and the vector field
$\partial_{\t_j}$, along with their covariant derivatives, are
trivially bounded in terms of $g_{\M^j}$.

After passing to a subsequence we may assume that the number $N$
of connected
components  of the initial time slice
$\M^j_0$ is independent of $j$.  Then the theorem follows
from the special  case when the
initial times slices are connected, since we may apply it to the
components separately.  Therefore we are reduced to proving the
theorem under the assumption that $\M^j_0$ is connected.

After passing to a subsequence,
Theorem \ref{theorem3.8} now gives a
pointed 
$(\psi_\infty,\infty)$-controlled
tuple
\begin{equation}
\left( \M^\infty,g_{\M^\infty}, \t_\infty,(\D_{\t})_\infty,\star_\infty,
\psi_\infty \right)
\end{equation}
satisfying (1)-(4) of Theorem \ref{theorem3.8}.  
We truncate $\M^\infty$ to the subset $\t_\infty^{-1}([0, \infty))$.
We now verify the claims of Theorem \ref{thm_convergence_flow_with_surgery}.

Part (1) of Theorem \ref{thm_convergence_flow_with_surgery} follows from
the statement of Theorem \ref{theorem3.8}.
Given $\overline{t} > 0 $ and $\overline{R} < \infty$,
Proposition \ref{curve} implies that there are
$r=r(\ol{R},\ol{t}) < \infty$ and $A=A(\ol{R},\ol{t}) < \infty$ so that 
the set 
$\{m \in \M^j_{\leq \ol{t}} \: : \: R(m)\leq \ol{R}\}$ is contained in
the metric ball $B(\star_j,r)$ in the Riemannian manifold
$(\psi_j^{-1}([0,A)),g_{\M^j})$.
Hence by the definition of $\psi_j$, and
part (1) of Theorem \ref{theorem3.8},
we get (\ref{eqn_u_j_contains}).

Pick 
$m_\infty \in \M^\infty$.
Put $t_\infty = \t(m_\infty)$.
By part (4) of Theorem \ref{theorem3.8}, we know that $m_\infty$ 
belongs to $V_j$ for 
large $j$.  Now part (3) of
Theorem \ref{theorem3.8} and the definition of $\psi_j$
imply that
\begin{equation}
\label{eqn_psi_infty}
\psi_\infty(m_\infty)=R(m_\infty)^2 + t_\infty^2\,.
\end{equation}
If $j$ is large then it makes sense to define
$(x_j,t_\infty)=(\Phi^j)^{-1}(m_\infty)\in \M^j$.
Then for $j$ large, $R(x_j,t_\infty)<R(m_\infty)+1$ and
so by Proposition \ref{curve}, there is a time-preserving curve
$\ga_j:[0,t_\infty]\ra \M^j$ such that
\begin{equation}
\max \left(
\length_{g_{\M^j}}(\ga_j),\max_{t \in [0, t_\infty]}
R(\ga_j(t)) \right) <C=C(R(m_\infty),t_\infty)\,.
\end{equation}
By part (1) of Theorem \ref{theorem3.8} 
we know that $\im(\ga_j)\subset U_j$ for 
large $j$.
By part (3) of Theorem \ref{theorem3.8}, for large $j$ the map $\Phi^j$ is an
almost-isometry.
Hence there are
$r=r(R(m_\infty),t_\infty) < \infty$ and 
$A=A(R(m_\infty),t_\infty) < \infty$ so that
$m_\infty$ is contained 
in the metric ball $B(\star_\infty,r)$ in the
Riemannian manifold
$((\psi_\infty)^{-1}([0,A)),g_{\M^\infty})$
Combining this with (\ref{eqn_psi_infty}) and
part (1) of Theorem \ref{theorem3.8} yields  (\ref{eqn_v_j_contains}).
This proves part (2) of Theorem \ref{thm_convergence_flow_with_surgery}.

Part (3) of 
Theorem \ref{thm_convergence_flow_with_surgery} now follows from 
part (2) of the theorem and  parts (2) and (3) of Theorem
\ref{theorem3.8}. 
 
Equation (\ref{Rest}) follows from (\ref{scalarsurg}) and the smooth
approximation in part (3) of Theorem \ref{thm_convergence_flow_with_surgery}.
Let $\V_j^{< \overline{R}}(t)$ be the volume of the
$\overline{R}$-sublevel set for the scalar curvature function
on $\M^j_t$. (If $t$ is a surgery time, we replace $\M_t$ by $\M_t^+$.) 
Let $\V_\infty^{< \overline{R}}(t)$ be the volume of the
$\overline{R}$-sublevel set for the scalar curvature function
on $\M^\infty_t$. Then
\begin{equation}
\V_\infty(t) = \lim_{\overline{R} \ra \infty}
\V_\infty^{< \overline{R}}(t) = 
\lim_{\overline{R} \ra \infty} \lim_{j \ra \infty}
\V_j^{< \overline{R}}(t) \le 
\limsup_{j \ra \infty} \V_j(t). 
\end{equation}
Part (4) of Theorem \ref{thm_convergence_flow_with_surgery}
now follows from combining this with (\ref{volumesurg}).

Next, we verify the volume convergence
assertion (5). Let $\| \cdot \|_T$ denote the sup norm on
$L^\infty([0, T])$.
By parts 
(2) and (3) of Theorem \ref{thm_convergence_flow_with_surgery}, there is some
$\epsilon_3^T(j, \overline{R}) > 0$,
with $\lim_{j \ra \infty} \epsilon_3^T(j, \overline{R}) = 0$,
so that
\begin{equation} \label{comb1}
\| \V_\infty^{< \overline{R}} - \V_j^{< \overline{R}} 
\|_T \le \epsilon_3^T(j, \overline{R}).
\end{equation}

Also,
if $\overline{S} > \overline{R}$ then
\begin{equation} \label{c1}
\| \V_\infty^{< \overline{S}} - \V_\infty^{< \overline{R}} \|_T  =
\lim_{j \ra \infty} 
\| \V_j^{< \overline{S}} - \V_j^{< \overline{R}} \|_T \le
\limsup_{j \ra \infty}
\| \V_j - \V_j^{< \overline{R}} \|_T.
\end{equation}
Proposition \ref{volprop} implies
\begin{equation} \label{comb2}
\| \V_j - \V_j^{< \overline{R}} \|_T \le 
\left( \epsilon_1^T(\overline{R}) + 
\epsilon_2^T(\delta_j(0)) \right) V_0.
\end{equation}
Combining (\ref{c1}) and (\ref{comb2}), and
taking $\overline{S} \ra \infty$, gives
\begin{equation} \label{comb3}
\| \V_\infty - \V_\infty^{< \overline{R}} \|_T  \le
\epsilon_1^T(\overline{R}) \: V_0.
\end{equation}
Combining (\ref{comb1}), (\ref{comb2}) and (\ref{comb3}) yields
\begin{equation}
\| \V_j - \V_\infty \|_T \le  
\left( 2 \epsilon_1^T(\overline{R}) + 
\epsilon_2^T(\delta_j(0)) \right) V_0 + \epsilon_3^T(j, \overline{R}).
\end{equation}
Given $\sigma > 0$, we can choose $\overline{R}< \infty$ so that
$2 \epsilon_1^T(\overline{R}) V_0 < \frac12 \sigma$. Given this
value of $\overline{R}$, we
can choose $J$ so that if $j \ge J$ then
\begin{equation}
\epsilon_2^T(\delta_j(0)) V_0 + \epsilon_3^T(j, \overline{R})
< \frac12 \sigma.
\end{equation}
Hence if $j \ge J$ then
$\| \V_j - \V_\infty \|_T < \sigma$. This shows that
$\lim_{j \ra \infty} \V_j = \V_\infty$, uniformly on compact
subsets of $[0, \infty)$, and proves part (5) of Theorem  
Theorem \ref{thm_convergence_flow_with_surgery}.

Finally, we check that $\M^\infty$ is a 
singular Ricci flow
in the sense of Definition \ref{def_singular_ricci_flow}.
Using part (3) of Theorem \ref{thm_convergence_flow_with_surgery}
one can pass 
the Hamilton-Ivey pinching condition, 
canonical neighborhoods, and the noncollapsing condition
from the $\M^j$'s to $\M^\infty$, and so parts (b) and (c) of  Definition 
\ref{def_singular_ricci_flow} hold.

We now verify part
(a) of Definition \ref{def_singular_ricci_flow}. We start with a statement
about parabolic neighborhoods in $\M^\infty$.
Given $T >0$ and $\overline{R} < \infty$,
suppose that $m_\infty\in\M^\infty$
has $\t(m_\infty)\leq T$ and $R(m_\infty)\leq \ol{R}$.
For large $j$, put 
$\widehat{m}_j =(\Phi^j)^{-1}(m_\infty)\in \M^j$.  
Lemma \ref{lemma2.1}
supplies parabolic regions centered at the 
$\widehat{m}_j$'s
which pass to $\M^\infty$. Hence
for some $r=r(T,\overline{R})>0$, the forward and backward parabolic regions
$P_+(m_\infty,r)$ and $P_-(m_\infty,r)$ are unscathed. 
There is some $K = K(T, \overline{R}) < \infty$ 
so that when
equipped with the spacetime metric $g_{\M^\infty}$, the union
$P_+(m_\infty,r)\cup P_-(m_\infty,r)$
is $K$-bilipschitz homeomorphic to a Euclidean parabolic region.

From (\ref{Rest}), we know that $R$ is bounded below on $\M^\infty_{[0, T]}$.
In order to show that $R$ is proper on $\M^\infty_{[0, T]}$, we
need to show that any sequence in a sublevel set of $R$ has a
convergent subsequence.
Suppose that $\{ m_k\}_{k=1}^\infty \subset \M^\infty$   is a sequence with
$\t(m_k)\leq T$ and $R(m_k)\leq \ol{R}$. 
After passing to a subsequence,
we may assume that $\t(m_k)\ra  t_\infty\in [0,\infty)$.  
Then the regions 
$P( m_k,\frac{r}{100}, r^2) \cup P( m_k,\frac{r}{100}, - r^2)$ will intersect
the time slice $\M_{t_\infty}$ in regions whose volume 
is bounded below by $\const r^3$.  By the volume bound in part (4) of
Theorem \ref{thm_convergence_flow_with_surgery}, only finitely many of these
can be disjoint.  Therefore,  after passing to a subsequence, 
$\{ m_k\}_{k=1}^\infty$ is contained in 
$P_+(m_\ell,\frac{r}{2})
\cup P_-(m_\ell,\frac{r}{2})$ for 
some $\ell$. As 
$P_+(m_\ell,\frac{r}{2})
\cup P_-(m_\ell,\frac{r}{2})$
has compact closure in $P_+(m_\ell,r) \cup P_-(m_\ell,r)$, a 
subsequence of $\{ m_k\}_{k=1}^\infty$ converges.  This verifies part
(a) of Definition \ref{def_singular_ricci_flow}.
\end{proof}

\section*{Part II}
\section{Basic properties of singular Ricci flows}
\label{sec_singular_ricci_flows_basic_properties}

In this section we prove some initial structural properties of
Ricci flow spacetimes and singular Ricci flows.
In Subsection \ref{subsec6.1} we justify the maximum principle on a Ricci flow
spacetime and apply it to a get a lower scalar curvature bound.
In Subsection \ref{subsec6.2} we prove some results about volume evolution
for Ricci flow spacetimes which satisfy certain assumptions, that
are satisfied in particular for singular Ricci flows. The main
result in Subsection \ref{subsec6.3} says that if $\M$ is a singular Ricci
flow and $\gamma_0, \gamma_1 : [t_0, t_1] \ra \M$ are two time-preserving
curves, such that $\gamma_0(t_1)$ and $\gamma_1(t_1)$ are in the
same connected component of $\M_{t_1}$, then
$\gamma_0(t)$ and $\gamma_1(t)$ are in the
same connected component of $\M_{t}$ for all $t \in [t_0, t_1]$.   

We recall the notion of a Ricci flow spacetime from 
Definition \ref{def_ricci_flow_spacetime}. In this section, 
we will consider it to only be defined for nonnegative time, i.e.
$\t$ takes value in $[0, \infty)$.
We also recall the metrics
$g_{\M}$ and $g_{\M}^{qp}$ from Definition \ref{spacetimemetric}. 
Let $n+1$ be the dimension of $\M$.
Our notation is
\begin{itemize}
\item $\M_t = \t^{-1}(t)$,
\item $\M_{[a,b]} = \t^{-1}([a,b])$ and
\item $\M_{\le T} = \t^{-1}([0, T])$.
\end{itemize}

\subsection{Maximum principle and scalar curvature} \label{subsec6.1}

In this subsection we prove a maximum principle on Ricci flow
spacetimes and apply it to get a lower bound on scalar curvature.

\begin{lemma} \label{maximumprin}
Let $\M$ be a Ricci flow spacetime.
Given $T \in (0, \infty)$,
let $X$ be a smooth vector field on $\M_{\le T}$ with $X \t = 0$.
Given a smooth function
$F : \R \times [0, T] \rightarrow \R$, 
suppose that $u \in C^\infty( \M_{\le T})$ is
a proper function, bounded 
above,
which satisfies
\begin{equation}
\partial_{\t} u \le \triangle_{g(\t)} u + X u + F(u,\t).
\end{equation}
Suppose further that $\phi : [0, T] \rightarrow \R$ satifies
\begin{equation}
\partial_{\t} \phi = F(\phi(\t), \t)
\end{equation}
with initial condition $\phi(0) = \alpha \in \R$.
If $u \le \alpha$ on $\M_0$ then $u \le \phi \circ \t$ on $\M_{\le T}$.
\end{lemma}
\begin{proof}
As in \cite[Pf. of Theorem 3.1.1]{Topping_lectures}, for $\epsilon > 0$, we
consider the ODE
\begin{equation}
\partial_t \phi_\epsilon = F(\phi_\epsilon(t), t) + \epsilon
\end{equation}
with initial condition $\phi_\epsilon(0) = \alpha + \epsilon$.
It suffices to show that for
all small $\epsilon$ we have $u < \phi_\epsilon$ on $\M_{\le T}$.

If not then we can find some $\epsilon > 0$ so that the property 
$u < \phi_\epsilon$ fails on $\M_{\le T}$.  As $u$ is proper and
bounded
above,
there is a first time $t_0$ so that the property
fails on $\M_{t_0}$, and an $m \in \M_{t_0}$ so that
$u(m) = \phi_\epsilon(t_0)$. The rest of the argument is the same
as in \cite[Pf. of Theorem 3.1.1]{Topping_lectures}.
\end{proof}

\begin{lemma} \label{scalarlem}
Let $\M$ be a Ricci flow spacetime.
Suppose that for each $T \ge 0$, the scalar curvature $R$
is proper and bounded below on $\M_{\le T}$.
Suppose that the initial scalar curvature bounded below by
$-C$, for some $C \ge 0$. Then
\begin{equation} \label{scalarcurv}
R(m) \: \ge \: - \: \frac{C}{1+\frac{2}{n} C \t(m)}.
\end{equation}
\end{lemma}
\begin{proof}
Since $R$ is proper on $\M_{\le T}$, we can apply 
Lemma \ref{maximumprin} to the evolution equation for $- R$
and follow the standard proof to get (\ref{scalarcurv}).
\end{proof}

\subsection{Volume} \label{subsec6.2}

In this subsection we first justify a Fubini-type statement for
Ricci flow spacetimes.  Then we show that certain standard
volume estimates for smooth Ricci flows extend to the setting
of Ricci flow spacetimes under two assumptions : first that the
quasiparabolic metric is complete along worldlines that do
not terminate at the time-zero slice, and second that in any
time slice almost all points have worldlines that extend
backward to time zero.

The Fubini-type statement is the following.

\begin{lemma} \label{vollemma}
Let $\M$ be a Ricci flow spacetime.
Given $0 < t_1 < t_2 < \infty$, 
suppose that $F : \M_{[t_1, t_2]} \ra \R$ is measurable 
and bounded below.
Suppose that $\M_{[t_1,t_2]}$ has finite volume with respect to 
$g_\M$.
Then
\begin{equation}
\int_{\M_{[t_1, t_2]}} F \: \dvol_{g_\M} =
\int_{t_1}^{t_2} \int_{M_t} F \: \dvol_{g(t)} \: dt.
\end{equation} 
\end{lemma}
\begin{proof}
As $F$ is bounded below on $\M_{[t_1, t_2]}$, for the purposes
of the proof we can add a
constant to $F$ and assume that it is positive.
Given $m \in \M_{[t_1, t_2]}$, there is an time-preserving embedding
$e : (a,b) \times X \rightarrow \M$ with $e_* (\partial_s) = \partial_{\t}$
(where $s \in (a,b)$)
whose image is a neighborhood of $m$.
We can cover $\M_{[t_1, t_2]}$ by a countable collection
$\{P_i\}$ of such neighborhoods, with a subordinate partition
of unity $\{\phi_i\}$. Let $e_i : (a_i, b_i) \times X_i \rightarrow
P_i$ be the corresponding map. As
\begin{equation}
\int_{\M_{[t_1,t_2]}} \phi_i \: F \: \dvol_{g_\M} =
\int_{t_1}^{t_2}
 \int_{e_i(\{t\} \times X_i)}
\phi_i \: F \: \dvol_{g(t)} \: dt,
\end{equation}
we obtain
\begin{align}
\int_{\M_{[t_1,t_2]}}  F \: \dvol_{g_\M} & =
\sum_i \int_{\M_{[t_1,t_2]}} \phi_i \: F \: \dvol_{g_\M} \\
& = \sum_i \int_{t_1}^{t_2}
\int_{e_i(\{t\} \times X_i)} \phi_i \: F \: \dvol_{g(t)} \: dt \notag \\
& = \int_{t_1}^{t_2} \sum_i 
\int_{e_i(\{t\} \times X_i)} \phi_i \: F \: \dvol_{g(t)} \: dt \notag \\
& = \int_{t_1}^{t_2} 
\int_{\M_t} F \: \dvol_{g(t)} \: dt. \notag 
\end{align}

This proves the lemma.
\end{proof}

We now prove some results about the behavior of 
volume in Ricci flow spacetimes.

\begin{proposition} \label{vol1}
Let $\M$ be a Ricci flow spacetime.
Suppose that 
\begin{enumerate}
\renewcommand{\theenumi}{\alph{enumi}}
\item The quasiparabolic metric 
$g_{\M}^{qp}$ of Definition 
  \ref{spacetimemetric} is complete along worldlines.
\item If $B_t \subset \M_t$ is the set of points whose maximal worldline 
does not extend
backward to time zero, then $B_t$ has measure zero with 
respect to $\dvol_{g(t)}$, for each $t \ge 0$.
\item The initial time slice $\M_0$ has volume $\V(0) < \infty$.
\item 
The scalar curvature is proper and
bounded below on time slabs $\M_{\le T}$,
and the initial time slice has scalar curvature bounded below
by $-C$, with $C \ge 0$.
\end{enumerate} 

Let $\V(t)$ be the
volume of $\M_t$.
Then
\begin{enumerate}
\item 
$\V(t) \le \V(0) \left(1 + \frac{2}{n} C t \right)^{\frac{n}{2}}$.
\item $R$ is integrable on $\M_{[t_1, t_2]}$.
\item For all $t_1 < t_2$,
\begin{equation}
\V(t_2) - \V(t_1) = - \int_{\M_{[t_1, t_2]}} R \: \dvol_{g_{\M}}.
\end{equation}
\item The volume function $\V(t)$ is absolutely continuous.
\item
Given $0 \le t_1 \le t_2 < \infty$, we have
\begin{equation}
\V(t_2) - \V(t_1) \le 
\frac{C}{1 + \frac{2}{n} C t_1} 
\left( 1 +  \frac{2}{n} C t_2 \right)^{\frac{n}{2}}
\V(0) \cdot (t_2 - t_1). 
\end{equation}
\end{enumerate}
\end{proposition}
\begin{proof}
Given $0 \le t_1 < t_2 < \infty$,
there is a partition $\M_{[t_1, t_2]} =
\M_{[t_1, t_2]}^\prime \cup \M_{[t_1, t_2]}^{\prime \prime} \cup 
\M_{[t_1, t_2]}^{\prime \prime \prime}$, where
\begin{enumerate}
\item A point in $\M_{[t_1, t_2]}^\prime$ has a   worldline
that intersects both $M_{t_1}$ and $M_{t_2}$.
\item A point in $\M_{[t_1, t_2]}^{\prime \prime}$ has a   worldline
that intersects $M_{t_1}$ but not $M_{t_2}$.
\item A point in $\M_{[t_1, t_2]}^{\prime \prime \prime}$ 
has a   worldline
that does not intersect $M_{t_1}$.
\end{enumerate} 
By our assumptions, 
$\M_{[t_1, t_2]}^{\prime \prime \prime} \subset
\bigcup_{t \in [t_1, t_2]} B_t$ has measure zero
with respect to $\dvol_{g_{\M}}$; c.f. the proof of
Lemma \ref{vollemma}. Put
$X_1 = \M_{[t_1, t_2]}^\prime \cap M_{t_1}$ and 
$X_2 = \M_{[t_1, t_2]}^{\prime \prime} \cap M_{t_1}$.
For $s \in [t_1, t_2]$, 
there is a natural embedding $i_s : X_1 \rightarrow \M_s$
coming from flowing along worldlines. The complement
$M_{t_2} - i_{t_2}(X_1)$ has measure zero.
Thus
\begin{equation} \label{add1}
\V(t_2) = \int_{M_{t_2}} \dvol_{g(t_2)} = 
\int_{X_1} i_{t_2}^* \dvol_{g(t_2)}.
\end{equation}
and
\begin{equation} \label{add2}
\V(t_1) = \int_{X_1} \dvol_{g(t_1)} + \int_{X_2} \dvol_{g(t_1)}
\end{equation}

Given $x \in X_1$, let $\gamma_x : [t_1, t_2] \rightarrow \M_{[t_1,t_2]}$
be its worldline. Put
\begin{equation} \label{jacobian}
J_s(x) = \frac{i_s^* \dvol_{g(s)}}{\dvol_{g(t_1)}} (x).
\end{equation}
From the Ricci flow equation,
\begin{equation} \label{jacrf}
J_s(x) = e^{- \int_{t_1}^s R(\gamma_x(u)) \: du}.
\end{equation} 
Using Lemma \ref{scalarlem},
\begin{align}
\V(t_2) & = \int_{X_1} J_{t_2}(x) \: \dvol_{g(t_1)}(x) \le
\int_{X_1} e^{\int_{t_1}^{t_2} \frac{C}{1 + \frac{2C}{n} u} \: du}
\: \dvol_{g(t_1)} \\
& = \V(t_1) \left( 
\frac{1 + \frac{2C}{n} t_2}{1 + \frac{2C}{n} t_1} \right)^{\frac{n}{2}}.
\notag
\end{align}
When $t_1 = 0$, this proves part (1) of the proposition.

Next,
\begin{align} \label{add3}
\int_{X_1}(i_{t_2}^* \dvol_{g(t_2)}-\dvol_{g(t_1})
& =\int_{X_1} \left(
J_{t_2}(x)
  -1 \right) \dvol_{g(t_1)}(x) \\
& =\int_{X_1} \int_{t_1}^{t_2} \frac{dJ_{s}(x)}{ds}
 \: ds \: \dvol_{g(t_1)}(x)
\notag \\
& = - \int_{X_1} \int_{t_1}^{t_2} R(\gamma_x(s)) 
\: J_s(x)
 \: ds \: \dvol_{g(t_1)}(x)
\notag \\
& = - \int_{t_1}^{t_2} \int_{X_1} R \: i_s^* \dvol_{g(s)}  \:  ds
\notag \\
& = -\int_{\M_{[t_1,t_2]}^{\prime}} R \: \dvol_{g_{\M}}, \notag
\end{align}
where we applied Lemma \ref{vollemma} with
$F \: = \:  R \: 1_{\M_{[t_1,t_2]}^{\prime}}$ in the last step.

Given $x \in X_2$, let $e(x) \in (t_1, t_2)$ be the supremal
extension time of its worldline.
From the completeness of $g_{\M}^{qp}$ on worldlines that do not terminate
at the time-zero slice,
\begin{equation} \label{used1}
\int_{t_1}^{e(x)} R(\gamma_x(u)) \: du = \infty.
\end{equation}
Thus $\lim_{s \rightarrow e(x)} J_s(x) = 0$, so
\begin{align} \label{add4}
- \int_{X_2} \dvol_{g(t_1)}
& = \int_{X_2} \int_{t_1}^{e(x)} \frac{dJ_s(x)}{ds}
 \: ds \: \dvol_{g(t_1)}(x) \\
& = - \int_{X_2} \int_{t_1}^{e(x)} R(\gamma_x(s)) 
\: J_s(x)
 \: ds \: \dvol_{g(t_1)}(x)
\notag \\
& = - \int_{X_2} \int_{t_1}^{e(x)} R
 \: ds \: i_s^* \dvol_{g(s)}
\notag \\
& = -\int_{\M_{[t_1,t_2]}^{\prime \prime}} R \: \dvol_{g_{\M}}. \notag
\end{align}

Part (3) of the proposition follows from combining equations 
(\ref{add1}), (\ref{add2}), (\ref{add3}) and
(\ref{add4}). Part (2) of the proposition is now an immediate consequence.

By Lemma \ref{vollemma} and part (3) of the proposition,
the function $t\mapsto \int_{\M_t}R\dvol$ is locally-$L^1$ on
$[0,\infty)$ with respect to Lebesgue measure. 
This implies part (4) of the 
proposition.

To prove part (5) of the proposition,
using Lemma \ref{scalarlem} and parts (1) and (3)
of the proposition, we have
\begin{align}
\V(t_2) - \V(t_1) & = 
- \int_{t_1}^{t_2} \int_{\M_t} R \: \dvol_{g(t)} \: dt \\
& \le \int_{t_1}^{t_2} \frac{C}{1 + \frac{2}{n} C t} \V(t) \: dt \notag \\
& \le
\frac{C}{1+ \frac{2}{n} C t_1} 
\left( 1 + \frac{2}{n} Ct_2 \right)^{\frac{n}{2}}
\V(0) \cdot (t_2 - t_1). \notag
\end{align}
This proves the proposition.
\end{proof}

\subsection{Basic structural properties of singular Ricci flows}
\label{subsec6.3}

In this subsection we collect a number of properties of singular Ricci
flows, the latter being 
in the sense of Definition \ref{def_singular_ricci_flow}. 
We first show the completeness of the quasiparabolic metric.

\begin{lemma} \label{quasip}
If $\M$ is a singular Ricci flow  then the
quasiparabolic metric $g_{\M}^{qp}$ of Definition 
\ref{spacetimemetric} is complete away from the
time-zero slice.
\end{lemma}
\begin{proof}
Suppose that $\gamma : [0, \infty) \rightarrow \M$ is a curve
that goes to infinity in $\M$, with
$\t \circ \gamma$ bounded away from zero. We want to show that
its quasiparabolic length is infinite. 
If $\t \circ \gamma$ is not bounded then
the quasiparabolic length of $\gamma$ is infinite
from the definition of $g_{\M}^{qp}$, so we can assume
that $\t \circ \gamma$ takes value in some interval $[0, T]$.
Since $R$ is proper and bounded below on $\M_{\le T}$, we have
$\lim_{s \rightarrow \infty} R(\gamma(s)) = \infty$.
After truncating the initial part of $\gamma$, we can assume that
$R(\gamma(s)) \ge r(T)^{-2}$ for all $s$. In particular, 
each point $\gamma(s)$ is in a canonical neighborhood.
Now
\begin{equation}
\frac{dR(\gamma(s))}{ds} = \frac{\partial R}{\partial t} 
\frac{dt}{ds} + 
\langle \nabla R, \gamma^\prime \rangle_g,
\end{equation}
so
\begin{equation}
\left| \frac{dR(\gamma(s))}{ds} \right|^2 \le
2 \left(
\left| \frac{\partial R}{\partial t} \right|^2
\left| \frac{dt}{ds} \right|^2 +  
\left| \nabla R \right|_g^2 \left| \gamma^\prime \right|_g^2
\right).
\end{equation}

The gradient estimates in (\ref{gradest}), of
the form
\begin{equation}
|\nabla R| < \const R^{\frac32}, \: \: \: \: \: \:
|\partial_t R| < \const R^2,
\end{equation}
are valid for points in a canonical neighborhood of a singular
Ricci flow solutions.
Then
\begin{equation}
\left| \frac{dR(\gamma(s))}{ds} \right|^2 \le C
R^2 
\left| \frac{d\gamma}{ds} \right|_{g^{qp}_{\M}}^2  
\end{equation}
for some universal $C < \infty$.
We deduce that
\begin{equation}
\int_0^\infty \left| R^{-1} \frac{dR}{ds} \right| \: ds
\: \le \: C^{\frac12} \int_0^\infty \left| \frac{d\gamma}{ds} 
\right|_{g^{qp}_{\M}}
\: ds.
\end{equation}
Since the left-hand side is infinite, the quasiparabolic length of
$\gamma$ must be infinite.  This proves the lemma.
\end{proof}

The next lemma gives the existence of unscathed forward and
backward parabolic neighborhoods of
a certain size around a point, along with geometric bounds on those
neighborhoods.

\begin{lemma} \label{singforback}
Let $\M$ be a singular Ricci flow.
Given $T < \infty$, there are numbers 
$\sigma = \sigma(T) > 0$, $i_0 = i_0(T) > 0$ and 
$A_k = A_k(T) < \infty$, $k \ge 0$, with the following
property.  If $m \in \M$ and $\t(m) \le T$, put 
$Q = |R(m)| + r(\t(m))^{-2}$. Then
\begin{enumerate}
\item The forward parabolic ball
$P_+(m, \sigma Q^{- \: \frac12})$
and the backward parabolic ball 
$P_-(m, \sigma Q^{- \: \frac12})$
are unscathed.
\item 
$|\Rm| \le A_0 Q$,
$\inj \ge i_0 Q^{- \: \frac12}$ and
$|\nabla^k \Rm| \le A_k Q^{1 + \frac{k}{2}}$
on the union
$P_+(m, \sigma Q^{- \: \frac12}) \cup
P_-(m, \sigma Q^{- \: \frac12})$
of the forward and backward parabolic balls.
\end{enumerate}
\end{lemma}
\begin{proof}
The proof is the same as that of Lemma \ref{lemma2.1}.
\end{proof}

The next two propositions characterize the high-scalar-curvature
part of a time slice.

\begin{proposition} \label{goodeps}
Let $\M$ be a singular Ricci flow.
For all $\eps_1>0$, there is a scale function $r_1:[0,\infty)\ra (0,\infty)$
with $r_1(t)\leq r(t)$, such that for every point 
$m \in \M$
with $R(m)>r_1(\t(m))^{-2}$
the
$\eps_1$-canonical neighborhood assumption holds, and moreover
$(\M,m)$ is $\eps_1$-modelled on a $\kappa$-solution.
(Recall that here $\kappa=\kappa(t)$, i.e. we are 
suppressing the time dependence in our notation.)
\end{proposition}
\begin{proof}
The proof is similar to the proof of  \cite[Theorem 52.7]{Kleiner-Lott_perelman_notes}.
Suppose that the proposition is not true.  Then there is a sequence of
singular Ricci flows $\{\M_j\}_{j=1}^\infty$, along with points
$m_j \in \M_j$, that together provide a counterexample. 
After passing to a subsequence, we extract a limiting flow,
as in Step 2 of the proof of \cite[Theorem 52.7]{Kleiner-Lott_perelman_notes}.
In that proof,
the existence of a limiting flow used a point selection argument from
Step 1 of the proof.
In the present case, because of the canonical neighborhood assumption
in the definition of singular Ricci flow, 
we do not have to perform point selection in order to extract a
limiting flow. 

Similarly, the use of the existing canonical
neighborhood assumption simplifies Step 2 of 
the proof of \cite[Theorem 52.7]{Kleiner-Lott_perelman_notes}.
The rest of the proof of the proposition 
is the same as in Steps 3 and 4 of the proof of
\cite[Theorem 52.7]{Kleiner-Lott_perelman_notes}.

As a further point,
time slices are not assumed to be compact
as in \cite[Theorem 52.7]{Kleiner-Lott_perelman_notes}, and
are therefore not necessarily complete. To deal with this, 
there are places in the
proof where one applies Lemma \ref{singforback} to
compensate for any incompleteness.
\end{proof}

\begin{proposition}
\label{prop_large_r_structure}
Let $\M$ be a singular Ricci flow.
For any $T<\infty$ and $\hat\eps>0$, 
there exist  $C_1=C_1(\hat\eps,T)<\infty$ and
$\ol{R}= \ol{R}(\hat\eps,T) < \infty$
such that for every 
$t\leq T$,  each  connected component of the time slice $\M_t$
has finitely many ends, each of which is an $\hat\eps$-horn.
Morever  for every $\ol{R}'\geq \ol{R}$, 
the superlevel set
$\M_t^{> \ol{R}'} = \{ m \in \M_t \: : \: R(m) > \ol{R}' \}$
is contained in a finite disjoint union
of properly embedded  three-dimensional submanifolds-with-boundary 
$\{N_i\}_{i=1}^k$ 
such that 
\begin{enumerate}
\item Each $N_i$ is contained in the superlevel set 
$\M_t^{> C_1^{-1}\ol{R}'}$.
\item The boundary $\D N_i$ has
scalar curvature  in the interval 
$(C_1^{-1}\ol{R}',C_1\ol{R}')$.
\item
For each $i$ one of the following holds:
\begin{enumerate}
\item $N_i$ is diffeomorphic to $S^1 \times S^2$ or
$I \times S^2$ and consists of $\hat\eps$-neck
points.  Note that here the interval $I$ can be open (a double horn), 
closed (a tube) or half-open (a horn).
\item $N_i$ is diffeomorphic to $D^3=\ol{B^3}$ or $\R P^3 - B^3$
and its boundary $\D N_i\simeq S^2$ consists of $\hat\eps$-neck points.
\item $N_i$ is diffeomorphic to $S^3$, $\R P^3$, or   $\R P^3 \# \R P^3$.
\item $N_i$ is diffeomorphic to
a spherical space form other than $S^3$ or $\R P^3$, and $(\M,x)$ is $\hat\eps$-modelled on a shrinking round spherical space form, for all $x\in N_i$.
\end{enumerate}
\item In cases (b) and (c) of (3), 
if $S_i\subset N_i$ is a subset consisting of
non-$\hat\eps$-neck points, such that for any two distinct elements
$s_1,s_2\in S_i$ we have $d_{\M_t}(s_1,s_2)> C_1R^{-\frac12}(s_1)$, 
then  the cardinality  $|S_i|$ is at most $1$ in case (b) and 
at most $2$ in case (c).   \item Each $N_i$ with nonempty boundary has volume at least 
$C_1^{-1}\left( \ol{R}' \right)^{-\frac32}$.
\end{enumerate}
\end{proposition}
\begin{proof}
The proof is the same as in \cite[Section 67]{Kleiner-Lott_perelman_notes}.
\end{proof}

We now prove a statement about preservation of connected
components when going backwards in time.

\begin{proposition} \label{sameconnected}
Let $\M$ be a singular Ricci flow.
If $\ga_0,\ga_1:[t_0,t_1]\ra \M$ are time-preserving curves, and 
$\ga_0(t_1)$, $\ga_1(t_1)$ 
lie in the same connected component of $\M_{t_1}$,
then $\ga_0(t)$, $\ga_1(t)$ lie in the same connected component
of $\M_{t}$ for every $t\in [t_0,t_1]$. 
\end{proposition}
\begin{proof}
The idea of the proof is to consider the values of $t$ for
which $\gamma_0(t)$ can be joined to $\gamma_1(t)$ in $\M_t$, and the
possible curves $c_t$ in $\M_t$ that join them.  
Among all such curves $c_t$, we look at one which minimizes the maximum
value of scalar curvature along the curve. Call this threshold
value of scalar curvature $R_{crit}(t)$. We will argue that $c_t$
can only intersect the high-scalar-curvature part of $\M_t$ in
its neck-like regions.  But the scalar curvature in a neck-like
region is strictly decreasing when one goes backward in time;
this will imply that $R_{crit}(t)$, when large, is 
decreasing when one goes backward in time,
from which the lemma will follow.

To begin the formal proof, suppose that the lemma is false.   Let
 $S\subset [t_0,t_1]$ be the set of times $t\in [t_0,t_1]$ such that 
$\ga_0(t)$ and $\ga_1(t)$ lie in the same connected component of $\M_t$;
note that $S$ is open.  
Put $\hat t=\inf\{t\mid [t,t_1]\subset S\}$. Then
$\hat t>\frac{1}{100}$ since $\M_{[0, \frac{1}{100}]}$ is a product.
Also, for $i\in \{0,1\}$ and $t \in [\hat t, t_1]$ close to $\hat t$,
if $\hat\ga_i(t)\in \M_t$ is the worldline of $\ga_i(\hat t)$ at time $t$
then $\hat\ga_i(t)$ lies in the same connected component of $\M_t$
as $\ga_i(t)$.  Therefore, after 
reducing $t_1$ if necessary,
we may assume without loss of generality that $\ga_i$ 
is a worldline.

For $t \in [t_0, t_1]$ and $\ol{R} < \infty$, 
put $\M^{\leq \ol{R}}_t=\{m\in\M_t\mid R(m)\leq\ol{R}\}$.
For $t\in (\hat t,t_1]$, let $\calr_t$ be the set of $\ol{R}\in \R$
such that $\ga_0(t)$ and $\ga_1(t)$ lie in the same connected component
of $\M^{\leq \ol{R}}_t$.  Put $R_{crit}(t)=\inf\calr_t$.
Since $R:\M_t\ra \R$ is proper, the sets 
$\{\M^{\leq\ol{R}}_t\}_{\ol{R}>R_{crit}(t)}$ 
 are compact
and nested, which implies that $\ga_0(t)$ and $\ga_1(t)$ lie in the same
component of
 $\M^{\leq R_{crit}(t)}_t=
\bigcap_{\ol{R}>R_{crit}(t)}\M^{\leq \ol{R}}_t$,  i.e.
$R_{crit}(t)\in \calr_t$.

By Lemma \ref{singforback}, 
there is a $C<\infty$
such that for any $t \in (\frac{1}{100}, t_1]$ and any $m\in \M_t$
 there is a 
$\tau=\tau(t_1,R(m))>0$, where $\tau$
is a continuous function that is
 nonincreasing in $R(m)$, such that the worldline $\ga_m$ of $m$ is defined
 and  satisfies 
\begin{equation}
\label{eqn_r_change_small}
\left|\frac{dR(\gamma_m(t))}{dt} \right| <C \cdot R(m)^2 
\end{equation}
in the time interval $(t-\tau,t+\tau)$.
Now for $t \in (\hat t, t_1]$
and $t^\prime$ satisfying $|t'-t|<\tau(t_1,R_{crit}(t))$,
let $Z_{t,t'}\subset \M_{t'}$ denote the result of flowing 
$\M^{\leq R_{crit}(t)}_t$ 
under $\D_{\t}$ for an elapsed time $t'-t$.
Then $Z_{t,t'}$ is well-defined
and contains $\ga_0(t')$ and $\ga_1(t')$
in the same component, so
\begin{equation}
R_{crit}(t')\leq \max_{Z_{t,t'}}R\leq R_{crit}(t)+C \cdot 
R_{crit}(t)^2 \cdot |t - t^\prime|.
\end{equation}
This implies that $R_{crit}:(\hat t,t_1]\ra \R$ is locally Lipschitz
(in particular continuous)
and that
\begin{equation}
\label{eqn_r_crit_blows_up}
R_{crit}(t)\ra \infty
\end{equation}
as $t\ra \hat t$ from the right; 
otherwise there would be a sequence $t_i \rightarrow \hat t$ along which
$R_{crit}$ is uniformly bounded above by some $\hat R < \infty$, which
would allow us to construct $Z_{t_i, \hat t}$ whenever
$|\hat t - t_i| < \tau(t_1, \hat R)$, contradicting the
definition of $\hat t$.

We now concentrate on $t$ close to $\hat t$.
Suppose that for some
$t\in (\hat t,t_1]$ we have
\begin{equation} \label{bigg}
R_{crit}(t)\gg\max(r^{-2}(t_1), \max_{[t_0,t_1]}R\circ \ga_0,
\max_{[t_0,t_1]}R\circ \ga_1)\,.
\end{equation}
Then by 
Proposition \ref{prop_large_r_structure}
the superlevel set 
$\M^{> (R_{crit}(t) - 1)}_t$ is contained in a
finite union of components 
$\{N_{i,t}\}_{i=1}^{k_t}$, 
each diffeomorphic to one
of the possibilities 
(a)-(d)
in the statement of the proposition.
Let $X_t$ be the result of removing from 
$\M_t$
the interior of each $N_{i,t}$ that is not of type 
(a).
Since $\ga_0(t)$ and
$\ga_1(t)$ lie outside $\bigcup_{i=1}^{k_t} N_{i,t}$, 
and each $N_i$ of type 
(b)-(d)
has at most one boundary component, it follows that $\ga_0(t)$ and $\ga_1(t)$
lie in the same connected component of $X_t$.

For $t^\prime < t$ 
close to $t$, let $X_{t'}\subset\M_{t'}$ be the result of flowing
$X_t$ under $\D_{\t}$.   Now $Y_t=X_t\cap \M_t^{> (R_{crit}(t)-1)}$ 
consists of
$\eps$-neck points, and at such a point the scalar curvature is 
strictly increasing
as a function of time. Hence there is a $\tau_1>0$ such that the worldline 
$\ga_m:[t-\tau_1,t]\ra\M$ of any $m\in Y_t$ satisfies
\begin{equation}
R(\ga_m(t'))<R(m)\leq R_{crit}(t)
\end{equation}
for $t'\in [t-\tau_1,t)$.
This implies that $R_{crit}(t')<R_{crit}(t)$ when $t'<t$ is close to $t$,
again under the assumption (\ref{bigg}),
which contradicts (\ref{eqn_r_crit_blows_up}).
\end{proof}

We now state a result about connecting a point in a singular Ricci flow
$\M$ to the time-zero slice by a curve whose length with respect to the spacetime metric $g_{\M}$ is quantitatively 
bounded, and along which the scalar curvature is quantitatively bounded.  This is similar to Proposition \ref{curve}.

\begin{proposition} \label{curve2}
Let $\M$ be a singular Ricci flow.
Given $T, R_0 < \infty$,
there are $L = L(R_0, T) < \infty$ and
$R_1 = R_1(R_0, T) < \infty$ with the following 
property. Suppose that $R(m_0)\leq R_0$, with $t_0 = \t(m_0) \leq T$.
Then
there is a time preserving curve $\ga:[0,t_0]\ra \M$ with
$\gamma(t_0) = m_0$ and $\length(\ga)\leq L$ so that
$R(\ga(t))\leq R_1$ for all $t\in [0,t_0]$.
\end{proposition}
\begin{proof}
The proof is the same as that of Proposition \ref{curve}.
\end{proof}

Finally, we give a compactness result for the space of
singular Ricci flows.

\begin{proposition} \label{rfcompactness}
Let $\{\M^i\}_{i=1}^\infty$ be a sequence of singular Ricci flows
with a fixed choice of parameters $\epsilon$, $r$ and
$\kappa$ whose initial conditions $\{\M^i_0\}_{i=1}^\infty$ lie in
a compact family in the smooth topology. Then a subsequence
of $\{\M^i\}_{i=1}^\infty$ converges, in the sense of Theorem 
\ref{thm_convergence_flow_with_surgery}, to a singular Ricci flow
 $\M^\infty$.
\end{proposition}
\begin{proof}
Using
Proposition \ref{curve2}, the proof is the same as that of
Proposition \ref{thm_convergence_flow_with_surgery}.
\end{proof}

\begin{remark}
In the setting of Proposition \ref{rfcompactness},
if we instead assume that the (normalized) initial conditions have a uniform
upper volume bound then we can again take a convergent subsequence
to get a limit  
Ricci flow spacetime $\M^\infty$.     In this case the
time-zero slice $\M^\infty_0$ will generally only be $C^{1,\alpha}$-regular,
but $\M^\infty$ will be smooth on $\t^{-1}(0, \infty)$. 
\end{remark}

\section{Stability of necks} \label{stab}

In this section we fix $\kappa>0$; the dependence
of various constants on this choice of $\kappa$ is implicit.  
We recall the notion of a $\kappa$-solution from Appendix \ref{appkappa}.
In this section we establish a new dynamical stability property
of caps and necks in $\kappa$-solutions, which we will use in 
Section \ref{sec_finiteness}  to show that
a bad worldline $\ga:I\ra\M$ is confined to a cap region as
$t$ approaches the blow-up time $\inf I$.

Conceptually speaking, the stability assertion is that among 
pointed $\kappa$-solutions, the round cylinder is an attractor under backward flow;
similarly, under forward flow, non-neck points form an attractor.  The rough idea
of the argument is as follows.  Suppose that $\de\ll 1$ and $(x_0,0)$ is a 
$\de$-neck in a $\kappa$-solution $\M$.  Then $(x_0,t)$ will also 
be  neck-like as long as $t < 0$ is 
not too negative.
One also knows that
 $\M$ has an asymptotic  soliton as $t\ra-\infty$, 
which is a shrinking round cylinder.   Thus $\M$ tends toward
neck-like geometry
as $t$ approaches $-\infty$, which is the desired stability property.  
However, there is a  catch here:
the asymptotic soliton is a pointed
limit of a sequence where the basepoints are not fixed, and so a priori it
says nothing about the asymptotic geometry near $(x_0,t)$ as $t\ra-\infty$. 
To address this we exploit the behavior of the $l$-function.

\subsection{The main stability asssertion}

We recall the notation ${\hat \M}(t)$  from
Section \ref{sec_notation} for the parabolic rescaling
of a Ricci flow spacetime $\M$. We recall the
notation $\cyl$ and $\sphere$  from
Section \ref{sec_notation} for the standard Ricci 
flow solutions. We also recall the notion of one Ricci
flow spacetime being
$\epsilon$-modelled on another one, from Appendix
\ref{appcloseness}. 

\begin{theorem}
\label{thm_stability_of_necks}
There is a 
$\de_{neck}=\de_{neck}(\kappa)>0$,
 and  
for all $\de_0,\de_1\leq\de_{neck}$
there is a 
$T=T(\de_0,\de_1,\kappa)\in(-\infty,0)$
with the following property.
Suppose  that
\begin{enumerate}
\renewcommand{\theenumi}{\alph{enumi}}
\item $\M$ is a  $\kappa$-solution with noncompact time slices,
\item $(x_0,0)\in\M$, 
\item $R(x_0,0)=1$, and 
\item $(\M, (x_0,0))$ is a $\de_0$-neck.
\end{enumerate}
Then either 
 $\M$ is  isometric
to the $\Z_2$-quotient of a shrinking round cylinder, or 
for all $t \in (- \infty, T]$,
\begin{enumerate}
\item 
$(\M, (x_0,t))$ is a $\de_1$-neck and
\item 
$\left( {\hat \M}(-t), (x_0, -1) \right)$
is $\delta_1$-close to $(\cyl, (y_0, -1))$, where $y_0 \in 
S^2 \times \R$.
\end{enumerate}
\end{theorem}

We recall the notion of a generalized neck from Appendix \ref{appcloseness}.

\begin{corollary}
\label{cor_de_de_over_4}
If $\de_{neck}=\de_{neck}(\kappa)$ as in the previous theorem, then
there is a $T=T(\kappa)<(-\infty,0)$ such that if $\M$ is a $\kappa$-solution
with noncompact time slices, $(x_0,0)\in\M$, $R(x_0,0)=1$, 
and $(x_0,0)$ is a generalized $\de_{neck}$-neck, then
$(x_0,T)$ is a generalized $\frac{\de_{neck}}{4}$-neck, and
$R(x_0,T)<\frac14 $. 
\end{corollary}

\subsection{Convergence of $l$-functions, asymptotic $l$-functions and asymptotic
solitons}
\label{subsec_convergence_l_functions}

In this subsection we prove some preparatory results about the convergence
of $l$-functions with regard to a convergent sequence of 
three-dimensional
$\kappa$-solutions.

Suppose that we are given that $R(x_1,t_1) \le C$ at some point $(x_1, t_1)$
in a $\kappa$-solution.
By compactness of the space of normalized pointed $\kappa$-solutions
(see Appendix \ref{appkappa}),
we obtain $R(x,t_1) \leq F(C, d_{t_1}(x,x_1))$ for some universal function $F$.
Since $R$ is pointwise nondecreasing in time 
in a $\kappa$-solution, 
we also have $R(x,t) \leq F(C, d_{t_1}(x,x_1)$
whenever $t \le t_1$.

\begin{lemma}[Curvature bound at basepoint] \label{curvbasepoint}
Let $\{(\M^j,(x_j,0))\}$ be a sequence of pointed $\kappa$-solutions, with 
$\sup_j R(x_j,0)<\infty$, and let $l_j:\M^j_{<0}\ra (0,\infty)$ be the reduced 
distance function with spacetime basepoint at $(x_j,0)$.  Then after passing to
a subsequence, we have convergence of
tuples 
\begin{equation}
(\M^j,(x_j,0),l_j)\ra (\M^\infty,(x_\infty,0),l_\infty)\,.
\end{equation}
Here the Ricci flow convergence is the usual smooth convergence on parabolic balls,
and $l_j$ converges to 
a locally Lipschitz function $l_\infty$
uniformly on compact subsets of 
$\M^\infty_{<0}$, after composition  with the comparison map.
\end{lemma}
\begin{proof}
Since $R(x_j,0)$ is uniformly bounded above, 
the compactness of the space of normalized pointed
$\kappa$-solutions implies that a subsequence,
which we relabel as $\{(\M^j,(x_j,0))\}$,
converges in the pointed smooth topology.   
Along the curve $\{x_j\} \times (-\infty, 0) \subset \M^j$, the scalar
curvature is bounded above by $R(x_j, 0)$. 
For any $a < b < 0$, this gives a uniform
upper bound on $l_j$ on $\{x_j\} \times [a,b] \subset \M^j$.
From (\ref{ye2.55}), we get that
$l$ is uniformly bounded on 
balls of the form $B(x_j, b, r) \subset \M^j_b$.
Then from (\ref{ye2.57}), we get that $l$ is uniformly bounded on
parabolic balls of the form
$P(x_j, b, r, a-b) \subset \M^j$. 
Using (\ref{ye2.54}) and passing to a
subsequence, we get convergence of $\{l_j\}$ to some $l_\infty$,
uniformly on compact subsets of $\M^\infty_{<0}$.
\end{proof}

\begin{proposition}[No curvature bound at basepoint]
\label{prop_convergence_no_curvature_bound}
Let $\{(\M^j,(x_j,0))\}$ be a sequence of pointed $\kappa$-solutions, and let 
$l_j:\M^j_{<0}\ra (0,\infty)$ be the reduced 
distance function with spacetime basepoint at $(x_j,0)$.  
Suppose that  $\{(y_j,-1)\in \M^j_{-1}\}$ is a sequence satisfying
$\sup_j l_j(y_j,-1)<\infty$.  Then after passing to a subsequence, the tuples
$(\M^j,(y_j,-1),l_j)$ converge to a  tuple $(\M^\infty,(y_\infty,-1),l_\infty)$
where:
\begin{enumerate}
\item 
The convergence $(\M^j,(y_j,-1))\ra (\M^\infty,(y_\infty,-1))$ is smooth
on compact subsets of $\M^\infty_{<0}$, and $\M^\infty$ is either a $\kappa$-solution
defined on $(-\infty,0)$ or the static solution on $\R^3$.
\item $\{l_j\}$ converges to 
a locally Lipschitz function
$l_\infty$
uniformly on compact subsets of $\M^\infty_{<0}$, 
after composition
with the comparison diffeomorphisms. 
\item The reduced volume functions $\tilde V_j:(-\infty,0)\ra (0,\infty)$
converge uniformly on compact sets to the function 
$\tilde V_\infty:(-\infty,0)\ra (0,\infty)$, where
\begin{equation} \label{redvol}
\tilde V_\infty(t)= (-t)^{-\frac{n}{2}} \int_{\M^\infty_t} e^{-l_\infty}\,
\dvol_{g(t)}.
\end{equation}
\item The function $l_\infty$ satisfies the differential inequalities
(24.6) and (24.8) of 
\cite{Kleiner-Lott_perelman_notes}.
If $\tilde V_\infty$ is constant on some time interval 
$[t_0,t_1]\subset (-\infty,0)$ then $\M^\infty$ is a gradient shrinking
soliton with potential $l_\infty$.
\end{enumerate}
\end{proposition}
\begin{proof}
From (\ref{ye2.57}), given $a < b < 0$, 
there is a uniform bound for $l_j$ on 
$\{y_j\} \times [a,b]$. From (\ref{ye2.53}),
there are bounds for $R_j$ on $\{y_j\} \times [a,b]$.
Then we get a uniform scalar curvature bound on the balls
$B(y_j, b, r) \subset \M^j_b$, and hence on the parabolic neighborhoods
$P(y_j, b, r, - \Delta t) \subset \M^j$.
Taking $a \rightarrow - \infty$ and $b \rightarrow 0$, and
applying a diagonal argument, we can assume that 
$\{(\M^j, (y_j, -1))\}$ converges
to a $\kappa$-solution $(\M^\infty, (y_\infty, -1))$. 
As in the proof of Lemma
\ref{curvbasepoint}, after passing to a subsequence we get
$l_j \ra l_\infty$ for some $l_\infty$ defined
on $\M^\infty_{<0}$. The rest is as in \cite{Ye_reduced,Ye_reduced_2}.
\end{proof}

\begin{lemma}
\label{lem_cyl_l_function}
Let $(\M,(x,0),l)$ be a shrinking round cylinder with $R\equiv 1$ at $t=0$,
where the $l$-function has spacetime basepoint $(x,0)$.   
Then:
\begin{enumerate}
\item For every $t\in (-\infty,0)$ the $l$-function $l:\M_t\ra \R$
attains its minimum uniquely
at $(x,t)$.
\item $\lim_{t \rightarrow - \infty} (\hat\M(-t),(x,-1),l) = 
(\cyl,(x_\infty,-1),l_\infty)$, where the coordinate $z$ for the
$\R$-factor in $\cyl$ satisfies
$z(x_\infty)=0$, and 
\begin{equation}
\label{eqn_l_infinity_cyl}
l_\infty=1+\frac{z^2}{-4t}\,.
\end{equation} 
\item $\lim_{t \ra - \infty} l(x,t) = 1$.
\end{enumerate} 
\end{lemma}
\begin{proof}
Part (1)  follows from the formula for $\L$-length:
\begin{equation}
\L(\ga)=\int_{\bar t}^0\sqrt{-t}(R+|\ga'(t)|^2)dt
\geq\int_{\bar t}^0\sqrt{-t}\,R\,dt
\end{equation}
with equality if and only if $\ga(t)=(x,t)$ for all $t\in [\bar t,0]$.
Part (2) follows from applying
Lemma \ref{lem_classification_gradient_soliton_kappa_solutions}
in Section \ref{subsec_gradient_shrinking_solitons}
to parabolic rescalings of $\M$. Part (3) is now immediate.
\end{proof}

From (\ref{eqn_l_infinity_cyl}), we get that $l_\infty$ is strictly decreasing
along backward worldlines,  except at its minimum value  $1$. In particular, if
$l_\infty(x,t_0)\leq 1+\hat{\eps}$ then $l_\infty(x,t_1)<1+\hat{\eps}$ for all
$t_1<t_0$.  By compactness, this stability property is inherited by any tuple
which approximates 
the shrinking round cylinder.

\begin{lemma}
\label{lem_cyl_l_function_stability}
For all $\hat\eps>0$ and $A\in (0,1)$, there is a $\bar\mu>0$ with the following
property.   Suppose that
$(\bar\M,(x,-1),l)$ is a $\kappa$-solution, with  $l$-function
based at  $(x,0)$, so that
\begin{enumerate}
\item $(\bar\M,(x,-1),l)$  is 
$\bar\mu$-close to $(\cyl,(y,-1),l_\infty)$ on the time
interval $[-A^{-1},-A]$, for some $(y,-1)\in\cyl$, and
\item
$l(x,t)\leq 1+\hat\eps$ for all $t\in [-1,-A]$.
\end{enumerate}
(Here $l_\infty(y, -1)$ need not be one.)
Then $l(x,t)< 1+\hat\eps$ for all $t\in [-A^{-1},-1]$.
\end{lemma}
\begin{proof}
If the lemma fails then for some $\hat\eps>0$ and $A\in (0,1)$, there is a sequence
$\{(\bar\M^j,(x_j,-1),l_j)\}$ so that for all $j$:
\begin{itemize}
\item $(\bar\M^j,(x_j,-1),l_j)$ is $j^{-1}$-close to $(\cyl,(y_j,-1),l_\infty)$
on the time interval $[-A^{-1},-A]$, 
for some 
$(y_j,-1)\in \cyl_{-1}$.
\item $l_j(x_j,t)\leq 1+\hat\eps$ for all $t\in [-1,-A]$.
\item 
$l_j(x_j,t_j)\geq 1+\hat\eps$ for some $t_j\in [-A^{-1},-1]$.
\end{itemize}
Using (\ref{eqn_l_infinity_cyl}) and
passing to a limit, 
we obtain $(y_\infty,-1)\in \cyl$ with the property
that $l_\infty(y_\infty,-A)\leq 1+\hat\eps$ and $l_\infty(y_\infty,t)\geq 1+\hat\eps$ for some
$t\in [-A^{-1},-1]$.  But 
this contradicts the formula $l_\infty(x,t)=1+\frac{z^2}{(-4t)}$.  
\end{proof}

\subsection{Stability of cylinders with moving basepoint}

In this subsection we prove a backward stability result for cylindrical
regions, initially without fixing the basepoint.  We then prove
a version fixing the basepoint, which will imply
Theorem \ref{thm_stability_of_necks}.

We start with a preliminary lemma about reduced volume.

\begin{lemma} \label{redvol2}
There are $\mu > 0$ and $\tau  >  0$ so that if $\M$ is a $\kappa$-solution defined
on $(- \infty, 0]$
with $R(x,0) = 1$ then $\tilde{V}(\tau) \le (4 \pi )^{\frac32} - \mu$.
\end{lemma}
\begin{proof}
Using the monotonicity of the reduced volume,
if the lemma fails then there is a sequence $\{\M^j\}_{j=1}^\infty$ of $\kappa$-solutions
and points $(x_j, 0) \in \M^j_0$ so that $R(x_j,0) = 1$ and
$ \tilde{V}_{\M^j}(j) \ge (4 \pi)^{\frac32} - \frac{1}{j}$.
After passing to a subsequence, there is pointed convergence to a pointed $\kappa$-solution
$\left( \M^\infty, (x_\infty, 0) \right)$. Then $R(x_\infty, 0) = 1$ and 
$\tilde{V}_\infty  (\tau) \ge (4 \pi)^{\frac32}$ for
all $ \tau > 0$. This is a contradiction \cite[Pf. of Proposition 39.1]{Kleiner-Lott_perelman_notes}.
\end{proof}

\begin{proposition}
\label{prop_cylinder_stability_1_noncompact}
For all $\hat\eps>0$ and $C<\infty$, 
there is a $T=T(\hat\eps,C)\in (-\infty,0)$ with the
following property.  Suppose that
 $\M$ is a noncompact $\kappa$-solution defined on 
$(-\infty,0]$  and:
\begin{itemize}
\item $(x,0)\in \M_0$ is a point with  $R(x,0)=1$.
\item $l:\M_{< 0}\ra \R$
is the $l$-function with spacetime basepoint $(x,0)$.
\item $t<T$, and $(y,t)\in \M_{t}$
is a point where the reduced distance
satisfies  $l(y,t)\leq C$.
\end{itemize}
Then one of the following holds:
\begin{enumerate}
\item  The tuple
$(\hat \M( -t),(y,-1),l)$ is $\hat\eps$-close to a triple
$(\cyl,(y_\infty,-1),l_\infty)$, where 
$y_\infty \in S^2 \times \R$
and $l_\infty$ is the asymptotic  $l$-function of
(\ref{eqn_l_infinity_cyl}).   Note that $y_\infty$
need not be at the minimum of $l_\infty$ in $\cyl_{-1}$.
\item $\M$ is isometric to a  $\Z_2$-quotient of a shrinking round cylinder.
\end{enumerate}
\end{proposition}
\begin{proof}
If the proposition were false then for some $\hat\eps>0$ there would be a sequence
$\{(\M^j,(x_j,0))\}$ of pointed $\kappa$-solutions, and a sequence
$T_j\ra -\infty$, such that
\begin{itemize}
\item $R(x_j,0)=1$.
\item There is a point $(y_j,T_j)\in \M^j_{T_j}$ with
$l(y_j,T_j)\leq C$, such that (1) and (2) fail
for $t = T_j$.
\end{itemize}

From Lemma \ref{redvol2}, there are constants $\mu > 0$ and $\tau  > 0$
so that for each $j$, we have
${\tilde V}_{\M^j}(\tau) \le (4 \pi)^{\frac32} - \mu$.
Using the monotonicity of the reduced volume and the existence of
the asymptotic soliton, there is a uniform positive lower bound
on ${\tilde V}_{\M^j}( - T_j)$.
After passing to a subsequence,
we can find $t_j  \in (T_j, 0)$
so that $\frac{T_j}{t_j}\ra \infty$, and the reduced volume
is  constant to within a factor $(1+j^{-1})$ on a time interval
$[A_jt_j,A^{-1}_jt_j]$ where $A_j\ra\infty$.

By Appendix \ref{appkappa}, for each $j$ there is a point
$(z_j, t_j) \in \M^j_{t_j}$ with $l(z_j, t_j) \le \frac32$. 
After passing to a subsequence,  by Proposition
\ref{prop_convergence_no_curvature_bound}, the sequence of
parabolically rescaled flows 
$\{(\hat \M^j(- t_j),(z_j,-1))\}$ pointed-converges 
to a gradient shrinking soliton $(\hat\S^\infty,(y_\infty,-1))$
that is either a $\kappa$-solution or is flat $\R^3$. Since
the reduced volume of $(\hat\S^\infty,(y_\infty,-1))$ is also
bounded above by $(4 \pi)^{\frac32} - \mu$, it cannot be flat $\R^3$.
In addition, $\hat\S^\infty$ cannot be a round spherical space form,
as each $\M^j_0$ is noncompact. Also,
$\hat\S^\infty$ cannot be a $\Z_2$-quotient of a shrinking 
round cylinder, because
then $\M^j_t$ would contain a one-sided $\R P^2$;
by the classification of noncompact $\kappa$-solutions
this would imply that $\M^j_t$ is isometric to the $\Z_2$-quotient of a 
round cylinder for all
$t$, contradicting the failure of (2).  Therefore $\hat\S^\infty$ must be a 
shrinking round cylinder.  

On the other hand, for each $j$, there is an asymptotic soliton of $\M^j$.
  By the previous argument involving spherical space forms and one-sided $\R P^2$'s,
  the asymptotic soliton must be a shrinking round cylinder.
It follows that for large $j$,
the reduced volume
${\tilde V}_{\M^j}(-t)$
is nearly constant in the interval
$(-\infty,t_j]$.
Proposition \ref{prop_convergence_no_curvature_bound} now implies that after passing to a subsequence,
$\{(\hat\M^j(- T_j),(y_j,-1),l_j)\}$ converges to a gradient shrinking soliton
$(\hat\M^{\infty},(y_\infty,-1),l_\infty)$,
which is a $\kappa$-solution and
whose asymptotic reduced volume is that of the 
shrinking round cylinder.  From Lemma \ref{lem_classification_gradient_soliton_kappa_solutions}, 
this implies
that $\hat\M^{\infty}$ is a shrinking round cylinder, contradicting the
failure of (1).
\end{proof}

We now use the stability result of Proposition
\ref{prop_cylinder_stability_1_noncompact}, together with Lemma
\ref{lem_cyl_l_function_stability}, to show that worldlines that
start close to necks have a nearly minimal value of $l$, provided one goes
at least a certain controlled amount backward in time.  

\begin{lemma}[Basepoint stability]
\label{lem_basepoint_stability}
For all $\hat\eps>0$, there exist $\de=\de(\hat\eps)>0$ and $T=T(\hat\eps)\in (-\infty,0)$
such that if $(\M,(x,0))$ is a pointed $\kappa$-solution with $R(x,0)=1$, 
and $(x,0)$ is a $\de$-neck, then $l(x,t)<1+\hat\eps$
for all $t<T$.  
\end{lemma}
\begin{proof}
Suppose the lemma were false.  Then for some $\hat\eps>0$,
there would be a sequence $\{( \M^j,(x_j,0))\}$
of pointed $\kappa$-solutions so that
\begin{enumerate}
\item $R(x_j,0)=1$ and 
\item $(x_j,0)$ is  a $\frac{1}{j}$-neck, but 
\item $l_j(x_j,\bar t_j)\geq 1+\hat\eps$ for some $\bar t_j\leq -j$.
\end{enumerate}
We can assume that $\hat\eps < 1$.
Let $\mu_1>0$ 
be a parameter to be determined later.

By Proposition \ref{prop_cylinder_stability_1_noncompact},
there is a $T_1\in (-\infty,0)$ such that for large $j$ and every
$(y,t)\in \M^j_{\leq T_1}$ with $l_j(y,t)<2$, 
 the tuple
$(\hat\M^j(-t),(y,-1),l_j)$ is $\mu_1$-close to $(\cyl,(y',-1),l_\infty)$, for
some $(y',-1)\in \cyl_{-1}$.

Since $(\M^j,(x_j,0))$ converges to the pointed round cylinder by assumption,
Lemma \ref{lem_cyl_l_function}(3) implies
there is a $T_2\in (-\infty,T_1]$ such that for large $j$, we
have $l_j(x_j,T_2)<1+\frac{\hat\eps}{4}$.

Put $t_j=\max\{t\in (-\infty,T_2] \: : \: l_j(x_j,t)\geq 1+\hat\eps\}$,
so $t_j \neq T_2$.
Note that there is an 
$A \in \left( 0, 1 \right)$ independent of $\mu_1$  
such that 
$\limsup_{j\ra\infty} \frac{T_2}{t_j} < A$
since
\begin{equation}
l_j(x_j,T_2)<1+\frac{\hat\eps}{4}
<1+\hat\eps=l_j(x_j,t_j)
\end{equation}
in view of the time-derivative bound on $l$
of (\ref{ye2.54}).

For large $j$, in 
$\hat\M^j(-t_j)$
we have 
$l_j(x_j,t)\leq 1+\hat\eps$ for $t\in [-1,-A]$, 
but $l_j(x_j,-1)=1+\hat\eps$.
Since $1+\hat\eps < 2$, we know 
from Proposition \ref{prop_cylinder_stability_1_noncompact} that
$(\hat\M^j(-t_j), (x_j, -1), l_j)$ is $\mu_1$-close to $(\cyl, (y^\prime, -1),
l_\infty)$ for some $(y^\prime, -1) \in \cyl_{-1}$,
again for large $j$.
If $\bar\mu$ is the constant from Lemma \ref{lem_cyl_l_function_stability},
and $\mu_1<\bar\mu$, then that lemma gives a contradiction.
\end{proof}

\begin{proof}[Proof of Theorem \ref{thm_stability_of_necks}]
This follows by combining 
Proposition \ref{prop_cylinder_stability_1_noncompact} 
with Lemma \ref{lem_basepoint_stability}.

\end{proof}

\section{Finiteness of points with bad worldlines}
\label{sec_finiteness}

In this section we study 
bad worldlines; recall from Definition 
\ref{def_worldline}  that a worldline $\ga:I\ra\M$ in a singular
Ricci flow $\M$  is bad if $\inf I>0$. 
In Theorem \ref{finitenessthm} we prove that only finitely many bad worldlines
intersect a given connected component in a given time slice.
We then give some applications.

The main result of this section is the following theorem.

\begin{theorem} \label{finitenessthm}
Let $\M$ be a singular Ricci flow.
For $T<\infty$,  every connected component $C_T$ of $\M_T$
intersects at most $N$ bad worldlines, where $N=N(T,\vol(\M_0))$.
In particular, the set of bad worldlines in $\M$ is at most countable.
Moreover if $\ga:I\ra\M$ is a bad worldline,
then for $t\in I$ sufficiently close to $\inf I$, $\ga(t)$ lies in a 
cap region of $\M_t$.
\end{theorem}

We begin with a lemma.

\begin{lemma}
\label{lem_controlled_diameter}
For all $D<\infty$ there exist $\hat \eps=\hat\eps(\kappa,D)>0$ and 
$A=A(\kappa,D)<\infty$ such that if $m\in \M_t$ and $(\M,m)$ is
 $\hat\eps$-modelled 
(see Appendix \ref{appcloseness}) 
on  a $\kappa$-solution
of diameter $\leq D$, then:
\begin{enumerate}
\item The connected component $N_t$ of $\M_t$ containing $m$ is compact and has
$\Rm>0$.
\item Let $g'(\cdot)$ be the Ricci flow on the 
(smooth manifold underlying)
$N_t$ defined on the time interval $[t,T')$, with initial condition at time
$t$ given by 
$g'(t)=g\restr_{N_t}$, and   blow-up time $T'$.  Let $\n$ be the corresponding
Ricci flow spacetime.  Then  
$\n_{\leq t'}$ is compact for all $t'<T'$, and 
if $\tilde\M$ is the connected component of $\M_{\geq t}$ containing
$N_t$, then $\tilde\M$ is isomorphic to $\n$.
\item $R\geq AR(m)$ on $\tilde\M$.
\end{enumerate}
\end{lemma}
\begin{proof}
(1).  Let $\kapsold$ be the collection of pointed $\kappa$-solutions
$(\M',(x,0))$ such that $R(x,0)=1$, and $\diam(\M'_0)\leq D$.
Then   every $(\M',(x,0))\in \kapsold$ has
$\Rm>0$, and since  $\kapsold$ is compact, there is a $\la>0$
such that $\Rm\geq \la$ 
in $\M'_0$ for all $(\M',(x,0))\in \kapsold$.
Part (1) now follows.

(2). Let $\n$ be as above.  Then $\n_{[t,t']}$ is compact for all $t'\in [t,T')$,
 and
by Hamilton's theorem  for manifolds with $\ric>0$, we know that
\begin{equation}
\label{eqn_r_blows_up}
\min\{R(m') \: : \: m'\in \n_{t'}\}\ra \infty\quad\text{as}\quad t'\ra T'\,.
\end{equation}

Consider an isometric embedding of Ricci flow spacetimes
$\n_{J}\hookrightarrow \M_{J}$
 that extends the isometric embedding
$\n_t\ra \M_t$, and which is defined on a maximal time interval
$J$ starting at time $t$.   Then $J$ cannot be a closed interval
$[t,\hat t]$, since then the embedding could extended to a larger time
interval using uniqueness for Ricci flows.
If $J=[t,\hat t)$ with $\hat t<T'$, then since $R$ is bounded on 
$\n_{[t,\hat t]}$, by the properness of $R$ on $\M_{\leq T'}$ we may 
extend the embedding to an isometric
embedding $\n_{[t,\hat t]}\ra \M_{[t,\hat t]}$, contradicting the maximality
of $J$.  Therefore there exists an isometric embedding
$\n_{[t,T')}\ra \M_{[t,T')}$ of Ricci flow spacetimes, as asserted.
The image is clearly an open subset of $\M_{\geq t}$; 
it is closed by (\ref{eqn_r_blows_up}).  This proves (2).  

(3).  From the proof of (1) above, 
for every $(\M',(x,0))\in \kapsold$, we have $R\geq 6\la$ on $\M'_0$.  Taking
$A=3\la$, and $\hat\eps$ sufficiently small, we get that $R\geq AR(m)$
in $N_t$, and therefore in $\n_t$ as well.  
By the maximum principle applied to the scalar curvature evolution equation,
we have $R\geq AR(m)$ on $\n$, and hence on $\tilde\M$ as well.
\end{proof}

\begin{proof}[Proof of Theorem \ref{finitenessthm}]
  In the proof below, we take $\kappa=\kappa(T)$.
  Recall that $\epsilon$ is the global parameter
    used in the definition of a Ricci flow spacetime
    from Definition \ref{def_singular_ricci_flow}.

Let $\eps_1, \De>0$ be constants, to be determined later. 
During the course of the argument below, we will state a  number of 
inequalities involving  $\eps_1$ and $\De$; these will be treated as
a cumulative set of constraints imposed on $\eps_1$ and $\De$, i.e.
we will be assuming that each inequality is satisfied. 

We recall that by Proposition 
\ref{sameconnected}, $C_T$ determines a connected component
$C_t$ of $\M_t$ for all $t\leq T$.  Let $\bad$ be the collection of bad 
worldlines intersecting $C_T$.

Choose $0<t_-<t_+\leq T$ such that $t_+-t_-<\De$, and let
$\bad_{[t_-,t_+)}$   be the set of
worldlines
$\ga:I\ra \M$ which belong to $\bad$, where $\inf I\in [t_-,t_+)$. 
 We will show that
if $t_+-t_-<\De=\De(\kappa,T,\eps)$, then $|\bad_{[t_-,t_+)}|$ is bounded by 
a function of $\kappa$, $T$ 
and $\vol(\M_0)$; the theorem then follows immediately.

\bigskip
\noindent
{\em Step 1. If $\De<\bar\De(\eps_1,\kappa,\eps,T)$, then
there exists $\hat T\in [0,T]$ such that: 
\begin{enumerate}
\renewcommand{\theenumi}{\alph{enumi}}
\item 
For every 
worldline
$\ga:I\ra\M$ in $\bad$ we have $\inf I< \hat T$.
\item 
For every 
worldline
$\ga:I\ra\M$ in $\bad_{[t_-,t_+)}$,
and every  
$t\in (\inf I,t_+]$
with $t\leq \hat T$, the pair  
$(\M,\ga(t))$
is $\eps_1$-modelled on   a noncompact $\kappa$-solution.
\end{enumerate} 
}

By the compactness of the space of 
normalized pointed $\kappa$-solutions (see
Appendix \ref{appkappa})
there exist $\eps_2=\eps_2(\eps_1) > 0$ and $D=D(\eps_1) < \infty$ 
such that if 
$(\M,\ga(t))$ is $\eps_2$-modelled on 
a pointed $\kappa$-solution
with diameter greater than $D$, then it
is $\frac12 \eps_1$-modelled on  a noncompact $\kappa$-solution.  Put 
$\eps_3= \frac12  \min(\hat \eps(\kappa,D),\eps_2)$, where $\hat\eps(\kappa,D)$
is as in Lemma \ref{lem_controlled_diameter}.

Let $\W$
be the set of $m\in \bigcup_{t\leq T}\,C_t$
such that 
the pair
$(\M,m)$ is $\eps_3$-modelled on 
a pointed $\kappa$-solution with diameter at most $D$.
Suppose first that $\W$ is nonempty.
Lemma \ref{lem_controlled_diameter} implies that $R(m)\leq A^{-1}
\inf_{C_T}R$, for all $m\in \W$; therefore
by the properness of $R:\M_{\leq T}\ra \R$,
the time function $\t$ attains a minimum value $\hat T$ on $\W$.  
Pick $m\in \W\cap \M_{\hat T}$.
Then by Lemma \ref{lem_controlled_diameter},
the connected component of $\M_{[\hat T,T]}$ containing $m$
is isomorphic to the $[\hat T,T]$-time slab
of the spacetime $\n$ of a Ricci flow on a compact manifold
with positive Ricci curvature.  In particular, it also
coincides with $\bigcup_{t\in [\hat T,T]}C_t$, and therefore
the curvature is bounded on the latter.  Hence
 for every
 $\ga:I\ra \M$
in  $\bad$, we have $\inf I<\hat T$.
If $\W$ is empty then we put $\hat T = T$; then
the conclusion of part (a) of Step 1 still holds, so we continue.

By Proposition \ref{goodeps},
 there exists $\hat R=\hat R(\eps_3,\kappa,r(T))$, such that 
for all $m\in \M_{\leq T}$ with $R(m)\geq \hat R$, the pair
$(\M,m)$ is $\eps_3$-modelled on a pointed
$\kappa$-solution.  By Lemma \ref{singforback}, 
if
$\De<\bar\De(\hat R,\kappa)$, 
 $\ga:I\ra\M$
belongs to  $\bad_{[t_-,t_+)}$ and 
$t\in [t_-,t_+]\cap I\cap [0,\hat T)$, then
we have $R(\ga(t))>\hat R$.    Hence 
either $(M,\ga(t))$
is (A) $\frac12 \eps_1$-modelled on  a noncompact $\kappa$-solution, or (B)
$\eps_3$-modelled on 
 a $\kappa$-solution of diameter at most $D$; but in case (B) we would
have $\ga(t)\in \W$, which is impossible 
because $t<\hat T$.  This proves that part (b) of Step 1 holds when
$t < \hat T$,
with $\epsilon_1$ replaced by $\frac12 \epsilon_1$ in the statement.
The borderline case $t=\hat T$ now follows by 
applying the previous arguments to times $t$ slightly less than $\hat T$
and taking the limit as $t\nearrow \hat T$.  This completes Step 1.

\bigskip

Hereafter we assume that $\De<\bar\De(\eps_1,\kappa,\eps,T)$.
By part (a) of Step 1, the set $\bad$ is the same as the set of bad worldlines
intersecting $C_{\hat T}$.  Hence we may replace $T$ by $\hat T$;
then by part (b) of Step 1, 
for every $\ga:I\ra\M$ in $\bad_{[t_-,t_+)}$,
and every  $t\in [t_-,t_+]$, the pair  
$(\M,\ga(t))$
is $\eps_1$-modelled on a noncompact $\kappa$-solution.

\bigskip
\noindent
{\em Step 2. Provided that 
$\eps_1 < \eps_1(\kappa)$,
for all $\ga:I\ra \M$
belonging to $\bad_{[t_-,t_+)}$ and every $t\in I\cap [t_-,t_+)$, 
the pair $(\M,\ga(t))$ is not  a generalized
$\frac{\de_{neck}}{2}$-neck.
(Here $\de_{neck}$ is the parameter from Corollary \ref{cor_de_de_over_4}, 
and a generalized neck is in the sense of
Appendix \ref{appcloseness}.)}

Suppose that $\ga:I\ra \M$ belongs to $\bad_{[t_-,t_+)}$, and 
$(\M,\ga(\hat t_0))$ is a generalized $\frac{\de_{neck}}{2}$-neck
for some 
$\hat t_0\in I\cap[t_-,t_+)$.
By Step 1 we know that $(\M,\ga(\hat t_0))$ is $\eps_1$-modelled
on a noncompact pointed $\kappa$-solution $(\M^1,(x_0,0))$.
If $\eps_1<\bar\eps_1(\frac{\de_{neck}}{2})$, 
then $(\M^1,(x_0,0))$ will be a 
generalized
$\de_{neck}$-neck.
Let $T_1=T(\de_{neck},\frac{\de_{neck}}{4})\in (-\infty,0)$ be as in 
Corollary \ref{cor_de_de_over_4}.   Then $(\M^1,(x_0,T_1))$
is a generalized $\frac{\de_{neck}}{4}$-neck.  If  
$\eps_1<\bar\eps_1(T_1,\frac{\de_{neck}}{2})$, then we get that:
\begin{itemize}
\item $\ga$ is 
defined at $\hat t_1=\hat t_0+R^{-1}(\ga(\hat t_0))T_1$.
\item $(\M,\ga(\hat t_1))$ is a generalized $\frac{\de_{neck}}{2}$-neck.
\item  $R(\ga(\hat t_1))<\frac12 R(\ga(\hat t_0))$.  
\end{itemize}
Thus we may iterate
this to produce a sequence $\{\hat t_0,\hat t_1,\ldots\}\subset I$ 
such that 
$\hat t_i \le \hat t_{i-1}+R^{-1}(\ga(\hat t_0))T_1$
for all $i$.  This contradicts the fact that $\inf I\in [t_-,t_+)$,
and completes Step 2.

\bigskip

Hereafter we assume that
$\eps_1 < \eps_1(\kappa)$.
Let $D_0<\infty$ be such that if $\M'$ is a noncompact 
$\kappa$-solution, $m_1,m_2\in \M'_t$, and neither $m_1$ nor $m_2$
is a $\frac{\de_{neck}}{4}$-neck, then 
\begin{equation}
\label{eqn_non_necks_close}
d_t(m_1,m_2)<D_0
R(m_1)^{-\frac12}\,.
\end{equation}

Let $D_1\in (2D_0,\infty)$ be a constant, to be determined 
in Step 4.

\bigskip
\noindent
{\em Step 3. Provided that 
$\eps_1<\bar\eps_1^\prime(\kappa,D_1)$,
if $\ga_1,\ga_2\in \bad_{[t_-,t_+)}$, and 
\begin{equation}
\label{eqn_d_1}
d_{t_+}(\ga_1(t_+),\ga_2(t_+))<D_1
R(\ga_1(t_+)))^{-\frac12}\,,
\end{equation}
then $\ga_1=\ga_2$.}

Let $I_i$ be the domain of $\gamma_i$, for $i \in \{1,2\}$.
Suppose that $t\in I_1\cap I_2\cap [t_-,t_+]$.
By Steps 1 and 2,
$(\M,\ga_1(t))$ is $\eps_1$-modelled on a
noncompact $\kappa$-solution and neither
$\ga_1(t)$ nor $\ga_2(t)$ is a  $\frac{\de_{neck}}{2}$-neck.
If $\eps_1<\bar\eps_1^\prime(D_1,\de_{neck})$ then using the
$\eps_1$-closeness to a noncompact $\kappa$-solution and  
(\ref{eqn_non_necks_close}), we can say that 
\begin{equation}
\label{eqn_3_d0}
d_{t}(\ga_1(t),\ga_2(t))<D_1
R(\ga_1(t)))^{-\frac12}
\,,
\end{equation}
implies
\begin{equation}
d_{t}(\ga_1(t),\ga_2(t))<2D_0
R(\ga_1(t)))^{-\frac12}
\end{equation}

Since we are assuming (\ref{eqn_d_1}), a continuity argument shows
that (\ref{eqn_3_d0}) holds for
all $t\in I_1\cap I_2\cap [t_+,t_-]$. 
 If $\inf I_1\geq \inf I_2$, then
 $\lim_{t\ra \inf I_1}R(\ga_2(t))=\infty$, so $\inf I_1=\inf I_2$;
similar reasoning holds if $\inf I_2\leq \inf I_1$.
Thus $\inf I_1=\inf I_2$.
Moreover, if 
$\eps_1<\bar \eps_1^{\prime \prime}(\kappa)$
then any geodesic from 
$\ga_1(t)$ to $\ga_2(t)$ in $C_t$ will lie in the set with $\ric > 0$, 
so $d_t(\ga_1(t),\ga_2(t))$ is a decreasing function of $t$.  Since
(\ref{eqn_3_d0}) implies that
$d_t(\ga_1(t),\ga_2(t))\ra 0$ as $t\ra \inf I_1$, it follows that 
$\ga_1=\ga_2$.  This completes Step 3.

\bigskip

Hereafter we assume that $\eps_1<\bar\eps_1^\prime(\kappa,D_1)$.

\bigskip
\noindent
{\em Step 4.  Provided that $\De<\bar\De(\kappa,T)$, 
the cardinality of $\bad_{[t_-,t_+)}$ is at most
$N=N(\kappa,T,\vol(\M_0))$. }

Take $\hat\eps=\frac{\de_{neck}}{2}$, and let 
$C_1=C_1(\hat\eps,T)$, $\ol{R}=\ol{R}(\hat\eps,T)$, and
$N_1,\ldots,N_k\subset \M_{t_+}$ be as in 
Proposition \ref{prop_large_r_structure}.
With reference to Step 3, 
take $D_1=C_1$. Let $\De<\ol{\De}(\ol{R},\kappa,r(T))$ be such that
if $\ga:I\ra \M$ belongs to $\bad_{[t_-,t_+)}$, and 
$t\in [t_-,t_+]\cap I$, then $R(\ga(t))\geq \ol{R}$;
c.f. the proof of Step 1.

Then the set 
\begin{equation}
S=\{\ga(t_+) \: : \: \ga\in \bad_{[t_-,t_+)}\}
\end{equation}
is contained in $\{m \in \M_{t_+} : R(m) \geq \ol{R} \}
\subset \bigcup_i N_i$.
By Step 3, for any two distinct elements $\ga_1,\ga_2\in \bad_{[t_-,t_+)}$
we have 
\begin{equation}
d_t(\ga_1(t_+),\ga_2(t_+))\geq D_1R^{-\frac12}(\ga_1(t_+))
=C_1R^{-\frac12}(\ga_1(t_+))\,.
\end{equation}
By Proposition \ref{prop_large_r_structure}, for all $i\in \{1,\ldots,k\}$
we have $|S\cap N_i|\leq 2$.   Therefore $|S|\leq 2k$.  This proves
that $\bad_{[t_-,t_+)}$ is finite,  and hence a weaker version of 
the theorem, namely that the set of all
bad worldlines is countable.    

Since the set of bad worldlines is countable, their union has measure zero
in spacetime.  Therefore we may apply Proposition \ref{vol1},
to conclude that 
\begin{equation}
\vol(C_{t+})\leq \vol(\M_{t+})
\leq\V(0) \left(1 + 2 t \right)^{\frac{3}{2}}.
\end{equation}
If $k\geq 2$, then each $N_i$ has nonempty boundary.
Hence part (5) of Proposition \ref{prop_large_r_structure} gives a bound
$k<k(\V(0),T)$.

This proves Theorem \ref{finitenessthm}; the assertion 
that bad worldlines
are confined to cap regions was established in the proof above.
\end{proof}

\begin{corollary}
\label{cor_volume_continuous}
If $\M$ is a singular Ricci flow then the 
conclusions of Proposition \ref{vol1} hold.
\end{corollary}
\begin{proof}
By Theorem \ref{finitenessthm} the set of bad worldlines is countable, and hence
has measure zero.  Combining this with  Lemma \ref{quasip} shows that the
hypotheses of 
Proposition \ref{vol1} are satisfied.
\end{proof}

\begin{remark}
For mean convex mean curvature flow, the continuity of the volume follows from the
weak continuity of the mass measure, as established in 
\cite{metzger_schulze}.
\end{remark}

Theorem \ref{finitenessthm} also has the following topological implications.
\begin{corollary}
Let $\M$ be a singular Ricci flow.
\begin{enumerate}
\item If $T\geq 0$ and $W\subset\M_T$ is an open subset that does not contain
any compact connected components of $\M_T$, then there is a smooth 
time-preserving map $\Ga:W\times [0,T]\ra\M$ that is a ``weak isotopy'', in 
the sense
that it maps $W\times \{t\}$ diffeomorphically onto an open subset of 
$\M_{t}$, for all $t\in[0,T]$.
\item For all $T\geq 0$, the pair $(\M,\M_{\leq T})$ is $k$-connected
for $k\leq 2$.
\end{enumerate}
\end{corollary}
\begin{proof}
(1).  Let $\C$ be the collection of connected components of $\M_T$.  
Pick $C\in\C$.
Let $B$ be the set of bad worldlines intersecting $C$.   
 By Theorem \ref{finitenessthm} the set $B$
 is  finite, so its intersection with $C$
 is contained in a $3$-disk $D^3$.
There  is a $t_C<T$ such that the worldline
of every $m\in C$ is defined in the interval $[t_C,T]$. Hence we get
a time-preserving map $F_C:C\times [t_C,T]\ra \M$ that is a diffeomorphism onto
its image. 

By assumption, either $C$ is noncompact, or $C$ is compact and 
$W\not\supset C$.   Therefore there is a
smooth homotopy $\{H_t:W\cap C\ra C\}_{t\in [t_C,T]}$ 
(purely in the time-$T$ slice)
such that $H_T:W\cap C\ra C$
is the inclusion map, 
$H_{t}:W\cap C\ra C$ is a diffeomorphism onto its image for all 
$t\in[t_C,T]$, and $H_{t_C}(W\cap C)\cap D^3=\emptyset$.
 We define $\Ga$ on $(W\cap C)\times [t_C,T]$ by 
$\Ga(m,t)=F_C(H_t(m),t)$, and extend this to 
$(W\cap C)\times [0,t_C]$ by following
worldlines.  Note that if $C_1,C_2\in\C$ are distinct components of $\M_T$
then $F_{C_1}(C_1\cap W)$ is disjoint from $F_{C_2}(C_2\cap W)$,
so the resulting map $\Ga$ has the property that $\Ga(\cdot,t):W\ra \M_t$
is an injective local diffeomorphism for every $t\in [0,T]$.

(2). Suppose that $0\leq k\leq 2$, and $f:(D^k,\D D^k)\ra (\M,\M_{\leq T})$ 
is a map 
of pairs, where $\D D^k=S^{k-1}$ if $k\geq 1$ and $\D D^0=\emptyset$.  
By Theorem \ref{finitenessthm}, 
the Hausdorff dimension of the bad worldlines is at most one.
Then after making a small homotopy we may assume that $f$ is smooth,
and that its image is disjoint from the bad worldlines.  We can now find
a homotopy through maps of pairs by using the backward flow of the time
vector field, which is well-defined on $f(D^k)$.
\end{proof}

\appendix
\section{Background material}
\label{app_background_material}

In this appendix we collect some needed facts about Ricci flows
and Ricci flows with surgery.  More information can be found in
\cite{Kleiner-Lott_perelman_notes}.

\subsection{Notation and terminology}
\label{sec_notation}

Let $(\M, \t, \partial_{\t}, g)$ be a Ricci flow spacetime 
(Definition \ref{def_ricci_flow_spacetime}).
For brevity, we will often write $\M$ for the quadruple.
In a Ricci flow with surgery, we will sometimes loosely write a point
$m \in \M_t$ as a pair $(x,t)$.

Given $s > 0$, the rescaled
Ricci flow spacetime is
$\hat \M(s) = (\M, \frac{1}{s} \t, s \partial_{\t}, \frac{1}{s} g)$.

Given $m \in M_t$, we write $B(m, r)$ for the open metric
ball of radius $r$ in $\M_t$. We write $P(m,r,\Delta t)$ for the
parabolic neighborhood, i.e. the set of points
$m^\prime$ in $\M_{[t, t + \Delta t]}$ if $\Delta t > 0$
(or $\M_{[t + \Delta t, t]}$ if $\Delta t < 0$) that lie on the
worldline of some point in $B(m,r)$.
We say that $P(m,r, \Delta t)$ is {\em unscathed} if 
$B(m,r)$ has compact closure in $\M_t$ and
for every
$m^\prime \in P(m,r, \Delta t)$, the maximal worldline $\gamma$ through
$m^\prime$ is defined on a time interval containing $[t, t+\Delta t]$
(or $[t+\Delta t, t]$). 
We write $P_+(m,r)$ for the
forward parabolic ball $P(m,r,r^2)$ and 
$P_-(m,r)$ for the
backward parabolic ball $P(m,r,-r^2)$. 

We write $\cyl$ for the standard Ricci flow on $S^2 \times \R$ that
terminates at time zero, with $g(t) = (-2t) g_{S^2} + dz^2$.
We write $\sphere$ for the standard round shrinking $3$-sphere
that terminates at time zero.

\subsection{Closeness of Ricci flow spacetimes} \label{appcloseness}

Let $\M^1$ and $\M^2$ be two Ricci flow spacetimes in the sense of
Definition \ref{def_ricci_flow_spacetime}.
Consider a time interval 
$[a,b]$. Suppose that $m_1 \in \M^1$ and
$m_2 \in \M^2$ have $\t_1(m_1) = \t_2(m_2) = b$. We say that
$(\M^2, m_2)$ is {\em $\epsilon$-close} to $(\M^1, m_1)$
on the
time interval $[a,b]$ if 
there are open subsets $U_i \subset \M^i$ with
$P(m_i, \epsilon^{-1}, a-b) \subset U_i$, $i \in \{1,2\}$, and
there is a pointed diffeomorphism 
$\Phi : (U_1, m_1) \ra (U_2, m_2)$ so that
\begin{itemize}
\item $B(m_i, \epsilon^{-1})$ has compact closure in $\M^i_b$,
\item $P(m_i, \epsilon^{-1}, a-b)$ is unscathed,
\item $\Phi$ is time-preserving, i.e. $\t_2 \circ \Phi = \t_1$,
\item $\Phi_* \partial_{{\t}_1} = \partial_{{\t}_2}$ 
and
\item  
$\Phi^* g_2 - g_1$ has norm less than $\epsilon$ in the
$C^{[1/\epsilon]+1}$-topology (as defined using $g_1$) on $U_1$.
\end{itemize}

\begin{remark}
  The notion of $\epsilon$-closeness is not symmetric with respect to
  $(\M^1, m_1)$ and $(\M^2, m_2)$, but this will not be an issue
  since we only use the associated topology.
  \end{remark}
  
Now consider an open time interval 
$(- \infty,  b)$.
Suppose that $\t_1(m_1) = \t_2(m_2) = c \in (- \infty, b)$.
After time shift and  parabolic rescaling, we can assume that
$c = -1$ and $b = 0$.
In this case, we say that 
$(\M^2, m_2)$ is $\epsilon$-close to $(\M^1, m_1)$
on the
time interval $(- \infty, 0)$ if 
there are open sets
$U_i \subset \M^i$ with
$(P(m_i, \epsilon^{-1}, 1-\epsilon) \cup P(m_i, \epsilon^{-1},
-\epsilon^{-2})) \subset U_i$,
$i \in \{1,2\}$, and there is a pointed diffeomorphism
$\Phi : (U_1, m_1) \rightarrow (U_2, m_2)$ so that
\begin{enumerate}
\item $B(m_i, \epsilon^{-1})$ has compact closure in
$\M^i_{-1}$,
\item $P(m_i, \epsilon^{-1}, 1-\epsilon)$ and $P(m_i, \epsilon^{-1},
-\epsilon^{-2})$ are unscathed,
\item $\Phi$ is time-preserving, i.e. $\t_2 \circ \Phi = \t_1$,
\item $\Phi_* \partial_{{\t}_1} = \partial_{{\t}_2}$ and
\item 
$\Phi^* g_2 - g_1$ has norm less than $\epsilon$ in the
$C^{[1/\epsilon]+1}$-topology on $U_1$.
\end{enumerate}

Suppose $(\M_1,m_1)$, $(\M_2,m_2)$ are as above, and 
$u_i:\M_i\ra \R$ is a continuous
function for $i\in \{1,2\}$.
Then we say that 
$(\M_2,m_2,u_2)$ is $\eps$-close to $(\M_1,m_1,u_1)$
if 
(1)-(5) above hold, and in addition 
$$
\sup\{u_2\circ \Phi(m)-u_1(m)\mid
m\in (P(m_1, \epsilon^{-1}, 1-\epsilon) \cup P(m_1, \epsilon^{-1},
  -\epsilon^{-2}))
\}<\eps\,.
$$
We will apply this notion when the $u_i$'s are reduced distance functions.

We define $\epsilon$-closeness similarly on other time intervals, whether
open or half-open.

If $(\M_1, m_1)$ and $(\M_2, m_2)$ are Ricci flow spacetimes then
we say that
$(\M_2,m_2)$ is {\em $\eps$-modelled} on $(\M_1,m_1)$
if after shifts in the time parameters so that $\t_1(m_1) = \t_2(m_2) = 0$,
and parabolic rescaling by $R(m_1)$ and $R(m_2)$ respectively,
the resulting Ricci flow spacetimes are $\eps$-close
on the time interval $[- \epsilon^{-1}, 0]$.
(It is implicit in the definition that $R(m_1) > 0$ and $R(m_2) > 0$;
this will be the case for us since we are interested in modelling
regions of high scalar curvature.)
A point $m$ in a Ricci flow spacetime $\M$ is a
{\em generalized $\eps$-neck} if $(\M,m)$ is $\eps$-modelled on 
$(\M',m')$, where $\M'$ is either a shrinking round cylinder or the
$\Z_2$-quotient of a shrinking round cylinder.

\subsection{Necks, horns and caps}

We say that $(\M, m)$ is a (strong) {\em $\delta$-neck}
if after time shifting and parabolic rescaling, it is $\delta$-close
on the time interval $[-1,0]$ 
to the product Ricci flow which,
at its final time, is isometric to the product of $\R$ with 
a round $2$-sphere of scalar curvature one. The basepoint is taken at
time $0$. 
In this case, we also say that $m$ is the 
{\em center of a $\delta$-neck}.

If $I$ is an open interval then a metric on an embedded copy of 
$S^2 \times I$ in $\M_t$,
such that each
point is contained in an $\delta$-neck, is called a {\em $\delta$-tube},
or a {\em $\delta$-horn}, or a 
{\em double $\delta$-horn}, if the scalar
curvature stays bounded on both ends, stays bounded on one end and
tends to infinity on the other, or tends to infinity on both ends,
respectively.  (Our definition differs slightly from that in
\cite[Definition 58.2]{Kleiner-Lott_perelman_notes}, where the definition
is in terms of the ``$\delta$-necks'' of that paper, as opposed
to the ``strong $\delta$-necks'' that we are using now.)

A metric on $B^3$ or $B^3 - \R P^3$, such that each point outside some
compact set is contained in a $\delta$-neck, is called a
{\em $\delta$-cap} or a {\em capped $\delta$-horn}, if the scalar
curvature stays bounded or tends to infinity on the end, respectively.

\subsection{$\kappa$-noncollapsing} \label{kappa}

Let $\M$ be an $(n + 1)$-dimensional Ricci flow spacetime.
Let $\kappa : [0,\infty) \ra (0, \infty)$ be a
decreasing function.
We say that $\M$ is {\em $\kappa$-noncollapsed at scales below
$\epsilon$} if for each $\rho < \epsilon$ and all $m \in \M$ with
$\t(m) \ge \rho^2$, whenever $P(m, \rho, - \rho^2)$ is unscathed
and $|\Rm| \le \rho^{-2}$ on $P(m, \rho, - \rho^2)$, then we also
have $\vol(B(m, \rho)) \ge \kappa(\t(m)) \rho^n$.
In the application to Ricci flow with surgery, $\epsilon$ will be
taken to be the global parameter.

We refer to \cite[Section 15]{Kleiner-Lott_perelman_notes} for the definitions of the
$l$-function $l(m)$ and the reduced volume $\tilde V(\tau)$.
For notation, we recall that the $l$-function is defined in terms of
${\mathcal L}$-geodesics going backward in time from a
basepoint $m^\prime \in \M$. The parameter $\tau$ is backward time
from $m^\prime$; e.g. $\tau(m) = \t(m^\prime) - \t(m)$. 

\subsection{$\kappa$-solutions} \label{appkappa}

Given $\kappa \in \R^+$,
a {\em $\kappa$-solution} $\M$ is a smooth 
Ricci flow solution defined on a time
interval of the form $(- \infty, C)$ (or $(- \infty, C]$) such that
\begin{itemize}
\item The curvature is uniformly bounded on each compact time interval, and
each time slice is complete.
\item The curvature operator is nonnegative and the scalar curvature is
everywhere positive.
\item The Ricci flow is $\kappa$-noncollapsed at all scales.
\end{itemize}

We will sometimes talk about $\kappa$-solutions without specifying $\kappa$.
Unless other specified, it is understood that $C = 0$.
If $(\M, m)$ is a pointed $\kappa$-solution then we will sometimes understand
it to be defined on the interval $(- \infty, \t(m)]$.

Examples of $\kappa$-solutions are $\cyl$ and $\sphere$.

Any pointed $\kappa$-solution $(\M, m)$ has an {\em asymptotic soliton}.
It is obtained by constructing the $l$-function using 
${\mathcal L}$-geodesics emanating backward from $m$. For any 
$t < \t(m)$, there is some point $m^\prime_t \in \M_t$ where
$l(m) \le \frac{n}{2}$. Put $\tau = \t(m) - t$.
Then the parabolic rescaling 
$\left( \hat \M(\tau), m^\prime_t \right)$
subconverges as $\tau \rightarrow \infty$ to a nonflat gradient
shrinking soliton called the asymptotic soliton
\cite[Proposition 39.1]{Kleiner-Lott_perelman_notes}.
(In the cited reference, the convergence is shown on the 
(rescaled) time interval
$\left[ -1, - \frac12 \right]$, but using the estimates of
Subsection \ref{lestimates} one easily gets pointed convergence on
the time interval $(- \infty, 0)$.)

Hereafter, we suppose that the spacetime $\M$ of the
$\kappa$-solution is four-dimensional.
A basic fact is that the space of pointed
$\kappa$-solutions $(\M, m)$, with $R(m) = 1$, is compact
\cite[Theorem 46.1]{Kleiner-Lott_perelman_notes}.

Given $\delta > 0$, let $\M_\delta$ denote the points in $\M$ that are
not centers of $\delta$-necks.
We call these {\em cap points}. Put $\M_{t, \delta} = \M_t \cap \M_\delta$. 
From \cite[Corollary 47.2]{Kleiner-Lott_perelman_notes}, if $\delta$ is small enough
then  
there is a $C = C(\delta, \kappa) > 0$ such that
if 
$\M_t$
is noncompact then
\begin{itemize} 
\item $\M_{t,\delta}$ is compact with 
$\diam(\M_{t,\delta}) \le C Q^{-\frac12}$ and
\item $C^{-1} Q \le R(m) \le C Q$ whenever $m \in M_{t,\delta}$,
\end{itemize} 
where $Q = R(m^\prime)$ for some $m^\prime \in \partial M_{t, \delta}$.

If $\M$ is noncompact, and not a round shrinking cylinder, then 
$\M_{t, \delta} \neq \emptyset$. A version of the preceding paragraph
that also holds for compact $\kappa$-solutions can be found in
\cite[Corollary 48.1]{Kleiner-Lott_perelman_notes}.

A compact $\kappa$-solution is either a quotient of the round shrinking
sphere, or is diffeomorphic to $S^3$ or $\R P^3$
\cite[Lemma 59.3]{Kleiner-Lott_perelman_notes}.

There is some $\kappa_0 > 0$ so that any $\kappa$-solution
is a $\kappa_0$-solution or a quotient of the round shrinking $S^3$
\cite[Proposition 50.1]{Kleiner-Lott_perelman_notes}

\subsection{Gradient shrinking solitons}
\label{subsec_gradient_shrinking_solitons}
\begin{lemma}
\label{lem_classification_gradient_soliton_kappa_solutions}
Let $\M$ be a three-dimensional gradient shrinking
soliton that is a $\kappa$-solution and  blows up as $t\ra 0$.
For $t < 0$ and a point $(y,t) \in \M$, 
let $l_{y,t} \in C^\infty(\M_{< t})$ be the 
$l$-function on $\M$
constructed using ${\mathcal L}$-geodesics going backward in time from 
$(y,t)$.
Then there is a function $l_\infty \in C^\infty(\M)$ so that
the limit $\lim_{t \rightarrow 0^-} l_{y,t} = l_\infty$ exists,
independent of $y$, with continuous convergence on compact subsets on $\M$.
Define
$\tilde V_\infty:(-\infty,0)\ra (0,\infty)$ as in
(\ref{redvol}).
Then one of the following holds:
\begin{enumerate}
\item $\M$ is the shrinking round cylinder solution $\cyl$ on 
$S^2 \times \R$, with $R(x,t)=(-t)^{-1}$, 
$l_\infty((x,z),t)=1+\frac{z^2}{(-4t)}$ and
$\tilde V_\infty(t) = \frac{16 \pi^{\frac32}}{e}$.
\item $\M$ is the $\Z_2$-quotient of the 
cylinder in (2), with $\tilde V_\infty(t) = \frac{8 \pi^{\frac32}}{e}$.
The pullback of $l_\infty$ to the cylinder is 
$1+\frac{z^2}{(-4t)}$.
\item $\M$ is a shrinking round spherical space form $\sphere/\Ga$, where
$\Ga\subset \SO(4)$, $R(x,t)=\frac{3}{2}(-t)^{-1}$,  $l_\infty(x,t)= \frac32$
and $\tilde V_\infty(t) = \frac{16 \pi^2 e^{- \: \frac32}}{|\Gamma|}$.
\end{enumerate}
\end{lemma}
\begin{proof}
The classification of the solitons follows from 
\cite[Corollary 51.22]{Kleiner-Lott_perelman_notes}.
Let $f$ be a potential for the soliton, i.e.
\begin{equation}
\ric + \hess(f) = - \frac{1}{2t} g
\end{equation}
and 
\begin{equation}
\frac{\partial f}{\partial t} = |\nabla f|^2.
\end{equation}
There is a constant $C$ so that $R + |\nabla f|^2 + \frac{1}{t} f = 
- \frac{C}{t}$.

From \cite[Theorem 3.7]{Enders_reduced_distance} and 
\cite[Proposition  3.8]{Naber_noncompact_shrinking},
for any sequence $t_i \rightarrow 0^-$, after passing to a subsequence
there is a limit $\lim_{i \rightarrow \infty} l_{y,t_i}$
with continuous convergence on compact subsets of $\M$.
In the case of a gradient shrinking soliton,
\cite[Chapter 7.7.3]{Chow_techniques_applications} implies that 
$\lim_{i \rightarrow \infty} l_{y,t_i} = f + C$; c.f.
\cite[Example 3.3]{Enders_reduced_distance}. Thus the limit
$\lim_{t \rightarrow 0^-} l_{y,t}$ exists and equals $f + C$, 
independent of $y$.
In our case, the formulas for $l_\infty$ and $\tilde V_\infty$ now follow from
straightforward calculation.
\end{proof}

\subsection{Estimates on $l$-functions} \label{lestimates}

We recall some estimates on the $l$-function that hold for
$\kappa$-solutions, taken from \cite{Ye_reduced}

The letter $C$ will denote a generic universal constant. 
From \cite[(2.53)]{Ye_reduced},
\begin{equation} \label{ye2.53}
R\leq \frac{Cl}{\tau}.
\end{equation}
From \cite[(2.54),(2.56)]{Ye_reduced},
\begin{equation} \label{ye2.54}
\max(|\nabla l|^2,|l_\tau|) \leq \frac{Cl}{\tau}.
\end{equation}
From \cite[(2.55)]{Ye_reduced},
\begin{equation} \label{ye2.55}
|\sqrt{l}(q_1,\tau)-\sqrt{l}(q_2,\tau)|\leq
\sqrt{\frac{C}{4\tau}}\;d(q_1,q_2,\tau).
\end{equation}
From \cite[(2.57)]{Ye_reduced},
\begin{equation} \label{ye2.57}
  (\frac{\tau_1}{\tau_2})^C\leq
\frac{l(q,\tau_2)}{l(q,\tau_1)}
\leq (\frac{\tau_2}{\tau_1})^C.
\end{equation}
From \cite[(3.7)]{Ye_reduced},
\begin{equation} \label{ye3.7}
-l(q_1,\tau)-1+C_1\frac{d^2(q_1,q_2,\tau)}{\tau}
\leq l(q_2,\tau)
\leq 2l(q_1,\tau)+C_2\frac{d^2(q_1,q_2,\tau)}{\tau}\,.
\end{equation}

\subsection{Canonical neighborhoods} \label{canon}

In this subsection we recall the notion of a canonical neighborhood for a 
Ricci flow with surgery, and define the notion of a canonical neighborhood
in a singular Ricci flow.  We mention that this rather complicated looking 
definition is motivated by the structure of   $\kappa$-solutions and the
standard (postsurgery) solution.  

Let $r : [0, \infty) \rightarrow (0, \infty)$ be a
decreasing
function.  Let $\epsilon > 0$ be small enough so that
the bulletpoints at the end of Subsection \ref{appkappa} hold
(with $\delta = \epsilon$).
Let $C_1 = C_1(\epsilon)$ and $C_2 = C_2(\epsilon)$ be the
constants in \cite[Definition 69.1]{Kleiner-Lott_perelman_notes}.

As in \cite[Definition 69.1]{Kleiner-Lott_perelman_notes},
a Ricci flow with surgery $\M$ defined on the time interval $[a,b]$ satisfies
the {\em $r$-canonical neighborhood assumption} if every
$(x,t)\in \M_t^\pm$ with scalar curvature $R(x,t)\geq r(t)^{-2}$
has a canonical neighborhood in the corresponding forward/backward 
time slice, in the following sense.   There is
an  $\hat r\in (R(x,t)^{-\frac12},C_1R(x,t)^{-\frac12})$ 
and an open 
set $U\subset \M_t^\pm$ with 
$\ol{B^\pm(x,t,\hat r)}\subset U\subset B^\pm(x,t,2\hat r)$ that
falls into one of the
following categories : 

(a) $U\times [t-\De t,t]\subset\M$  is a strong $\eps$-neck
for some $\De t>0$.  (Note that after parabolic rescaling the scalar
curvature at $(x,t)$ becomes $1$, so the scale factor must be
$\approx R(x,t)$, which implies that $\De t\approx R(x,t)^{-1}$.)

(b) $U$ is an $\eps$-cap which, after rescaling, is $\epsilon$-close to the
corresponding piece of a $\kappa_0$-solution or a time slice of a standard
solution.

(c) $U$ is a closed manifold diffeomorphic to $S^3$ or $\R P^3$.

(d) $U$ is  $\epsilon$-close to 
a closed manifold of constant positive sectional 
curvature.

\noindent
Moreover, the scalar curvature in $U$ lies between $C_2^{-1}R(x,t)$
and $C_2R(x,t)$.
In cases (a), (b), and  (c), the volume of $U$ is greater than 
$C_2^{-1}R(x,t)^{-\frac32}$. In case (c), the infimal
sectional curvature of $U$ is greater than $C_2^{-1}R(x,t)$.

Finally, we require that
\begin{equation}
\label{gradest}
|\nabla R(x, t)|<\eta R(x, t)^{\frac32}, \: \: \:  \: \: \: \: 
\left|\frac{\D R}{\D t} (x, t)\right|<\eta R(x, t)^2,
\end{equation}
where $\eta$ is the constant from 
\cite[(59.5)]{Kleiner-Lott_perelman_notes}.  Here the time 
dervative $\frac{\D R}{\D t} (x, t)$ should be interpreted
as a one-sided derivative when the point $(x,t)$ is added or removed
during surgery at time $t$.

We use a slightly simpler definition of canonical neighborhood in the 
case of singular Ricci flows, for Definition \ref{def_singular_ricci_flow}.
We do not need to consider forward/backward time slices, and in 
case (b), we do not need to consider the case that $U$ is close to 
a time slice of a standard solution. 

\begin{remark}
Alternatively, 
for a singular Ricci flow,
one could replace the above definition of canonical neighborhood with the 
requirement that every point with $R\geq r(t)^{-2}$ is
$\eps$-modelled on a
$\kappa(t)$-solution.  This is quantitatively equivalent to the 
definition above, as
follows from Proposition \ref{goodeps} and
\cite[Lemma 59.7]{Kleiner-Lott_perelman_notes}.
\end{remark}

\subsection{Ricci flow with surgery} \label{RFsurgery}

We recall that there are certain parameters 
in the definition of Ricci flow with surgery, namely a number
$\epsilon > 0$ and 
positive nonincreasing
functions $r, \kappa, \delta : [0, \infty) \ra
(0, \infty)$. The function $r$ is the canonical neighborhood scale;
c.f. Appendix \ref{canon}. The function $\kappa$ is the noncollapsing
parameter; c.f. Appendix \ref{kappa}.
The parameter $\epsilon > 0$ is a global parameter in the definition of
a Ricci 
flow with surgery \cite[Remark 58.5]{Kleiner-Lott_perelman_notes}.

The function
$\delta : [0, \infty) \rightarrow (0, \infty)$ is a surgery parameter.
There is a further parameter
$h(t) < \delta^2(t) r(t)$ so that if a point
$(x,t)$ lies in an $\epsilon$-horn and has $R(x,t) \ge h(t)^{-2}$, then
$(x,t)$ 
is the center of a $\delta(t)$-neck
\cite[Lemma 71.1]{Kleiner-Lott_perelman_notes}.
One can then perform surgery
on such cross-sectional $2$-spheres
\cite[Sections 72 and 73]{Kleiner-Lott_perelman_notes}.
Perelman showed that there are positive nonincreasing step functions
$r_P$, $\kappa_P$ and $\overline{\delta}_P$
so that if the (positive nonincreasing) function $\delta$
satisfies $\delta(t) < \overline{\delta}_P(t)$ then
there is a well-defined Ricci flow with surgery, with a discrete set
of surgery times
\cite[Sections 77-80]{Kleiner-Lott_perelman_notes}.

In particular, we can assume that $\delta$ is 
strictly decreasing.
If $r \le r_P$ and $\kappa \le \kappa_P$
are positive functions then the $r_P$-canonical neighborhood 
assumption implies
the $r$-canonical neighborhood assumption, 
and $\kappa_P$-noncollapsing implies
$\kappa$-noncollapsing.  Hence Ricci flow with surgery also exists in
terms of the parameters $(r, \kappa, \delta)$. Consequently, we
can assume that $r$, $\kappa$ and $\delta$ are
strictly decreasing.

  \begin{remark} \label{surgeryremark}
    We remark that for the purposes of this paper, it is necessary to impose slightly stricter conditions on the surgery process than those that are needed for the proof of the geometrization conjecture.  Specifically, in the definition of Ricci flow with $(r,\delta)$-cutoff \cite[Definition 73.1]{Kleiner-Lott_perelman_notes}, the surgery procedure involves choosing $\delta$-necks inside $\eps$-horns, cutting along cross-sectional $2$-spheres of the $\delta$-necks, and discarding the tips of the $\eps$-horns; see steps B-D of \cite[Definition 73.1]{Kleiner-Lott_perelman_notes}.  Here we require that the $\delta$-necks are chosen in such a way that the discarded tips have scalar curvature at least $\frac{1}{100}h^{-2}$.
  \end{remark}

As in \cite[Section 68]{Kleiner-Lott_perelman_notes},
the formal structure of a {\em Ricci flow with surgery} is given by 
\begin{itemize}
\item   A collection of Ricci flows
$\{(M_k\times[t_k^-,t_k^+),g_k(\cdot))\}_{1\leq k \leq N}$,
where $N\leq \infty$, $M_k$ is a compact (possibly empty) manifold,
$t_k^+=t_{k+1}^-$ for  all
$1\leq k<N$, and the flow $g_k$ 
goes singular at $t_k^+$ 
for each $k < N$. We allow $t_N^+$ to be $\infty$.
\item A collection of limits
$\{(\Om_k,\bar g_k)\}_{1\leq k\leq N}$, in the sense of 
\cite[Section 67]{Kleiner-Lott_perelman_notes},
at the respective final times $t_k^+$ that are
singular if $k < N$. (Here
$\Omega_k$ is an open subset of
$M_k$.)
\item A collection of isometric embeddings 
$\{\psi_k:X_k^+\ra X_{k+1}^-\}_{1\leq k<N}$
where $X_k^+\subset \Om_k$ and $X_{k+1}^-\subset M_{k+1}$,
$1 \le k< N$, are compact $3$-dimensional submanifolds with 
boundary.   The $X_k^\pm$'s are the subsets which survive the 
transition from one flow to the next, and the $\psi_k$'s give
the identifications between them.   
\end{itemize}

We will say that $t$ is a {\em singular time} if $t=t_k^+=t_{k+1}^-$
for some $1\leq k<N$, or $t=t_N^+$ and the metric goes singular
at time $t_N^+$.

A Ricci flow with surgery does not necessarily have to have any
real surgeries, i.e. it could be a smooth nonsingular flow. 

We now describe the Ricci flow spacetime associated with a Ricci flow with surgery.  We mention that in \cite{Kleiner-Lott_perelman_notes}, a spacetime object was associated with a Ricci flow with surgery; that construction was slightly different from the one used here, and did not produce a Ricci flow spacetime in the sense of Definition~\ref{def_ricci_flow_spacetime}. 
We begin with the time slab
$M_k\times [t_k^-,t_k^+]$ for $1\leq k\leq N$, which has a time function 
$\t:M_k\times [t_k^-,t_k^+]\ra [t_k^-,t_k^+]$
given by  projection onto the second factor, and a time
vector field $\D_{\t}$ inherited from the coordinate vector field
on the factor $[t_k^-,t_k^+]$.  
 
For every $1<k\leq N$, put
$W_k^-=(M_k\setminus \Int(X_k^-))\times\{t_k^-\}$
and for $1\leq k<N$, let $W_k^+=(M_k\setminus \Int(X_k^+))\times \{t_k^+\}$.
Since $W_k^\pm$ is a closed subset of the $4$-manifold with boundary
$M_k\times [t_k^-,t_k^+]$, the complement 
$Z_k=(M_k\times [t_k^-,t_k^+])\setminus
(W_k^-\cup W_k^+)$ is a $4$-manifold with boundary,  where
$\D Z_k=(M_k\times\{t_k^-,t_k^+\})\setminus (W_k^-\cup W_k^+)$.
Note that the Ricci flow $g_k(\cdot)$ with singular limit $\bar g_k$ 
defines a smooth metric $\hat g_k$ on the subbundle $\ker d\t\subset TZ_k$ that
satisfies $\L_{\D_{\t}}\hat g_k=-2\ric(\hat g_k)$.

For every $1\leq k<N$, we glue $Z_k$ to $Z_{k+1}$ using the identification
$\Int(X_k^+)
\stackrel{\psi_k}{\ra} \Int(X_{k+1}^-)$,
to obtain a smooth $4$-manifold
with boundary $\M$, where $\D \M$ is the image of $W_1^-\cup W_N^+$ under
the quotient map 
$\bigsqcup_k Z_k\ra \M$.  The time functions, time vector fields,
and metrics descend to $\M$, yielding a tuple 
$(\M,\t,\D_{\t},g)$ which is a Ricci flow spacetime in the sense of Definition
\ref{def_ricci_flow_spacetime}.

Recall the notion of a normalized Riemannian manifold from the
introduction. 
Our convention is that the trace of the curvature operator
is the scalar curvature.
From \cite[Appendix B]{Kleiner-Lott_perelman_notes},
if a smooth three-dimensional Ricci flow
$\M$ has normalized initial condition then the scalar curvature
satisfies 
\begin{equation} \label{scalarsurg}
R(x,t) \: \ge \:  - \: \frac{3}{1+2t}.
\end{equation}
If follows that the volume satisfies 
\begin{equation} \label{volumesurg}
\V(t) \le (1 + 2t)^{\frac32} \V(0).
\end{equation}
These estimates also hold for a Ricci flow with surgery.

Let $A$ be a symmetric $3 \times 3$ real matrix.  
Let $\lambda_1$ denote its smallest eigenvalue.
For $t \ge 0$, put
\begin{align} \label{hicone}
K(t) =  \{ A \: : \: & \tr(A) \ge - \frac{3}{1+t}, \mbox{ and if }
\lambda_1 \le - \frac{1}{1+t} \mbox{ then } \\ 
& \tr(A) \ge - \lambda_1 \left( \log(- \lambda_1) + \log(1+t) - 3 \right)\}.
\notag
\end{align}
Then $\{K(t\}_{t \ge 0}$ is a family of $O(3)$-invariant convex sets which
is preserved by the ODE on the space of curvature operators
\cite[Pf. of Theorem 6.44]{Chow-Lu-Ni}. 
If a smooth three-dimensional Ricci flow
has normalized initial conditions then the time-zero curvature operators
lie in $K(0)$. Using (\ref{scalarsurg}), 
we obtain the Hamilton-Ivey estimate that
whenever the lowest eigenvalue $\lambda_1(x,t)$ of the curvature operator
satisfies $\lambda_1 \le - \frac{1}{1+t}$, we have
\begin{equation} \label{hi}
R \ge - \lambda_1 \left( \log(- \lambda_1) + \log(1+t) - 3 \right).
\end{equation}
The surgery procedure is designed to ensure that (\ref{hi}) also holds
for Ricci flows with surgery.

(Perelman's definition of a normalized Riemannian manifold is
slightly different; he requires that the sectional curvatures be
bounded by one in absolute value \cite[Section 5.1]{Perelman_surgery}.
With his convention, $R(x,t) \: \ge \: - \: \frac{6}{1+4t}$ and
$\V(t) \le (1 + 4t)^{\frac32} \V(0)$.)

\newpage

\section{Extension of
Proposition 
\ref{prop_cylinder_stability_1_noncompact}}
\label{technical}

The main result in this appendix is
Proposition \ref{prop_critical_transitional},
which is  an extension of Proposition
\ref{prop_cylinder_stability_1_noncompact} to 
general $\kappa$-solutions.   It uses 
similar ideas as Proposition \ref{prop_cylinder_stability_1_noncompact} 
but is a bit more complicated
to state.  We include this  for the sake of 
completeness, as it gives a general quantitative 
picture of the behavior of $\kappa$-solitons
over time.  It is not needed in the body of the paper.

Suppose that $\M$ is a $\kappa$-solution defined on $(-\infty,0]$, and
$l:\M_{<0}\ra \R$
is the $l$-function with spacetime basepoint $(x,0)\in \M_0$.  
Roughly speaking,
Proposition \ref{prop_critical_transitional}
says that the reduced volume for a $\kappa$-solution, as a function of time,
has two types of behavior: it can remain close to the reduced volume
of a gradient shrinking soliton over a 
long time interval, or it
can transition relatively
quickly between two such values. During long time intervals
of the former kind, the $\kappa$-solution and $l$-function
are close, modulo parabolic scaling, to a
gradient shrinking soliton and its $l$-function.
We now give the details. 

Fix $\kappa > 0$.
By Lemma \ref{lem_classification_gradient_soliton_kappa_solutions}
there are only finitely many
 $3$-dimensional gradient shrinking solitons that are 
$\kappa$-solutions.
Let $\V_{sol}=\{0<\tilde V_1<\ldots <\tilde V_K\}$ be the set of their
reduced
volumes, 
in the sense of 
Lemma \ref{lem_classification_gradient_soliton_kappa_solutions}.

\begin{proposition}
\label{prop_critical_transitional}
For all $\hat\eps>0$ and $C<\infty$ there is a 
$\rho = \rho(\hat\eps, C) <\infty$ with the following
property. 
Suppose that
\begin{itemize}
\item $\M$ is  a $\kappa$-solution defined on $(-\infty,0]$,
\item $(x,0)\in \M_0$ and
\item $l:\M_{<0}\ra \R$ is the 
$l$-function on $\M$ with spacetime basepoint $(x,0)$.
\end{itemize}
Then there is a partition 
$$
\P=\{-\infty<T_0<\ldots<T_i=-1\}\,,
$$ 
of the interval $(-\infty,-1]$
(where $i\leq 2|\V_{sol}|$) such that if $I_j=[T_{j-1},T_j]$ for $j>0$,
and $I_0 = (-\infty,T_0]$, then the following holds :
\begin{enumerate}
\item If $j$ is odd then $\frac{T_{j-1}}{T_j}<\rho$.
\item If $j$ is even, 
then
	\begin{enumerate}
	\item  For every  $(y,t)\in \M_{I_j}$ with 
	$l(y,t)<C$, the rescaled triple 
	$(\hat\M(-t),(y,t),l)$ is $\hat\eps$-close to a triple
	$(\M^\infty,(y_\infty,-1),l_\infty)$ where $\M^\infty$ is a 
	gradient shrinking soliton with potential function $l_\infty$,
	and $\tilde V(t)$ is $\hat\eps$-close to the reduced volume of $\M^\infty$.
	(See Appendix \ref{appcloseness} 
	for the definition of $\hat\eps$-closeness of triples).
	\item  $|\tilde V(I_j)\cap \V_{sol}|=1$ if $j>0$ and  
	$|\tilde V(I_j)\cap \V_{sol}|=0$ if $j=0$.
	\item $\tilde V(t)$ tends to 
	an element of $\V_{sol}$ as $t\ra-\infty$.
	\end{enumerate}
\end{enumerate}
\end{proposition}

We begin with a lemma.
\begin{lemma}
\label{lem_nearly_constant_implies_soliton}
For all $\eps_1>0$ and $C<\infty$ there exist  $\mu=\mu(\eps_1,C)<\infty$ and 
$\th(\eps_1,C)>0$ with the following
property.

Let
 $\M$ be  a $\kappa$-solution defined on $(-\infty,0]$, let
 $(x,0)\in \M_0$ be a basepoint and let 
 $l:\M_{<0}\ra \R$ be the 
corresponding $l$-function.
Suppose that
\begin{itemize}
\item  $T'<T<0$,
\item  $\frac{T'}{T}\geq\mu$,
\item  $\tilde V(T)\leq\tilde V(T')+\th$ and
\item  $(y,T)\in \M_T$ is a point where $l(y,T)<C$.
\end{itemize}
Then the rescaled
triple $(\hat\M(-T),(y,-1),l)$ is 
$\eps_1$-close to a triple
	$(\M^\infty,(y_\infty,-1),l_\infty)$ where $\M^\infty$ is a 
	gradient shrinking soliton with potential function $l_\infty$,
	and $\tilde V(T)$ is $\eps_1$-close to the reduced volume of $\M^\infty$.
\end{lemma}
\begin{proof}
Suppose the lemma were false.  Then for some $\eps_1>0$ and $C<\infty$, 
there would
be sequences $\{\M^j\}_{j=1}^\infty$, $\{T_j'\}_{j=1}^\infty$,
$\{T_j\}_{j=1}^\infty$ and $\{y_j\}_{j=1}^\infty$ with
$T_j'<T_j<0$,
$(y_j,T_j)\in \M^j$,
$\lim_{j \rightarrow \infty} \frac{T_j'}{T_j} = \infty$ and
$\lim_{j \rightarrow \infty}
|\tilde V(T_j) - \tilde V(T_j')| = 0$,
but for each $j$,
the rescaled triple $(\hat M^j(-T_j),(y_j,-1),l_j)$ does not satisfy
the conclusion of the lemma.  
Proposition \ref{prop_convergence_no_curvature_bound} implies that
after passing to a subsequence, we have convergence of triples to
a $\kappa$-solution $\left( \M^\infty, (y_\infty, -1), l_\infty \right)$.
From the assumed properties of $\{\M^j\}_{j=1}^\infty$, the reduced volume
$\tilde V_\infty(t)$ of $\M^\infty$ will be constant for $t \le -1$.
Hence $\M^\infty$ is a gradient shrinking soliton, which gives a
contradiction.
\end{proof}

\bigskip

\begin{proof}[Proof of Proposition \ref{prop_critical_transitional}]
From the monotonicity of the reduced volume $\tilde V(t)$, the existence of the
asymptotic soliton and the fact that $\lim_{t \rightarrow 0}
\tilde V(t) = \tilde V_{\R^3}$, the
reduced volume  of flat $\R^3$ (as a gradient shrinking soliton),
it follows that for all $t\in (-\infty,0)$, we have
$\tilde V(t)\in [\tilde V_1,\tilde V_{\R^3}]$. 
 
Let $\mu_1=\mu(\hat\eps,C)$ and $\th_1=\th(\hat\eps,C)$ be as in Lemma \ref{lem_nearly_constant_implies_soliton}.
Choose $\th_1'<\min(\frac{\th_1}{2},\tilde V_1)$ small enough so that the
intervals $\{[\tilde V_j-\th_1',\tilde V_j+\th_1']\}_{1\leq j\leq K}$
are disjoint. Put
 $\tilde J_j=[\tilde V_j+\th_1',\tilde V_{j+1}-\th_1']$
for  $1\leq j<K$, and  $\tilde J_K=[\tilde V_K+\th_1',\tilde V_{\R^3}]$.

Putting $J_j=\tilde V^{-1}(\tilde J_j)\cap (-\infty,-1]$ for $1\leq j\leq K$, 
we obtain a 
collection $\{J_j\}_{1\leq j\leq K}$ of at most $K$ intervals in 
$(-\infty,-1]$.  

We know that for each $t < 0$, there is a point
$y_t \in \M_t$ with $l(y_t, t) \le \frac32$.
Let $\mu_2=\mu(\th_1', 2)$ and $\th_2=\th(\th_1', 2)$ be as in 
Lemma \ref{lem_nearly_constant_implies_soliton}.
Note that for every $1\leq j\leq K$, since 
$\dist(\tilde V(t),\V_{sol})\geq \th_1'$ for all $t\in J_j$, 
 the contrapositive of Lemma 
\ref{lem_nearly_constant_implies_soliton} implies that for any
$T',T\in J_j$ with $\tilde V(T)\leq \tilde V(T')+\th_2$,
we have $\frac{T'}{T}<\mu_2$.  Iterating this  $1+[\th_2^{-1}\length(J_j)]$
times gives a bound
$$
\frac{\inf J_j}{\sup J_j}<\mu_3=\mu_3(\hat\eps,C)\,.
$$ 

The complement of $\bigcup_j J_j$ in $(-\infty,-1]$ is a union of open intervals.
Let $I=(t_-,t_+)$ be one such interval.  Then $\tilde V\restr_I$ takes values
in an interval of length at most $2 \th_1^\prime \leq \th_1$.  Therefore  by Lemma 
\ref{lem_nearly_constant_implies_soliton}, if $t\in I$ and $t\geq \frac{t_-}{\mu_1}$,
then for all $y\in \M_t$ with $l(y,t)<C$, conclusion (2a) of Proposition 
\ref{prop_critical_transitional} holds.

Thus we may obtain the desired partition of $(-\infty,-1]$ by enlarging
each nonempty interval $J_j=[t_{j-1},t_j]$ to $J_j'=[t_{j-1},\frac{t_j}{\mu_1}]$,
forming  the union $\bigcup_j J'_j$, and taking the associated partition of 
$(-\infty,-1]$. Conclusion (2b) of the proposition follows from the
construction.  Conclusion (2c) of the proposition
follows from the finiteness of ${\mathcal V}_{sol}$, the monotonicity of
the reduced volume and the
existence of the asymptotic soliton.
\end{proof}

\begin{corollary} \label{gradcor}
With the notation of Proposition \ref{prop_critical_transitional}, 
depending  on the topology
of $\M_0$, the sequence 
$\{\M^{2j}\}_{0\leq 2j\leq i}$ of gradient shrinking solitons 
must be one of the following:
\begin{enumerate}
\item $\M_0$ is diffeomorphic to a spherical space form other 
than $S^3$ or $S^3/\Z_2$:
  $\{\sphere/\Ga\}$.
\item $\M_0$ is diffeomorphic to the cylinder: $\{Cyl\}$.
\item $\M_0$ is diffeomorphic to the $\Z_2$-quotient of the cylinder:
$\{Cyl/\Z_2,Cyl\}$, $\{Cyl/\Z_2\}$, $\{Cyl\}$.
\item $\M_0$ is diffeomorphic to $S^3$: $\{Cyl,\sphere\}$, $\{\sphere\}$, 
$\{Cyl\}$.
\item $\M_0$ is diffeomorphic to $S^3/\Z_2$: 
 $\{Cyl/\Z_2,\sphere/\Z_2\}$,
$\{\sphere/\Z_2\}$, $\{Cyl/\Z_2\}$, $\{Cyl/\Z_2,Cyl\}$, $\{Cyl\}$.
\end{enumerate}
\end{corollary}
\begin{proof}
The relevant gradient shrinking solitons are $\cyl/\Z_2$, $Sphere/\Z_2$,
$\cyl$ and $\sphere$.  From Lemma
\ref{lem_classification_gradient_soliton_kappa_solutions},
their reduced volumes are
$16.39$, $17.62$, $32.78$ and $35.24$, respectively.
The corollary follows from
the fact that the reduced volumes of the $\M^{2j}$'s 
form a strictly increasing sequence, the local stability of quotients
of $\sphere$, and topological restrictions.
\end{proof}

\begin{remark}
From explicit calculations, one can rule out the $\{\cyl\}$ possibility
in case (3) of Corollary \ref{gradcor}.  We expect that one can also
rule out the $\{\cyl\}$ possibility in case (5).
\end{remark}

\section{Nonnegative isotropic curvature in four dimensions} \label{pic} 

In this appendix we extend the results of the paper 
to four-dimensional
Ricci flow with nonnegative isotropic curvature, under an assumption of
no incompressible embedded spherical space forms. The basic results
about such flows are due to Hamilton \cite{Hamilton_four_manifolds}.

Let $(M,g)$ be a closed oriented $4$-manifold with nonnegative isotropic 
curvature. Suppose that under Ricci flow, the isotropic curvature does
not immediately become positive. From 
\cite[Theorem 4.10]{Micallef-Wang_nonnegative_isotropic}, 
one of the following happens:
\begin{enumerate}
\item $(M,g)$ is flat.
\item The universal cover $(\widetilde{M}, \widetilde{g})$
is an isometric product of metrics on 
$S^2$ and a surface $\Sigma$.
\item $M$ is biholomorphic to $\mathbb{C} P^2$ and $g$ is a K\"ahler metric
with positive first Chern class.
\end{enumerate}
In these three cases, the Ricci flow is well understood.  In case
(1), it is just a static flow.  In cases (2) and (3), the ensuing
Ricci flow will be smooth on a time interval $[0, T)$ and
nonexistent thereafter. Hence we
assume that $(M,g)$ has positive isotropic curvature.  This condition
is preserved under the Ricci flow. 

There is a version of the Hamilton-Ivey inequality for Ricci flow
on a $4$-manifold with positive isotropic curvature
\cite[Section B]{Hamilton_four_manifolds}. It has the implication that any
blowup limit is a $\kappa$-solution with positive isotropic curvature, and
restricted isotropic curvature
in the sense of \cite[(2.4)]{Chen-Zhu_positive_isotropic}. Any such compact solution
is diffeomorphic to $S^4$.  Any such noncompact solution is
diffeomorphic to $\R^4$, or is an isometric product $\R \times Z$ 
where $Z$ is diffeomorphic to a spherical space form.

Using \cite{Naber_noncompact_shrinking}, we have the following extension of Lemma 
\ref{lem_classification_gradient_soliton_kappa_solutions}.

\begin{lemma}
\label{4Dlemma}
Let $\M$ be a four-dimensional gradient shrinking
soliton that is a $\kappa$-solution, has positive isotropic curvature and
restricted isotropic curvature in the terminology of
\cite[(2.4)]{Chen-Zhu_positive_isotropic}, and  blows up as $t\ra 0$.
For $t < 0$ and a point $(y,t) \in \M$, 
let $l_{y,t} \in C^\infty(\M_{< t})$ be the 
$l$-function on $\M$
constructed using ${\mathcal L}$-geodesics going backward in time from 
$(y,t)$.
Then there is a function $l_\infty \in C^\infty(\M)$ so that
the limit $\lim_{t \rightarrow 0^-} l_{y,t} = l_\infty$ exists,
independent of $y$, with continuous convergence on compact subsets on $\M$.
Define
$\tilde V_\infty:(-\infty,0)\ra (0,\infty)$ as in
(\ref{redvol}).
Then one of the following holds:
\begin{enumerate}
\item $\M$ is isometric to the shrinking round cylinder solution on
$(S^3/\Gamma) \times \R$, where $\Gamma \subset \SO(4)$,
with $R(x,t)= \frac32 (-t)^{-1}$, 
$l_\infty((x,z),t)=\frac32 +\frac{z^2}{-4t}$ and
$\tilde V_\infty(t) = \frac{32 \pi^{\frac52} e^{- \: \frac32}}{|\Gamma|}$.
\item $\M$ is a shrinking round $S^4$, with
$R(x,t)= 2 (-t)^{-1}$,  $l_\infty(x,t)= 2$
and $\tilde V_\infty(t) = \frac{96 \pi^2}{e^2}$.
\end{enumerate}
\end{lemma}

Let $(M, g)$ be a compact $4$-manifold with positive isotropic curvature
and no embedded $\pi_1$-injective spherical space forms.  The
construction of Ricci flow with surgery, starting from $(M, g)$ was
initiated by Hamilton \cite{Hamilton_four_manifolds}. There was a problem with
\cite{Hamilton_four_manifolds} because no canonical neighborhood result was
available at the time.  The construction was revisited by Chen-Zhu
\cite{Chen-Zhu_positive_isotropic}, based on Perelman's work in the three-dimensional
case \cite{Perelman_entropy,Perelman_surgery} and a preprint version of
\cite{Kleiner-Lott_perelman_notes}.
The result of Lemma \ref{4Dlemma}, which follows from
\cite{Naber_noncompact_shrinking}, was not available when
\cite{Chen-Zhu_positive_isotropic} was written. With the incorporation 
of Lemma \ref{4Dlemma},
the construction of the Ricci flow with surgery
is strictly analogous to Perelman's work.

The results of the present paper now extend to the
setting of four-dimensional singular Ricci flows with positive
isotropic curvature, again under the assumption that the initial
time slice does not have any embedded $\pi_1$-injective spherical
space forms.  

Returning to Ricci flow with surgery, suppose that there 
are such embedded spherical space forms in the initial time slice.
There is still a well-defined
Ricci flow with surgery, but the time slices may be orbifolds with 
isolated singularities  \cite{Chen-Tang-Zhu_positive_isotropic}. It should be
possible to extend the results of the present paper to the orbifold
setting, using \cite{Kleiner-Lott_orbifollds}, but we do not address 
the subject here.

\bibliography{ricci_limits}
\bibliographystyle{alpha}

\end{document}